\documentclass[a4paper,11pt,oneside,reqno]{amsart}
\usepackage{graphicx,fancyhdr,array,multicol,multirow,enumitem}
\usepackage{tikz}
\usepackage[hmargin={2.5cm,2.5cm},vmargin={2cm,2cm},includefoot, marginparwidth=60pt]{geometry}
\usepackage{fancybox}
\usepackage{xcolor}
\usepackage{cancel}
\usepackage[normalem]{ulem}
\usepackage{lmodern}
\usepackage{soul} 
\usepackage[utf8]{inputenc}
\usepackage[T1]{fontenc}
\usepackage[english]{babel}
\usepackage{amssymb,mathtools,amsfonts,amsthm}
\usepackage{stmaryrd}
\usepackage{mathrsfs}
\usepackage{setspace}
\usepackage{hyperref}
\usepackage{cite} 
\usepackage{dsfont}
\newtheorem{theorem}{Theorem}[section]
\newtheorem{definition}[theorem]{Definition}
\newtheorem{lemma}[theorem]{Lemma}

\newtheorem{proposition}[theorem]{Proposition}

\theoremstyle{definition}
\newtheorem{remark}[theorem]{Remark}
\newcommand{\mc}[1]{{\mathcal #1}}

\newcommand{\eps}{\varepsilon}

\newcommand{\E}{\mathbb E}
\newcommand{\N}{\mathbb N}
\renewcommand{\P}{\mathbb P}
\newcommand{\R}{\mathbb R}
\newcommand{\bQ}{\mathbb{Q}}

\newcommand{\Z}{\mathbb Z}
\newcommand{\I}{\ensuremath{\mathbf{i}}}
\newcommand{\rZ}{\mathcal{Z}^\eps}

\newcommand{\scZ}{\mathscr{Z}^\eps}
\newcommand{\p}{\mathsf{p}}

\newcommand{\mm}[1]{#1}
\newcommand{\oo}[1]{#1}
\renewcommand{\leq}{\leqslant}
\renewcommand{\geq}{\geqslant}
\renewcommand{\bar}{\overline}

\newcommand{\cL}{\mathcal{L}}

\newcommand{\rightarrowplus}{\overset{+}{\longrightarrow}}
\newcommand{\leftarrowplus}{\overset{+}{\longleftarrow}}

\newcommand{\leftarrowminus}{\overset{-}{\longleftarrow}}

\newcommand{\Res}[1]{\underset{{#1}}{\mathbf{Res}}}

\begin{document}
	
	\title{Weakly asymmetric facilitated exclusion process}
	
	\author[G. Barraquand]{Guillaume Barraquand}
	\address{G. Barraquand,
		Laboratoire de physique de l’\'ecole normale supérieure, ENS, Universit\'e PSL, CNRS, Sorbonne Universit\'e, Universit\'e Paris-Cité, Paris, France}
	\email{guillaume.barraquand@ens.fr}
	
	\author[O. Blondel]{Oriane Blondel}
	\address{O. Blondel, Universite Claude Bernard Lyon 1, CNRS, Centrale Lyon, INSA Lyon, Université Jean Monnet, ICJ UMR5208,
		69622 Villeurbanne, France.}
	\email{blondel@math.univ-lyon1.fr} 
	
	\author[M. Simon]{Marielle Simon}
	\address{M. Simon, Universite Claude Bernard Lyon 1, CNRS, Centrale Lyon, INSA Lyon, Université Jean Monnet, ICJ UMR5208,
		69622 Villeurbanne, France.}
	\email{msimon@math.univ-lyon1.fr} 
	
	\begin{abstract}
		We consider the facilitated exclusion process, an interacting particle system on the integer line where particles hop to one of their left or right neighbouring site only when the other neighbouring site is occupied by a particle. A peculiarity of this system is that, starting from the step initial condition, the density profile develops a downward jump discontinuity around the position of the first particle, unlike other exclusion processes such as the asymmetric simple exclusion process (ASEP). In the weakly asymmetric regime, we show that the field of particle positions around the jump discontinuity converges to the solution of the multiplicative noise stochastic heat equation (\textit{i.e.}~the exponential of a solution to the KPZ equation) on a half-line subject to  Dirichlet boundary condition, with initial condition given by the derivative of a Dirac delta function. We prove this result by reformulating the problem in terms of ASEP on a half-line with a boundary reservoir, for which we extend known proofs of convergence to deal with Dirichlet boundary condition and the very singular type of initial condition that arises in our case.
	\end{abstract}

	\maketitle

\section{Introduction}
\subsection{Facilitated exclusion process}

The facilitated exclusion process (FEP) was introduced in the physics literature \cite{pastor2000FEP} as a representative of a universality class for absorbing phase transitions. It is an interacting particle system on a lattice in which particles can jump to empty neighbors provided there is a particle in their neighborhood. So far our mathematical understanding of this model is restricted to dimension $1$, where it has been studied under different lights. Its main feature is the absorbing transition mentioned above, at the critical particle density $1/2$: for particle densities below $1/2$ (subcritical regime), the system fixates on a configuration with isolated particles which cannot move, while for densities above $1/2$ (supercritical regime) it remains active forever and holes become eventually isolated. 

The totally asymmetric version of this process (where particles only jump to the right) has been studied in \cite{basu2009FTASEP} (approach to the phase transition) and \cite{chen2018limiting,goldstein2019FTASEP,goldstein2021FTASEP} (identification of the stationary states).  
Starting from a step initial condition, contrary to the well-studied totally asymmetric simple exclusion process (TASEP), a downstep leads to a rarefaction fan with a discontinuity \cite{gabel2010FEP}. For the same initial condition, at large time $t$,  particle positions fluctuate on the $t^{1/3}$ scale with Tracy-Widom GUE statistics, while the fluctuations of the rightmost particle, \textit{i.e.}~at the discontinuity, have Tracy-Widom GSE statistics \cite{baik2018FTASEP}.
The symmetric FEP has been studied as well, on the periodic lattice. It was found \cite{blondel2020FEP,blondel2021FEP} that in the diffusive space-time scaling, and under the hydrodynamic limit,  the macroscopic density $\rho$ evolves according to  a Stefan problem written as $\partial_t\rho=\partial_{xx}\big(\frac{2\rho-1}{\rho}\mathbf{1}_{\rho>1/2}\big)$, with  the space variable $x$ belonging to the one-dimensional torus of size $1$. In other words, starting the microscopic dynamics from a density profile with both supercritical and subcritical regions,  the diffusive supercritical phase progressively invades the subcritical phase via moving interfaces, until one or the other phase disappears. For the partially  asymmetric version, where particles jump to the right at rate $p<1$ and to the left at rate $q<p$, invariant measures have been characterized   on the torus \cite{gabel2010FEP} and  on the line \cite{ayyer2020stationary}. Recently, \cite{ESZ} showed that the hydrodynamic limit (in the hyperbolic scaling) is given by the unique entropy solution of $\partial_t\rho+(2p-1)\partial_x\big(\frac{(1-\rho)(2\rho-1)}{\rho}\mathbf{1}_{\rho>1/2}\big)$, with $(t,x)\in\R_+\times \R$. Finally, in \cite{EZ23}, the authors start the facilitated exclusion process from a stationary measure, and prove central limit results for the density fluctuation field in the symmetric, weakly asymmetric, and totally asymmetric cases.

\medskip

One natural question concerns the fluctuations of the FEP in the asymmetric case (FASEP), \mm{starting from the step initial condition}. As it has been noted for the totally asymmetric case in \cite{baik2018FTASEP}, the problem can be reformulated in terms of the asymmetric simple exclusion process (ASEP) on the half-line with jump rates $p>q$, and with a specific boundary condition, where particles enter the system at a rate $p$ and can never exit (see Figure~\ref{fig:FASEP-OpenASEP}). For fixed $p$ and $q$,  we expect that, up to scaling constants, the fluctuations should be similar as for the totally asymmetric FEP\footnote{Proving this requires an analogous statement for current fluctuations of ASEP on a half-line, which was not available when the first version of this article was posted. A limit theorem for the current at the boundary  in half-line ASEP  is now proved in \cite{he2024boundary}, generalizing \cite{barraquand2018stochastic, tracy2013bose, tracy2013asymmetric}.  This translates into fluctuations of the first particle's position in the FEP. The analysis of fluctuations far away from the first particle remains an open problem.}.

In the present paper, we study the fluctuations of particle positions in the FASEP in a weakly asymmetric asymptotic regime. In the bulk (that is, far from the jump discontinuity of the density profile), we do not expect that the facilitation rule will have any effect on the scaling limit of fluctuations.  In particular, we expect that the field of particle positions, appropriately rescaled, should converge to a solution of the Kardar-Parisi-Zhang (KPZ) equation on the real line.
However, for the step initial condition, if we consider the field of particle positions in the FASEP around the first particle, the situation is more interesting and the facilitation rule plays a role. The limit should be described by a stochastic PDE on a semi-infinite interval with a specific boundary condition. The main goal of the present paper is to describe this stochastic PDE \oo{and prove the convergence}.

\subsection{KPZ equation and Hopf-Cole transform} \label{sec:kpz}
In order to state our main result, let us first recall how to solve the KPZ equation on the full one-dimensional line, which reads as 
\begin{equation}\label{e:KPZ}
	\partial_t h=\frac{1}{2}\partial_{xx} h+\frac{1}{2}(\partial_x h)^2+\xi, \qquad t\ge 0, x \in \R,
\end{equation}
with $\xi$ the standard space-time white noise on $\R_+\times\R$. One usually considers the Hopf-Cole transform of a putative solution $h$, namely $\mathcal Z(t,x)=e^{h(t,x)}$. If we apply the chain rule in \eqref{e:KPZ}, ignoring all issues of regularity, the function  $\mathcal Z$ solves the Stochastic Heat Equation (SHE) with multiplicative noise 
\begin{equation}\partial_t\mc Z=\frac{1}{2}\partial_{xx}\mc Z+\mc Z\xi.
	\label{eq:SHE}
\end{equation}
The latter equation can be solved through standard SPDE techniques, and whenever it can be shown that $\mc Z>0$ \cite{mueller1991support}, this procedure yields a Hopf-Cole solution to \eqref{e:KPZ}. In their seminal paper \cite{bertini1997stochastic}, Bertini and Giacomin noticed that the discrete Hopf-Cole transform (introduced by G\"artner \cite{gartner1987transform}), defined on the microscopic space variable $k\in \Z$ by
\begin{equation}
	Z_t(k):=e^{-\lambda h_t(k)+\nu t}, \qquad t\ge 0,
	\label{eq:discreteHopfCole}
\end{equation}
of the ASEP height function $h_t(k)$ ($k\in\Z$), satisfies, for well-chosen parameters $\lambda, \nu$, a martingale problem that is a discrete analogue of the martingale problem satisfied by $\mc Z$. In the weakly asymmetric regime, say $p=\frac12e^\varepsilon, q=\frac12e^{-\varepsilon}$ with $0<\varepsilon\ll 1$, as assumed in this paper, it can then be showed that solutions of the discrete martingale problem converge to solutions of the continuous one. Given uniqueness of the solution to the continuous martingale problem, this is enough to conclude that ASEP height function, suitably rescaled, converges to a solution of the KPZ equation. Further, \cite{dembo2016weakly} identifies a whole class of models to which this method may apply, in the sense that a generalization of the discrete Hopf-Cole transform can be found and convergence to the SHE can be proved. 

Other approaches to solving \eqref{e:KPZ} that do not require a detour through the SHE were also looked for in the last decades, and can be useful in cases where there is no applicable discrete Hopf-Cole transform (unlike the present paper). Let us mention regularity structures  \cite{hairer2013kpz,hairer2014regularity}; energy solutions  \cite{goncalves2014BG,gubinelli2018uniqueness}, which have been applied \textit{e.g.~}in \cite{blondel2016,GJSi2017,GPS2020,
	yang2021stochastic};
and paracontrolled distributions \cite{gubinelli2015paracontrolled,gubinelli2017kpz}. 

\subsection{KPZ equation on the positive half-line} 
A half-space analogue of Bertini-Giacomin's result \cite{bertini1997stochastic} was proved in \cite{corwin2016open}. More  precisely,  under some condition on injection and ejection rates at the origin of ASEP on the half-line, and assuming that the effective density at the origin imposed by those rates scales as $\rho\approx\frac{1}{2}(1+(A+\frac{1}{2})\eps)$, \cite{corwin2016open} shows that the  height function of ASEP on the half-line, properly rescaled, converges to the KPZ equation on $\mathbb R_{+}$ with Neumann type boundary condition $\partial_x h(t,0)=A$. As in \cite{bertini1997stochastic}, this result is proved via the discrete Hopf-Cole transform  \eqref{eq:discreteHopfCole} of ASEP height function $h_t(k)$ ($k\in\N$), which is shown to converge to the SHE on $\R_+$ with Robin type boundary condition 
\begin{equation}\begin{cases} \partial_t\mc Z=\frac{1}{2}\partial_{xx}\mc Z+\mc Z\xi, \\
		\partial_x \mc Z(t,0) = A \mc Z(t,0).\end{cases} 
	\label{eq:SHEhalf}
\end{equation}
Since $\mc Z(t, \cdot)$ is not differentiable, the boundary condition should rather be imposed on the half-space heat kernel which is used to define the solution, we refer to \cite{corwin2016open} for details. The result of \cite{corwin2016open} was restricted to $A\geq 0$ and near-equilibrium initial conditions (see Definition~\ref{de:near-eq-cont} below). It was then  extended to all $A\in \mathbb R$ and the empty initial condition in \cite{parekh2017kpz}.  Let us note that on the full-line, the extension of the convergence result of \cite{bertini1997stochastic} to step initial condition was first discussed in \cite{amir2011probability}, with considerably less details than in \cite{parekh2017kpz}.  Some of these results were further extended in \cite{yang2020kardar} to generalizations of ASEP, in the spirit of \cite{dembo2016weakly}.  An alternative way to make sense of the KPZ equation on a half-line via regularity structures was also considered in \cite{gerencser2017singular}.

In the physics literature, the solution  $\mc Z$ to the SHE is understood as the partition function for a continuous Brownian directed polymer in a white noise potential  $\xi$. The boundary parameter $A$ can then be understood as controlling an extra energy collected by polymer paths, given by reflected Brownian motions, along the boundary \cite{borodin2016directed}. Alternatively, we may assume that there is no extra potential on the boundary, but the Brownian paths in the polymer partition function have elastic reflection on the boundary controlled by the parameter $A$. Discrete directed polymers are another family of models, with exclusion processes, which converge to the KPZ equation in full-space  \cite{alberts2014intermediate} or half-space \cite{wu2018intermediate, parekh2019positive, barraquand2022stationary}.

\begin{remark}
	The KPZ equation (and the associated SHE via Hopf-Cole transform) has also been considered on an interval with Neumann type boundary conditions. The discrete Hopf-Cole transform of open ASEP height function satisfies Robin type boundary conditions \cite{gonccalves2017nonequilibrium} and converges to the KPZ equation on an interval \cite{corwin2016open, parekh2017kpz, GPS2020, yang2020kardar}.
\end{remark}

\subsection{Main results}

As previously mentioned, the problem of fluctuations of the FASEP \oo{started from a step initial condition} reduces to the fluctuations of ASEP on the half-line with, at the origin, injections of particles at the right-jump rate $p=\frac12e^{\eps}$ and no ejection of particles. This translates into a microscopic boundary condition for the discrete Hopf-Cole transform \eqref{eq:discreteHopfCole} given by 
$$ Z_t(-1)=\mu Z_t(0),\qquad \text{with }\mu\approx  1-\eps,$$ 
as stated more precisely in \eqref{eq:discDelta} below. In contrast, in the setting of \cite{corwin2016open, parekh2017kpz}, the boundary parameter $\mu$ is scaled as $\mu\approx 1-\eps^2 A$. 
We will prove below that in our case, the appropriate scaling of the solution is different from \cite{corwin2016open, parekh2017kpz}: we will set 
$\mathcal Z_{t}^{\eps}(x) = \varepsilon^{-2}Z_{\varepsilon^{-4}t}(\varepsilon^{-2}x)$ for any macroscopic point $x$ and macroscopic time $t$ (with linear interpolation), and prove that starting from the step initial condition for the FASEP (\textit{i.e.}~empty initial condition for the half-line ASEP), $\mathcal Z_{t}^{\eps}(x)$ weakly converges as a continuous process (see Theorem \ref{thm:conv} for a precise statement) to the solution of the SHE 
\begin{equation}
	\begin{cases} \partial_t\mc Z=\frac{1}{2}\partial_{xx}\mc Z+\mc Z\xi, \\
		\mc Z(t,0) = 0,  \\
		\mc Z(0,x) = -2 \delta_0'(x),\end{cases} 
	\label{eq:SHEours}
\end{equation} where $\delta_0'$ is the derivative of the delta Dirac distribution at $0$.
We also prove a similar statement for near-equilibrium initial conditions (Theorem \ref{t:cvneareq}), under the same scaling as in  \cite{corwin2016open}. 
The SHE with Dirichlet boundary condition ($\mc Z(t,0) = 0$) was already considered in \cite{parekh2019positive} in the context of directed polymer models, but only for another type of initial condition. Here the initial condition $\delta_0'$ that we consider is very singular. We provide a more precise definition of this stochastic PDE in Definition \ref{d:she} and prove existence and uniqueness of the solution in Proposition \ref{prop:existenceuniqueness} below. Dealing with this very singular initial condition is one of the main novelties of this paper. 
Eventually, our main result is therefore the following: the macroscopic fluctuations of the field of the first particles' positions in the weakly asymmetric FASEP \oo{or ASEP} are given (\textit{via} the Hopf-Cole transform) by the solution to the KPZ equation on the half-line $\mathbb R_+$, with initial condition being the derivative of a Dirac distribution, and with Dirichlet boundary condition at the origin.

\subsection{Comparison with previous literature on half-space KPZ equation}
\label{s:comp}

In a sense, the microscopic boundary condition $Z_t(-1)=\mu Z_t(0)$ with $\mu\approx 1-\eps$ that we are considering is the $A\rightarrow\infty$ limit of the setting in \cite{corwin2016open, parekh2017kpz} which considered $\mu\approx 1-A\eps^2$. \oo{It poses no additional difficulty in terms of handling the boundary condition in heat kernel estimates, except that we need to control precisely the killing probability of a random walk following this heat kernel. This novel estimate is presented in Lemma~\ref{lem:killing}. Moreover}, as already said, our initial condition is much more singular than the ones in \cite{corwin2016open, parekh2017kpz, parekh2019positive}. In particular, it turned out to be surprisingly challenging, first to prove tightness, then 
to identify rigorously the initial condition.

To resolve these difficulties, one important tool that we use is an explicit moment integral formula for half-line ASEP coming from  \cite{barraquand2022markov}, using a Markov duality. The latter has already appeared in proofs of convergence to the KPZ equation. For ASEP, Gartner's  microscopic Hopf-Cole transform \cite{gartner1987transform} may be seen as a very special case of a more general Markov duality \cite{schuetz1997duality} between ASEP, viewed as a Markov process on $\lbrace 0,1\rbrace^{\mathbb Z}$, and another ASEP on $\mathbb Z$ with a fixed number $m$ of particles. Gartner's result essentially corresponds to the case when the dual process has only $m=1$ particle (see \cite[Section 3]{schuetz1997duality}).  In \cite{corwin2020stochastic}, the authors use Markov duality for the stochastic six vertex model. The $m=1$ particle duality is used to obtain a discrete Hopf-Cole transform which satisfies a discrete heat equation, as usual, but authors go further: they  also use  the $m=2$ particle duality in order to obtain second moment estimates useful in the proof of tightness. In the present work, the Hopf-Cole transform is the same as usual, and we prove tightness using a standard method going back to  \cite{bertini1997stochastic}. Our use of duality however comes when proving the convergence for the delta prime initial condition. We use a formula for the $m$-th moment of $\mathcal Z_{t}^{\eps}(x)$, coming from the $m$--particle duality satisfied by half-line ASEP \cite{barraquand2022markov}. As hinted above, it enters twice in the proof. First, in the proof of tightness, to obtain the moment bounds which are crucial to prove that the solution is near-equilibrium after running the process for any positive time $\delta$ {(Sections~\ref{subs:momentsZ} and \ref{subs:neareq})}. Second, we use the formula from \cite{barraquand2022markov} for the second moment to obtain a very precise control on the  second moment of $\mathcal Z_{t}^{\eps}(x)$ in terms of the continuous Dirichlet heat kernel {(Lemma~\ref{lem:limit2})}. This is crucial in the rigorous identification of the initial condition {(Section \ref{sec:conclusion})}.                                   

The continuous directed polymer model corresponding to the solution of \eqref{eq:SHEours} was considered in the physics paper \cite{gueudre2012directed}. The polymer paths are conditioned not to hit the boundary, and \cite{gueudre2012directed} studied the distribution of the partition function of polymers starting and ending at a location $\eta$, after letting $\eta \to 0 $ and appropriately rescaling the partition function by $\eta^2$. Restricting on test functions $f:\mathbb R_+\to \R$ such that $f(0)=0$, the distribution $\frac{1}{\eta}\delta_0(\cdot -\eta)$, where $\delta_0$ is the delta Dirac distribution, converges to $-\delta_0'$. This explains why the initial condition that we consider in the present paper is the physically natural one for the SHE with Dirichlet boundary condition, though the solution was never mathematically constructed before. 

Finally, let us mention that the KPZ equation on a half-line with Dirichlet boundary condition is also considered in \cite{gerencser2017singular}, but there the Dirichlet  boundary condition $h(t,0)=0$ is imposed on the KPZ equation itself and not on the SHE, so that this corresponds to a completely different stochastic PDE than the one we consider in the present paper.

\subsection{Outline of the paper}
In Section~\ref{s:coupling}, we define the FASEP and \mm{half-line} ASEP and construct a mapping that connects the two. Section~\ref{s:results} is devoted to the presentation of the Hopf-Cole transform, in terms of which we state our main results on the fluctuations of the particle positions in a weakly asymmetric regime. We then provide several preliminary results: first, in Section \ref{sec:proofshe}, the existence and uniqueness of the macroscopic SHE with $\delta_0'$ initial condition and Dirichlet boundary condition ; and second, in Section \ref{sec:heatkernel}, some explicit estimates on the discrete heat kernel with diverging Robin boundary condition \mm{together with \oo{a new killing estimate}}. The main convergence results cover two types of initial conditions which require different scalings: the \mm{empty} initial condition (Theorem~\ref{thm:conv})  and near-equilibrium initial conditions (Theorem~\ref{t:cvneareq}), which both require new arguments. We actually need the understanding of the latter as an intermediate step towards Theorem~\ref{thm:conv}, so we start by proving Theorem~\ref{t:cvneareq} in Section~\ref{sec:neareq}. Our main result, Theorem~\ref{thm:conv} is finally proved in Section~\ref{s:deltaprime}. 

We also provide in Appendix~\ref{app:Holderlimit} a proof of a H\"older estimate \mm{for the rescaled microscopic Hopf-Cole transform}, in the case of \mm{empty} initial condition, in the $\eps\rightarrow 0$ limit. This in itself is not sufficient to show the near-equilibrium property that we need to prove Theorem~\ref{thm:conv}, but we include it because it gives a better understanding of where the regularity comes from.

\section{Microscopic models and mapping}\label{s:coupling}

In the following, $\N$ denotes the set of non-negative integers, $\N^*$ the set of positive integers. We use $k,\ell, j$ to denote the microscopic (discrete) space variables, and $x,y, z$ to denote the macroscopic ones.

\medskip

The \emph{facilitated asymmetric exclusion process} in dimension one (FASEP) is a Markov process on the state space $\Omega:=\{0,1\}^\Z$ which is denoted by $\{\eta_t(k)\; ; \; k \in\Z\}_{t\geqslant 0}$. Each component $\eta_t(k)\in\{0,1\}$ is the occupation variable of the configuration of particles at site $k\in\Z$.

Let $p,q\in(0,1)$ be two asymmetry parameters. The time evolution of the particle configurations is ruled by the Markov generator $\cL_F$ which acts on functions $f:\Omega\to\R$ as follows:
\begin{align} \cL_F f(\eta)=& \sum_{k\in\Z}p\eta(k-1)\eta(k)(1-\eta(k+1))\big[f(\eta^{k,k+1})-f(\eta)\big] \notag\\ & + \sum_{k\in\Z}  q\eta(k+1)\eta(k)(1-\eta(k-1))\big[f(\eta^{k,k-1})-f(\eta)\big]\label{eq:L}\end{align}
where $\eta^{k,\ell}$ is the configuration obtained from $\eta$ after the exchange of the occupation variables $\eta(k) \leftrightarrow \eta(\ell)$, namely:
$\eta^{k,\ell}(j):=\eta(k)\mathbf{1}_{j=\ell} + \eta(\ell) \mathbf{1}_{j=k}+\eta(j)\mathbf{1}_{j\notin\{ k,\ell\}}.$
In other words, as it can be read on the generator, particles are displayed on the lattice $\Z$ and jump to their neighbouring sites at rates which encode the following rules:
\begin{itemize}
	\item a jump to the right from site $k$ to site $k+1$ occurs with rate $p$ if and only if site $k$ is occupied by a particle, site $k+1$ is empty \emph{and} site $k-1$ is occupied ; 
	\item a jump to the left from site $k$ to site $k-1$ occurs with rate $q$ if and only if site $k$ is occupied by a particle, site $k-1$ is empty \emph{and} site $k+1$ is occupied.
\end{itemize}
We say that  site $k\in\Z$ is occupied by an \emph{active particle}\footnote{\mm{Note that we use a definition of \emph{active particles} which is different from the one chosen in \cite{DAES} for instance, but it has no impact on the rest of the paper.}} if it can jump either to $k-1$ or $k+1$ with positive rate, in other words, it is such that
\[\eta(k-1)\eta(k)(1-\eta(k+1))+(1-\eta(k-1))\eta(k)\eta(k+1)=1. \]
Let us now map this process onto another exclusion process. We need to introduce some notation. Let 
$L<R$ be two integers, and define:
\[\bar{\Omega}_{L,R}:=\Bigg\{\eta\in\Omega\; ; \; \begin{cases} \eta(k)=0 & \text{ if }k\geq R, \\ \eta(k)=1 & \text{ if }k\leq L, \\  \eta(k)+\eta(k+1)\geq 1 & \text{ if }k<R. \end{cases}\; \Bigg\}\]
and \[ \bar\Omega:=\bigcup_{L<R} \bar\Omega_{L,R}.\]
In other words, if $\eta \in \bar\Omega_{L,R}$ then the active particles are all contained in the box $\{ L, L+1, \dots, R-1\}$.
The set $\bar\Omega$ is remarkable because it is preserved along time evolution of the dynamics generated by $\cL_F$.

For any $\eta\in\bar \Omega$, we can label particles from right to left, by the following recursive procedure: define 
\begin{align*}
	X_1(\eta) & :=\max\lbrace j; \eta(j)=1\rbrace , \\ X_{i+1}(\eta) & :=\max\big\{j<X_i(\eta)\; ; \;  \eta(j)=1\big\}.
\end{align*}
We are now ready to construct the mapping: let  $\mathfrak{S}:\bar{\Omega}\rightarrow \{0,1\}^{\N^*}$ be the application such that,
\[ \text{ for any }\eta\in\bar\Omega, \quad \sigma:=\mathfrak{S}(\eta) \quad \text{satisfies} \quad \sigma(i)=1-\eta(X_i-1).\] 
In other words, site $i$ is occupied by a particle in $\sigma$ (\textit{i.e.}~$\sigma(i)=1$) if and only if the  $i$-th particle in $\eta$ has an empty site to its left.

\begin{figure}[h!]
	\begin{center}
		
		\begin{tikzpicture}[scale=0.65]
			\draw[thick] (-9.5,0) -- (8.5,0);
			\foreach \x in {-9,...,8}
			\draw (\x, -0.1) -- (\x,0.1);
			\fill (-9,0.6) circle(0.35);
			\draw[white] (-9,0.6) node{\footnotesize  $\mathbf{11}$};
			\fill (-8,0.6) circle(0.35);
			\draw[white] (-8,0.6) node{\footnotesize  $\mathbf{10}$};
			\fill (-7,0.6) circle(0.35);
			\draw[white] (-7,0.6) node{\footnotesize  $\mathbf{9}$};
			\fill (-6,0.6) circle(0.35);
			\draw[white] (-6,0.6) node{\footnotesize  $\mathbf{8}$};
			\fill (-4,0.6) circle(0.35);
			\draw[white] (-4,0.6) node{\footnotesize  $\mathbf{7}$};
			\fill (-2,0.6) circle(0.35);
			\draw[white] (-2,0.6) node{\footnotesize  $\mathbf{6}$};
			\fill (-1,0.6) circle(0.35);
			\draw[white] (-1,0.6) node{\footnotesize  $\mathbf{5}$};
			\fill (0,0.6) circle(0.35);
			\draw[white] (0,0.6) node{\footnotesize  $\mathbf{4}$};
			\fill (2,0.6) circle(0.35);
			\draw[white] (2,0.6) node{\footnotesize  $\mathbf{3}$};
			\fill (4,0.6) circle(0.35);
			\draw[white] (4,0.6) node{\footnotesize  $\mathbf{2}$};
			\fill (5,0.6) circle(0.35);
			\draw[white] (5,0.6) node{\footnotesize  $\mathbf{1}$};
			\draw (-6,-.5) node{$L$};
			\draw (6,-.5) node{$R$};
			\draw[-stealth, thick, red] (5,1) to[bend left] (6,1);
			\draw (5.5,1.4) node{$p$};
			\draw[-stealth, thick, green] (0,1) to[bend left] (1,1);
			\draw (0.5,1.4) node{$p$};
			\draw[-stealth, thick, orange] (-6,1) to[bend left] (-5,1);
			\draw (-5.5,1.4) node{$p$};
			\draw[-stealth, thick, blue] (4,1) to[bend right] (3,1);
			\draw (3.5,1.4) node{$q$};
			\draw[-stealth, thick, magenta] (-2,1) to[bend right] (-3,1);
			\draw (-2.5,1.4) node{$q$};
			\node[anchor=west] at (8.7,0) {FASEP on $\Z$};

			\begin{scope}[xshift=-7cm, yshift=-3.3cm, scale=1.2] 
				\draw[thick] (0,0) -- (11,0);
				\node[anchor=west] at (11.1,0) {half-line ASEP};

				\draw[thick] (-1,0) circle(1);
				\draw (-1,0) node{\footnotesize reservoir};
				\foreach \x in {1,...,10}
				\draw (\x, 0.1) -- (\x,-0.1) node[anchor=north]{$\x$};
				\foreach \x in {2,3,6,7} 
				\fill (\x, 0.5) circle(0.3);
				\draw[-stealth, thick, red] (-0.3,0.8) to[bend left] (1,0.8);
				\draw (0.5,1.4) node{$p$};
				\draw[-stealth, thick, green] (3,0.8) to[bend left] (4,0.8);
				\draw (3.5,1.4) node{$p$};
				\draw[-stealth, thick, orange] (7,0.8) to[bend left] (8,0.8);
				\draw (7.5,1.4) node{$p$};
				\draw[-stealth, thick, blue] (2,0.8) to[bend right] (1,0.8);
				\draw (1.5,1.4) node{$q$};
				\draw[-stealth, thick, magenta] (6,0.8) to[bend right] (5,0.8);
				\draw (5.5,1.4) node{$q$};
			\end{scope}
		\end{tikzpicture} 
	\end{center}
	\caption{The top figure represents a configuration $\eta$ in $\bar\Omega$ and the possible transitions with their respective rates. The bottom figure represents the mapped configuration $\mathfrak{S}(\eta)$ on $\lbrace 0,1\rbrace^{\mathbb N^*} $ and the possible transitions, in the same color as the corresponding transitions in $\eta$.}
	\label{fig:FASEP-OpenASEP}
\end{figure}
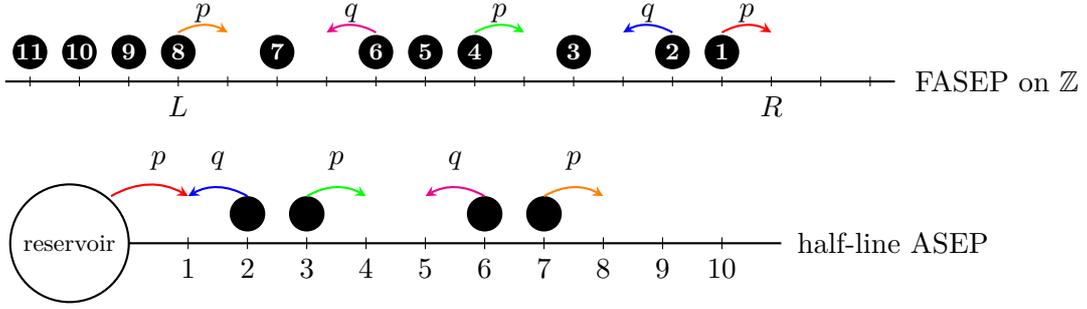

In particular, the \emph{step initial condition} $\eta_0$ given by \[ \eta_0(k)=\mathbf{1}_{k\leqslant k_0} \text{  for some }k_0\in\Z,\] belongs to $\bar\Omega$, and it corresponds to an \emph{empty configuration} $\sigma_0=\mathfrak{S}(\eta_0)$, namely \begin{equation}\label{rem:step}\sigma_0(k)=0 \qquad\text{for any } k\in\N^*.\end{equation}  These initial conditions will be important in what follows. 
The dynamics of the mapped process is described in the following lemma. Note that with the step initial condition, for all $t\geq 0$, $X_i(\eta_t)-X_{i+1}(\eta_t)\in\{1,2\}$, and $\mathfrak{S}(\eta_t)$ encodes all the information about these spacings.
\begin{lemma}
	Assume that the generator of the Markov process $\{\eta_t\}_{t\geq 0}$ is $\cL_F$ given in \eqref{eq:L}, then $\{\mathfrak{S}(\eta_t)\}_{t\geq 0}$ is an \emph{asymmetric simple exclusion process} on the infinite half-line $\N^*$ with a boundary reservoir which injects particles at rate $p$. More precisely its generator is given as follows: for any $\sigma \in \Sigma:=\{0,1\}^{\N^*}$, for any $f:\Sigma\to\R$,
	\begin{align}\cL f(\sigma)=& \sum_{k=1}^\infty \Big(p\sigma(k)(1-\sigma(k+1))+q\sigma(k+1)(1-\sigma(k))\Big)\Big[f(\sigma^{k,k+1})-f(\sigma)\Big] \notag \\ & \qquad +\,p(1-\sigma(1))\Big[f(\sigma^1)-f(\sigma)\Big],\label{e:genASEP}\end{align}
	where $\sigma^1$ is obtained after the creation of a particle at site 1, \emph{i.e.}~$\sigma^1(k)=\mathbf{1}_{k=1}+\sigma(k)\mathbf{1}_{k\neq 1}$.
\end{lemma}

\begin{proof}
	The  proof is straightforward by looking at every possible transition, see Fig.~\ref{fig:FASEP-OpenASEP}.
\end{proof}

\section{Main results and strategy of the proof}\label{s:results}

From now on we consider the half-line ASEP denoted by $\{\sigma_t\}_{t\geq 0}$, which is a Markov process on $\Sigma=\{0,1\}^{\N^*}$ generated by $\cL$ given in \eqref{e:genASEP}. 

\subsection{Hopf-Cole transform}

Let us define the \emph{height function} associated with the particle system: for any $k\in\N^*$, 
\[h_t(k):=h_t(0)+\sum_{i=1}^k (2\sigma_t(i)-1), \qquad \text{with }h_t(0)=2\sum_{i=1}^\infty  \sigma_0(i) - 2\sum_{i=1}^\infty  \sigma_t(i).\] 
This is well-defined when $\sum_{k=1}^\infty\sigma_0(k)<\infty$. In that case,  note that $- h_t(0)$ is equal to \emph{twice} the number of particles that have entered into the system between times $0$ and $t$ (no particle can exit the system by hypothesis). For any initial condition $\sigma_0\in \Sigma$ and any $t>0$, this number is bounded by a Poisson random variable with parameter $pt$, and in particular, it is almost surely finite. In the case where the initial number of particles in the system is infinite, we can still define $h_t(0)$ as minus twice the number of particles that have entered the system. Thus, starting from any initial condition $\sigma_0\in \Sigma$, the height function satisfies: for any $t>0$, any $k\in\N$,  $h_t(k)<\infty$ a.s. For $\nu,\lambda\in\R$ which will be chosen later we then define the \emph{Hopf-Cole transform}\begin{equation}Z_t(k):=e^{-\lambda h_t(k)+\nu t}, \qquad k \in \N.\label{eq:HopfCole}\end{equation}
Alternatively, $Z_t$ can be defined as a function of the positions of the particles in the FASEP. It is straightforward to check that, under the coupling described in Section~\ref{s:coupling}, for the FASEP started from the step initial condition, the position of the $m$-th particle in the FASEP is related to the height function with empty initial condition through
\[X_m(\eta_t)=\sum_{i=m}^\infty\sigma_t(i)-m.\]
Consequently, the Hopf-Cole transform can be recast as 
\begin{equation}
	Z_t(k)=e^{2\lambda X_{k+1}(\eta_t)+(3k+2)\lambda+\nu t}.
\end{equation}
We will now mainly work with the Hopf-Cole transform $Z$ and we state our results in terms of this quantity. In the following, for any function $f:\Z\to\R$ we define its left and right gradients by
\[ \nabla^+f(k):=f(k+1)-f(k), \qquad \nabla^-f(k):=f(k-1)-f(k)\] and its discrete Laplacian by
$\Delta f(k):=f(k+1)+f(k-1)-2f(k).$ One knows that
$Z$ satisfies
\[dZ_t(k)=(\nu Z_t(k)+\cL Z_t(k))dt+dM_t(k), \qquad k \in\N,\]
where $\{M_t(k)\}_{t\geqslant 0}$ are martingales whose quadratic variations will be computed below (\mm{Lemma \ref{lem:quadratic}}). As in \cite{corwin2016open,parekh2017kpz} (see also \cite{gonccalves2017nonequilibrium}),  we look for conditions on $\nu,\lambda$ so that $\nu Z_t(k)+\cL Z_t(k)$ can be rewritten as $D \Delta Z_t(k)$ for some diffusion coefficient $D$. 
After straightforward computations (given in Appendix \ref{app:comp}) we choose 
\begin{equation}\label{eq:choice} \lambda=\frac12 \log \frac{q}{p}, \qquad \nu=q+p-2\sqrt{pq},\end{equation} which imply: for any $k\in\N^*$,
\begin{equation}\label{eq:discDelta}
	dZ_t(k)=D \Delta Z_t(k) dt + dM_t(k), \qquad \text{with } D=\sqrt{pq}.\end{equation}
It remains to define $Z_t(-1)$ in order that  \eqref{eq:discDelta} remains valid at the boundary point $k=0$, which can be done if we let
\begin{equation}\label{eq:boundary} Z_t(-1): = \mu Z_t(0), \qquad\text{with } \mu=\sqrt{\frac{q}{p}}=e^\lambda.\end{equation}
We therefore obtain the following result.
\begin{lemma}\label{lem:quadratic}We assume the choice of parameters \eqref{eq:choice} and define $\mu=\sqrt{q/p}$. Let $\Delta^\mu$ be the discrete Laplacian on $\N$ with the following boundary condition:
	\begin{equation}
		\label{eq:deltamu}
		\Delta^\mu f(k):= \begin{cases}
			f(k+1)+f(k-1)-2f(k) & \text{ if } k >0 \\
			f(1)+\mu f(0) - 2 f(0) & \text{ if }k=0.
		\end{cases}
	\end{equation}
	Then, for any $k\neq \ell\in\N$, the following three continuous-time processes are martingales: 
	\begin{align} &M_t(k):= Z_t(k)-Z_0(k)-D\int_0^t \Delta^\mu Z_s(k)ds,\\
	&M_t(k)^2 - \int_0^t \frac{d}{ds} \oo{\langle M(k)\rangle}_s \; ds,\\
	&M_t(k)M_t(\ell),\vphantom{\bigg(}
	\end{align} where $\oo{\langle M(k)\rangle}_\cdot$ denotes the \oo{predictable} quadratic variation of the martingale $M_\cdot(k)$.
	Moreover, 
	\[
	\frac{d}{dt}\oo{\langle M(k)\rangle}_t = \begin{cases} Z_t(k)^2\left({\sigma}_t(k)(1-{\sigma}_t(k+1))\frac{(p-q)^2}{p}+{\sigma}_t(k+1)(1-{\sigma}_t(k))\frac{(p-q)^2}{q}\right) & \text{if } k >0 ,\vphantom{\bigg(}
		\\
		Z_t(0)^2(1-{\sigma}_t(1))\frac{(p-q)^2}{p} & \text{if } k=0. \vphantom{\bigg(} \end{cases}
	\]
\end{lemma}

\begin{proof}
	This is \oo{standard (see e.g.~\@\cite[Appendix 1.5]{KL}): for instance, the predictable quadratic variation satisfies,
	for $k>0$,
	 \begin{align*}
	     \frac{d}{dt}\oo{\langle M(k)\rangle}_t&=\left[p\sigma_t(k)(1-\sigma_t(k+1))(e^{2\lambda}-1)^2+q\sigma(k+1)(1-\sigma(k))(e^{-2\lambda}-1)^2
	     \right]Z_t(k)^2,
	     \\
	     \frac{d}{dt}\oo{\langle M(0)\rangle}_t&=p(1-\sigma_t(1))(e^{2\lambda}-1)^2Z_t(0)^2.
	    \end{align*}}

\end{proof}

\subsection{Weak asymmetry}

From now on, we consider the half-line ASEP in the \emph{weak asymmetry regime} where 
\begin{equation}
	\label{eq:pq} p = \tfrac{1}{2}e^{\varepsilon} \qquad\text{and} \qquad q=\tfrac12 e^{-\varepsilon}, \qquad \text{for }\varepsilon>0.
\end{equation}
Rewriting everything in terms of $\varepsilon$, 
the Hopf-Cole transform reads as
\[ Z_t(k)=e^{\eps h_t(k)+\nu t}, \qquad \text{where }\nu =\tfrac12e^\eps(e^{-\eps}-1)^2.\] 
The boundary parameter $\mu$ appearing in \eqref{eq:boundary} and the diffusion coefficient $D$ become
\[ \mu=e^{-\eps}, \qquad D=\tfrac12.\]
In this weak asymmetry regime, the quadratic variation of the above martingale satisfies:

\begin{lemma}\label{lem:quadeps} As $\eps\to 0$ we have
	\begin{align} \label{eq:quadx}
		\frac{d}{dt}\oo{\langle M(k)\rangle}_t &= \eps^2Z_t(k)^2+\nabla^+Z_t(k)\nabla^-Z_t(k)+o(\eps^2)Z_t(k)^2,\quad \text{for any }k>0\\
		\frac{d}{dt}\oo{\langle M(0)\rangle}_t&=\eps^2Z_t(0)^2-\eps Z_t(0)\nabla^+Z_t(0)+o(\eps^2)Z_t(0)^2.\label{eq:quad0}
	\end{align}
	
\end{lemma}
Note that the term $\eps Z_t(0)\nabla^+Z_t(0)$ at the boundary is new in our case, it did not appear in the case of \cite{corwin2016open}.

\begin{proof}
\oo{Using 
\[\frac{Z_t(k+1)}{Z_t(k)}=\sqrt{\frac q p}\;\big(1-{\sigma}_t(k+1)\big)+\sqrt{\frac p q}\;{\sigma}_t(k+1),\]
we can rewrite the expression in Lemma \ref{lem:quadratic} as
\begin{align*}\frac{d}{dt} \langle M(k)\rangle_t=&\; \sqrt{pq}(Z_t(k-1)-\mu^{-1} Z_t(k))(\mu Z_t(k+1)-Z_t(k)) \\ & +\sqrt{pq}(\mu^{-1} Z_t(k-1)-Z_t(k))(Z_t(k+1)-\mu Z_t(k)).
	\end{align*}
	In the weak asymmetry regime, we get
	\begin{align*} 
	\nabla^+Z_t(k)&=  \eps Z_t(k) \big( 2 \sigma_t(k+1)-1 + o(1)\big),\\
	\nabla^-Z_t(k)&= \eps Z_t(k) \big(1- 2 \sigma_t(k) + o(1)\big),	
	\end{align*}
	which allows us to conclude easily.}
\end{proof}

\subsection{Main theorems}

Before stating our main results, let us start by defining the notion of solution for the stochastic heat equation which  is at the core of the convergence results of this paper.

\begin{definition}\label{d:she} Let $\xi$ be the standard space-time white noise on $\R_+\times\R_+$, on some probability space $(\Omega,\mathcal{F},\mathbb{P})$.
	We say that $(\mathcal{Z}_t(x))_{t\in[0,T], x\in\R_+}$ solves the stochastic heat equation (SHE)
	\begin{equation} \label{eq:she}
		\partial_t\mc Z=\frac{1}{2}\partial_{xx} \mc Z+\mc Z\xi 
	\end{equation} on the time interval $[0,T]$,  with Dirichlet boundary condition, and initial condition $\mc Z_{\mathrm{ini}}$,  if for any $t\in (0,T]$, $x\in\R_+$
	\begin{equation}
		\mc Z_t(x) = \int_{\mathbb R_+}dy P_t^{\rm Dir}(x,y) \mc Z_{\mathrm{ini}}(y)  + \int_0^t ds \int_{\mathbb R_+}dy P^{\rm Dir}_{t-s}(x,y) \mc Z_s(y)\xi(s,y),
		\label{eq:Duhamelform}
	\end{equation}
	where the Dirichlet half-space heat kernel $P_t^{\rm Dir}$ is defined on $(\R_+)^2$ by 
	\begin{equation}
		P_t^{\rm Dir}(x,y)=\frac{1}{\sqrt{2\pi t}}\left(e^{-(x-y)^2/(2t)} - e^{-(x+y)^2/(2t)} \right),
		\label{eq:PtDir}
	\end{equation}
and where the second integral in \eqref{eq:Duhamelform} is in It\^o-Walsh sense.
\end{definition}

\subsubsection{Empty initial condition}

We define, for any $x\in\varepsilon^2\N$, the scaled Hopf-Cole process
\begin{equation} \label{eq:scaled1}\mathcal{Z}_t^\eps(x):=\eps^{-2} Z_{\eps^{-4}t}(\eps^{-2}x)\end{equation} and we extend $\mathcal{Z}_t^\eps(\cdot)$ to the continuous half-line $\R_+$ by linear interpolation. 

\begin{remark} Let us emphasize that the scaling in \eqref{eq:scaled1} is \emph{not} the one which will appear later when the initial condition is supposed to be \emph{near-equilibrium} (see Section~\ref{ssec:neareq}, \eqref{eq:scaled} and Definition \ref{de:near-eq-cont}), nor the one in 
	\cite{parekh2017kpz}, where the macroscopic initial condition is the delta Dirac function $\delta_0$ and the prefactor there is $\eps^{-1}$ instead of $ \eps^{-2}$.
\end{remark}

The initial condition \eqref{rem:step} implies  that 
$Z_0(k)=\mu^k,$ for any $k\in \N$. In order to identify the possible limit points of the rescaled process \eqref{eq:scaled1}, and in particular the limit initial condition when $t=0$, one needs to understand the limit, as $\eps \to 0$, of 
\begin{equation}\label{eq:init}\eps^2 \sum_{k=0}^\infty \phi_\eps(\eps^2k) \mathcal{Z}_0^\eps(\eps^2k),\end{equation} where $\phi_\eps$ is a test function, whose behaviour near the origin should reflect the boundary condition in the discrete Laplacian \eqref{eq:deltamu}. We take $\phi,\psi:\R_+\rightarrow\R$ smooth and compactly supported functions with $\phi(0)=\psi'(0)=0$ and $\psi(0)=1$, then define $\phi_\eps=\phi+\eps\phi'(0)\psi$. Using a Taylor series approximation near $0$, one can easily check that the limit of \eqref{eq:init} is $2\phi'(0)$. Therefore,
the initial condition of the continuous limit of the rescaled process $\mathcal{Z}_t^\varepsilon$ is $-2\delta_0'$ where $\delta_0'$ is the derivative of the delta Dirac distribution. \mm{In Section \ref{sec:limitZ} we will give another computation which illustrates the emergence of this initial condition, see Lemma \ref{lem:expectationinitialdata}.}

Our first result, which will be proved in Section \ref{sec:proofshe}, is the existence and uniqueness of solutions to \eqref{eq:she} for the above $\delta_0'$ initial condition and Dirichlet boundary condition:
\begin{proposition} Let $\xi$ be the standard space-time white noise on $\R_+\times\R_+$, on some probability space $(\Omega,\mathcal{F},\mathbb{P})$. There exists a $C(\mathbb R_+)$--valued process $\left(\mathcal Z_t\right)_{t\geq 0}$ which is adapted to the filtration $\mathcal F_t=\sigma\left( \lbrace \xi(s,\cdot)\rbrace_{s\leq t}\right)$, and solves the SHE \eqref{eq:she} in the sense of Definition \ref{d:she} with initial condition $\mathcal Z_{\rm ini}=-2\delta_0'$. 
	This solution is unique in the class of adapted continuous processes satisfying
	\begin{equation}
		\sup_{\substack{
				x\in \mathbb R_+ \\ s\in (0,t)
		}} \left\lbrace s^2 \mathbb E\left[ \mathcal Z_s(x)^2 \right] \right\rbrace<\infty.
		\label{eq:classoffunctions}
	\end{equation} 
	\label{prop:existenceuniqueness}
\end{proposition}

\begin{remark}The existence and uniqueness of solutions to \eqref{eq:she} are already proved in \cite[Theorem 4.1, Prop. 4.2]{parekh2019positive} when $\mc Z_{\rm ini}$ is \emph{near-equilibrium} (see Definition \ref{de:near-eq-cont} below). This is not the case of the $\delta_0'$ initial condition that we consider here, and therefore we need to provide a new proof. Existence and uniqueness is also proved in \cite[Proposition 4.3]{parekh2017kpz} for Robin type boundary condition and $\delta_0$ initial condition. As we will see below in Section \ref{sec:proofshe}, the proof of  Proposition \ref{prop:existenceuniqueness} involves different estimates than the ones in \cite{parekh2017kpz, parekh2019positive}. \end{remark}

The main result of this paper is the following convergence:

\begin{theorem}\label{thm:conv} Fix $T>0$.
	Assume the initial particle configuration is empty as in \eqref{rem:step}. Then the rescaled process $\{\mathcal{Z}_t^\eps\}_{t\in{(0,T]}}$ converges as $\eps \to 0$ to the solution of the stochastic heat equation \eqref{eq:she} on the time interval $[0,T]$ with Dirichlet boundary condition, and  initial condition $\mc Z_{\mathrm{ini}} =-2\delta_0'$ (as defined in Definition \ref{d:she}),  in the sense of weak convergence of probability measures on the path space $D({(0,T]},C(\R_+))$ endowed with the Skorokhod topology.
\end{theorem}

\mm{Theorem \ref{thm:conv} will be proved in Section \ref{s:deltaprime}. In particular, the proof involves an intermediate convergence result, dealing with a \emph{near-equilibrium} initial condition, which is detailed in the next paragraph.}

\subsubsection{Near-equilibrium initial condition} \label{ssec:neareq}

Let us now study a simpler case of near-equilibrium initial condition.
For that purpose, in this section we define another scaled process as follows:
\begin{equation} \label{eq:scaled}\mathscr{Z}_t^\eps(x):=\eps^2 \rZ_t(x)=Z_{\eps^{-4}t}(\eps^{-2}x), \qquad x \in \eps^2 \N,
\end{equation}
which differs from $\mathcal{Z}_t^\eps$ by a factor $\eps^2$, and we allow $Z_0(k)$ to be different from $\mu^k$. \mm{As before we extend $\mathscr{Z}_t^\eps(\cdot)$ to $\R_+$ by interpolation.}

\begin{definition}\label{de:near-eq-cont}
	We say that a sequence of random functions  $\mathscr{F}^\eps \in C(\R_+)$ is \emph{near-equilibrium} if it satisfies the following: there exists $a>0$ such that, for any $n \in \N$, any $\alpha \in {[0},\frac12)$, there exists some constant $C=C(\alpha,n)>0$ such that, for any $x,x' \in \R_+$, any $\eps >0$ \oo{small enough}, 
	\begin{align}
		\|\mathscr{F}^\eps(x)\|_n & \leqslant Ce^{ax} \label{e:momentZ0}\\ 
	\text{and} \qquad 	\|\mathscr{F}^\eps(x)-\mathscr{F}^\eps(x')\|_n & \leqslant C|x-x'|^\alpha e^{a(x+x')}, \label{e:momentdiffZ0}
	\end{align}
	where $\|G\|_n:=\E[|G|^n]^{1/n}$ denotes the $L^n$--norm with respect to the probability measure.
\end{definition}

An auxiliary -- although important -- result is the following:

\begin{theorem}\label{t:cvneareq}
	Fix $T>0$. Assume that the (possibly random) initial condition $\scZ_0 \in C(\R_+)$ is near-equilibrium (in the sense of Definition~\ref{de:near-eq-cont}), and that $\mathscr{Z}_0^\eps$ weakly converges for the standard topology on $C(\R_+)$ to some initial condition $\mathscr{Z}_{\rm ini}$ as $\eps\rightarrow 0$. 
	
	Then the rescaled process $\left(\mathscr{Z}_t^\eps\right)_{t\in[0,T]}$ converges as $\eps\rightarrow 0$ to the solution of the stochastic heat equation \eqref{eq:she}
	with initial condition $\mathscr{Z}_{\rm ini}$ in the time interval $[0,T]$ (as defined in Definition \ref{d:she}), in the sense of weak convergence of probability measures on the path space $D([0,T],C(\R_+))$ endowed with the Skorokhod topology.
\end{theorem}

\mm{Theorem \ref{t:cvneareq} will be proved in Section \ref{sec:neareq}.}

\subsubsection{An example of near-equilibrium initial condition} 

Theorem \ref{t:cvneareq} may be useful independently from its \oo{relevance} to the proof of Theorem \ref{thm:conv}. Let us consider the half-line ASEP generated by $\mathcal{L}$ in the weakly asymmetric regime \eqref{eq:pq}, but with initial condition given by product Bernoulli, that is, we assume that the variables $\sigma_0(k), k \in \N^*$, are independent Bernoulli variables with parameter $\varrho$.  Let us scale $\varrho =( 1-\eps (B+\frac12))/2$, so that $\eps h_0(\eps^{-2}x)$ converges to a Brownian motion with drift $-(B+1/2)$.
Then, it can be shown that $\mathscr{Z}_0^\eps$ is near-equilibrium\footnote{The term near-equilibrium actually comes from the fact that such initial condition is stationary for the ASEP on $\mathbb Z$.}, and thus Theorem \ref{t:cvneareq} can be applied, to find that $(\mathscr{Z}_t^\eps)$ converges to the SHE with Dirichlet boundary  condition, and initial condition given by the exponential of a Brownian motion with drift $-(B+1/2)$. 

Denoting by $\mathscr Z(t,x)$ this solution, \cite{parekh2019positive} showed that we have the identity in distribution 
\begin{equation}
	\lim_{x\to 0} \frac{\mathscr Z(t,x)}{x} = \widetilde{\mathscr Z}(t,0)
	\label{eq:symmetry}
\end{equation}
where $\widetilde{\mathscr Z}$ is the solution of the SHE on $\mathbb R_+$ with Robin boundary parameter $B$ and delta Dirac function  as initial condition. In the special case $B=-1/2$, that is $\varrho=1/2$, the law of $\widetilde{\mathscr Z}(t,0)$ is explicitly known and related to eigenvalue GOE (Gaussian Orthogonal Ensemble) statistics   \cite{barraquand2018stochastic, parekh2017kpz}. For other values of $B$, the law of $\widetilde{\mathscr Z}(t,0)$ was computed recently \cite{imamura2022solvable} (see also \cite{krajenbrink2020replica}).

\begin{remark}
	Discrete analogues of the identity \eqref{eq:symmetry}, allowing to exchange the roles of the boundary and initial condition parameters, also exist for directed polymers \cite[Prop. 8.1]{barraquand2018half}, last passage percolation \cite[Lemma 6.1]{baik2018pfaffian} or more general models defined through Pfaffian Schur measures \cite[Corollary 7.6]{baik2001algebraic}. This suggests to compare the height function at the origin for ASEP on $\mathbb N^*$  in the following two situations:
	\begin{enumerate}
		\item The reservoir has injection rate $p$, ejection rate $0$, and the initial condition is i.i.d. Bernoulli with parameter $\varrho$; 
		\item The reservoir has injection parameter $\alpha=p\varrho$, ejection parameter $ \gamma=q(1-\varrho)$, and the initial configuration is empty.
	\end{enumerate}
	When $p=1, q=0$, that is in the case of TASEP, the height functions at the origin in the above two situation have the same distribution\footnote{More precisely, one needs to consider the TASEP on $\mathbb N^*$ with an initial condition such that the site $1$ is occupied, and all other sites are occupied according to Bernoulli i.i.d.~random variables. Then, the height function for this model has the same distribution \cite{prahofer2001current} as the height function associated to the last passage percolation model with boundary considered in \cite{baik2018pfaffian}. The symmetry between the boundary and the initial condition parameters follows from \cite[Lemma 6.1]{baik2018pfaffian} (choosing the parameter $\alpha$ there as $\alpha=\varrho$).}. 
	Whether this equality in law generalizes to ASEP is an open question.  
\end{remark}

\subsection{Strategy of the proof}
\label{sec:strategy}
As stated in Definition \ref{d:she}, a mild solution $\mathcal Z_t(x) $ to the stochastic heat equation \eqref{eq:she} with Dirichlet boundary condition satisfies \eqref{eq:Duhamelform}. Solutions can equivalently be characterized by the following martingale problem. The equivalence of the two notions of solutions is proved in {\cite[Prop. 5.6]{parekh2017kpz}} in the case of Neumann type boundary condition. The argument applies mutatis mutandis in the Dirichlet case. 

Let us define the space of test functions
\begin{equation} \label{eq:testfunction}
		\mathcal{H}:=\big\{\phi\in C_c^\infty(\R)\ \colon\ \phi(0)=0\big\},
	\end{equation}
where $C_c^\infty(\R)$ is the set of smooth functions defined on $\R$, with compact support. We denote by $(f,g):=\int_{\R_+} f(x)g(x) d x$ the standard Euclidean product in $L^2(\R)$.
\begin{definition}\label{de:solution} A solution \oo{$\mathscr Z=(\mathscr Z_t(x))_{t\in[0,T],x\in\R_+}$} to the martingale problem for the stochastic heat equation
	\eqref{eq:she} with Dirichlet boundary condition, initial condition $\mathscr{Z}_{\mathrm{ini}}$ and time interval $[0,T]$,  is {a random element of $C([0,T],C(\R_+))$}, such that for all $\phi\in\mc H$ and $t\in(0,T]$,  the following quantities are martingales:
	\begin{align}
		&N_t(\phi):=(\mathscr{Z}_t,\phi)-(\mathscr{Z}_{\rm ini},\phi)-\frac{1}{2}\int_0^t(\mathscr{Z}_s,\phi'')ds,\label{e:contmartpb}\\
		&Q_t(\phi):=(N_t(\phi))^2-\int_0^t(\mathscr{Z}_s^2,\phi^2)ds.\label{e:contmartpb2}
	\end{align}
	 \mm{This notion of solution is equivalent to the one given in Definition \ref{d:she}.}
\end{definition}

We divide the proof of Theorem \ref{thm:conv} (Section \ref{s:deltaprime}) and Theorem \ref{t:cvneareq} (Section \ref{sec:neareq}) into several steps, following the strategy of \cite{parekh2017kpz}: 
\begin{enumerate}
	\item For a process $(\mathscr{Z}_t^\eps)_{t\in[0,T]}$ starting from a near-equilibrium initial condition   $\mathscr{Z}_0^\eps$, which satisfies the weak convergence $\mathscr{Z}_0^\eps \Rightarrow \mathscr{Z}_{\rm ini}$  as $\eps\to 0$,   we prove its convergence  towards the solution to the heat equation with Dirichlet boundary condition starting from $\mathscr{Z}_{\rm ini}$  (a.k.a.~Theorem~\ref{t:cvneareq}). This is split in two steps: proof of tightness (Section~\ref{s:tightnessneareq}), and identification of the limit point (Section \ref{s:limitpoints}). The latter step uses the martingale problem above: we show that the discrete martingale problem gives the continuous one in the limit, and the control of the error terms is a consequence of the tightness estimates. 
	
	\item We come back to the initial condition $Z_0(k)=\mu^k$ of Theorem~\ref{thm:conv}. We prove that at time $\delta>0$, $\mathcal{Z}^\eps_\delta$  defined in \eqref{eq:scaled1} is \emph{near-equilibrium} in the sense of Definition \ref{de:near-eq-cont}. In order to prove this result, we first establish in Proposition~\ref{p:momentZ} that $\mathcal{Z}^\eps_\delta$ has bounded moments of any order, using exact integral formulas (Section~\ref{subs:momentsZ}). Then, using Duhamel's formula, Hölder properties for the heat kernels (Lemma~\ref{l:Holder-heat-kernel} and Lemma~\ref{l:integralF}), we deduce the Hölder properties for $\mathcal Z^\eps_\delta$ (Proposition~\ref{prop:HolderZdelta'}). This is the purpose of Section \ref{subs:neareq}. Similarly to the first point, this property gives us tightness in $D([\delta,T],C(\R_+))$ for any $\delta \in (0,T)$, and that any limit point is solution to \eqref{eq:she} with Dirichlet boundary condition.
	
	\item The missing point is to push $\delta$ to $0$. We use consistency, and identify the initial condition. The structure of the argument is similar to \cite{parekh2017kpz} but we use a different method in order to obtain second moment bounds (Lemma \ref{lem:limit2bis}). We rely on the analysis of exact integral formulas for the moments of $Z_t(k)$ obtained using a Markov  duality stated in \cite{barraquand2022markov}. 
	
\end{enumerate}


\section{Existence and uniqueness of the macroscopic solution} 
\label{sec:proofshe}  
This section is devoted to the proof of existence and uniqueness of the solution to  the SHE (Definition \ref{d:she}) with Dirichlet boundary condition and initial condition $\mathcal Z_{\rm ini}=-2\delta_0'$, namely Proposition \ref{prop:existenceuniqueness}. Although we follow a standard argument from \cite{walsh1986introduction} of Picard iteration (see also \cite[Section 4]{parekh2017kpz}), the initial condition that we consider is much more singular than the one considered in previous works, and this requires us to prove refined estimates. 

It will be convenient to introduce the notation \oo(recall that $P_t^{\rm Dir}$ is defined in \eqref{eq:PtDir}) : for any $x \in \R_+$, $t> 0$,
\begin{equation}
	dP^{\rm Dir}_t(x,0):= - P_t^{\rm Dir}(x,\cdot)\ast 2\delta_0'(\cdot)  =  2\partial_y P_t^{\rm Dir}(x,y)\big\vert_{y=0} = \frac{2\sqrt{\frac{2}{\pi }} x e^{-\frac{x^2}{2 t}}}{t^{3/2}}.
	\label{eq:defdP}
\end{equation}
We start with an estimate involving the quantity $dP^{\rm Dir}_t(x,0)$, which we will re-use in Sections \ref{sec:limitZ} and \ref{sec:conclusion}. 
\begin{lemma}
	Define, for any $s\in(0,t)$ and $x\in\R_+$,  
	\begin{equation}
		G_{t}(s,x): = \int_{\mathbb R_+} dy \left( P^{\rm Dir}_{t-s}(x,y) \right)^2 \left( dP_s^{\rm Dir}(y,0)\right)^2.
		\label{eq:defGt}
	\end{equation}
	There exists a constant $C>0$ such that for any $t>0$, any $s\in(0,t)$ and any $x\in \mathbb R_+$, 
	\begin{equation}
		G_{t}(s, x) \leq C \frac{\sqrt{t}}{\sqrt{s(t-s)}} \left(dP^{\rm Dir}_t(x,0)\right)^2.
		\label{eq:boundGt}
	\end{equation}
	\label{lem:asymptoticsGt}
\end{lemma}
\begin{proof}
	Using the explicit expressions for $P^{\rm Dir}_{t-s}(x,y)$ from \eqref{eq:PtDir} and $dP_s^{\rm Dir}(y,0)$ in \eqref{eq:defdP}, we can compute the integral \eqref{eq:defGt} and obtain 
	$$ G_t(s, x) = \frac{2e^{\frac{-x^2}{t}} \Big(\frac{t(t-s)}{s}\Big( 1-e^{\frac{-sx^2}{t(t-s)}}\Big)+2x^2 \Big) }{\pi^{3/2}t^{5/2}\sqrt{s(t-s)}}.$$ 
	Dividing by $(dP^{\rm Dir}_t(x,0))^2$, we get 
	$$ \frac{G_t(s, x)}{(dP^{\rm Dir}_t(x,0))^2}  = \frac{\sqrt{t}\Big(\frac{t(t-s)}{s}\Big( 1-e^{\frac{-sx^2}{t(t-s)}}\Big)+2x^2\Big)}{4\sqrt{\pi}\sqrt{s(t-s)}}.$$
	Using the bound $1-e^{-x}\leq x$, and simplifying the resulting expression,  we obtain that 
	$$ \frac{G_t(s, x)}{(dP^{\rm Dir}_t(x,0))^2}  \leq  \frac{3}{4\sqrt{\pi}} \frac{\sqrt{t}}{\sqrt{s(t-s)}}.$$
\end{proof}

Let us now turn to the proof of Proposition \ref{prop:existenceuniqueness}.

\begin{proof}[Proof of Proposition \ref{prop:existenceuniqueness}]
	Fix a terminal time $T>0$ and consider the Banach space $\mathcal B$ of adapted processes $(\mathcal{Z}_t)$ satisfying 
	\[ \Vert \mathcal{Z} \Vert_{\mathcal B}^2 := \sup_{\substack{
			x\in \mathbb R_+ \\ t\in (0,T)
	}} \left\lbrace t^2\; \mathbb E\left[ \mathcal Z_t(x)^2 \right] \right\rbrace <\infty.\]
	We define a sequence of processes defined for $t\leq T, x\in \mathbb R_+$, by 
	\begin{align*} 
		\mathcal{U}_0(t,x)  &:= dP^{\rm Dir}_t(x,0).\\
		\mathcal{U}_{n+1}(t,x) &:= \int_0^t \int_{\mathbb R_+} P^{\rm Dir}_{t-s}(x,y)\mathcal{U}_n(s,y)\xi(s,y)dyds.
	\end{align*}
	This implies that, if $S_n=\sum_{k=0}^n \mathcal{U}_n$,  then
	\[ S_{n+1}(t,x) = dP_t^{\rm Dir}(x,0)+ \int_0^t \int_{\mathbb R_+} P^{\rm Dir}_{t-s}(x,y)S_n(s,y)\xi(s,y)dyds.\]
	In order to show the existence of the solution to {\eqref{eq:Duhamelform} with $\mathcal Z_{\rm ini}=-2\delta_0'$}, it suffices to show that the series $S_n$ converges in the space $\mathcal B$, and for that we will show that $\sum \Vert \mathcal{U}_n\Vert_{\mathcal B}$ converges. Regarding uniqueness, it follows from the same argument as in \cite[Proposition 4.2]{parekh2017kpz}.
	In order to estimate $\Vert \mathcal{U}_n\Vert_{\mathcal B}$,  we introduce 
	\[ f_n(s) =  \sup_{r\in[0,s],y \in \R_+
	} \left\lbrace \frac{\mathbb E\left[ \mathcal{U}_n(r,y)^2 \right]}{dP_r^{\rm Dir}(y,0)^2}\right\rbrace.\]
	By It\^o isometry, we have, for any $t\in[0,T]$
	\[ \mathbb E\left[ \mathcal{U}_{n+1}(t,x)^2 \right] = \int_0^t ds  \int_{\mathbb R_+} dy \left( P^{\rm Dir}_{t-s}(x,y) \right)^2 \mathbb E \left[ \mathcal{U}_n(s,y)^2\right].\]
	Thus, we may write, recalling the definition of $G_t(s, x)$ in \eqref{eq:defGt}, 	\begin{align*}
		\mathbb E\left[ \mathcal{U}_{n+1}(t,x)^2 \right] &\leq  \int_0^t ds \int_{\mathbb R_+}dy  \left( P^{\rm Dir}_{t-s}(x,y) \right)^2  (dP_s^{\rm Dir}(y,0))^2   f_n(s) \\
		&\leq \int_0^t ds\, G_t(s, x) f_n(s).
	\end{align*}   
	Using Lemma \ref{lem:asymptoticsGt}, we obtain that 
	\begin{equation}
		\mathbb E\left[ \mathcal{U}_{n+1}(t,x)^2 \right]  \leq  C \left( dP_t^{\rm Dir}(x,0)\right)^2 \sqrt{t} \int_{0}^t ds\frac{f_n(s)}{\sqrt{s(t-s)}}.
		\label{eq:recurrencebound}
	\end{equation} 
	Dividing both sides of $\eqref{eq:recurrencebound}$ by $( dP_t^{\rm Dir}(x,0))^2$,  we obtain that for $t\in {[0,T]}$, 
	\[ f_{n+1}(t) \leq C \sup_{r\leq t}\sqrt{r}\int_0^r ds \frac{f_n(s)}{\sqrt{s(r-s)}}.\]
	By a simple change of variable and using the monotonicity of $f_n$, we get that for any $r\in[0,t]$,
	\begin{align*}
	 \int_0^r ds \frac{f_n(s)}{\sqrt{s(r-s)}}=\int_0^1 du\frac{f_n(ru)}{\sqrt{u(1-u)}}
	 \leq \int_0^1 du\frac{f_n(tu)}{\sqrt{u(1-u)}}
	 =\int_0^t ds \frac{f_n(t)}{\sqrt{s(t-s)}}.
	\end{align*}
	We conclude that 
	\begin{equation}
	 f_{n+1}(t) \leq C \sqrt{t}\int_0^tds \frac{f_n(s)}{\sqrt{s(t-s)}}.
	\end{equation}
	Iterating this inequality, we get 
	\begin{align*}
		f_{n+2}(t) & \leq C\sqrt t \int_0^t dr \sqrt{r} \int_0^r ds  \frac{f_n(s)}{\sqrt{r(t-r)}\sqrt{s(r-s)}} \\
		& = C\sqrt t \int_0^t ds \frac{f_n(s)}{\sqrt{s}} \int_s^t \frac{dr}{\sqrt{t-r}\sqrt{r-s}} 
		 = C \sqrt{t} \int_0^t ds \frac{f_n(s)}{\sqrt{s}},
	\end{align*} by exchanging the integration order. We deduce by induction that $f_n(t) \leq C^n t^{n/2}/(\lfloor n/2 \rfloor)!$. 	
	Hence, we have obtained that 
	\[ \Vert \mathcal{U}_n\Vert_{\mathcal B}^2 \leq C  \sup_{\substack{
			x\in \mathbb R_+ \\ s\in (0,T)
	}} \left\lbrace \frac{ \mathbb E\left[ \mathcal{U}_n(s,x)^2 \right]}{dP_s^{\rm Dir}(x,0)^2} \right\rbrace ={C}f_n(T) \leq   \frac{{C^{n+1}} T^{n/2}}{(\lfloor n/2 \rfloor)!},\]
	where in the first inequality we have simply used that $dP_s^{\rm Dir}(x,0)\leq C s^{-1}$ (which is easy to check using the explicit expression \eqref{eq:defdP}).  This shows that $\sum_{n=0}^{+\infty} \Vert \mathcal{U}_n \Vert_{\mathcal B}<\infty$ so that the series $\sum_{n=0}^{+\infty} \mathcal{U}_n $ exists in the space $\mathcal B$, and it concludes the proof. 
\end{proof}

\section{Intermediate technical results} 
\label{sec:heatkernel}

In both proofs of Theorems \ref{thm:conv} and \ref{t:cvneareq}, we will need several estimates which we collect in this section. We start by defining the discrete heat kernel with Robin boundary condition, denoted below by $\p_t^R$, and $\p_t^\eps$ when the boundary condition diverges. In Section \ref{ssec:heat1}, we collect the usual decay estimates on $\p^\eps$ and provide a new proof of the ``key cancellation'' property. In Section \ref{sec:moment} we use exact computations to give a long time estimate for $\p^R$ and a Hölder estimate for the convolution of $\p^\eps $ with the empty initial condition. In Section \ref{ssec:holder} we prove a H\"older property for $\p^\eps$ via random walk estimates. In Section \ref{ssec:killing}, we bound the killing probability via exact computations. Finally in Section \ref{ssec:useful} we collect two essential technical lemmas which deal with the microscopic Hopf-Cole process $Z_t(k)$.

\begin{definition}\label{d:heatkernel}
	Let $\mu<1$. The discrete heat kernel with Robin boundary condition (and parameter $\mu$), is defined as the solution to the following: for any $k,\ell \in \N$
	\begin{equation*}
		\partial_t  \p_t^R(k,\ell) = \frac 1 2 \Delta_\ell  \p^R_t(k,\ell), \qquad\p_0^R(k,\ell) = \mathbf{1}_{k=\ell}, \qquad \p_t^R(-1,\ell)=\mu\; \p_t^R(0,\ell),
	\end{equation*} 
	where $\Delta_\ell$ denotes the discrete Laplacian acting on functions of the variable $\ell$. 
\end{definition}

Since we assume $\mu<1$, $\p_t^R(k,\ell)$ corresponds to the transition probability for a continuous time random walk on $\N$ which behaves as the symmetric simple random walk on positive integers, while at $0$, after an exponentially distributed waiting time with mean $1$, it jumps to $1$ with probability $\frac 1 2$, it stays at $0$ with probability $\frac{\mu}{2}$, and it is killed with probability $\frac{1-\mu}{2}$. This gives us a representation of $\p^R$ in terms of the transition probabilities $(\p_n)$ of the underlying discrete random walk (which moves similarly but at integer times):
\begin{equation}\label{e:representationdiscretetime}
	\p^R_t(k,\ell)=\sum_{n=0}^\infty e^{-t}\frac{t^n}{n!}\p_n(k,\ell).
\end{equation}
Moreover, we have the following representation for $\p_t^R$, in terms of the kernel $p_t$ for the continuous time symmetric simple random walk on $\Z$, see \cite[Section 4.1]{corwin2016open}:
\begin{equation}
	\p_t^R(k,\ell) = p_t(k-\ell)+ \mu p_t(k+\ell+1)+ (1-\mu^{-2}) \sum_{j=2}^{+\infty} \mu^j p_t(k+\ell+j).
	\label{eq:imagesmethodrepresentation}
\end{equation} 
In particular, $\p_t^R(k,\ell)$ is increasing in $\mu$.

\begin{definition}\label{d:epsheatkernel} Let
	$\{\p^\eps_t(k,\ell)\}_{k,\ell\in\N}$ denote the discrete heat kernel with Robin boundary condition and parameter $\mu=e^{-\eps}$.
\end{definition}

\subsection{First properties on the heat kernel} \label{ssec:heat1}

\begin{lemma}[Heat kernel bounds] \label{lem:bounds} Fix $T>0$. We have the following estimates. \begin{enumerate}
		\item[(i)]  For any $a_1 \ge 0$, $a_2 \ge 0$, there exists $C=C(\oo{a_1,a_2},T)\in\R_+$ such that, for all $t\leq \eps^{-4}T$, $k\in\N$, 
		\begin{align}
			\sum_{\ell=0}^\infty \p_t^\eps(k,\ell){e^{a_1|k-\ell|(1\wedge t^{-1/2})}e^{a_2\eps^2 \ell}}\leq Ce^{{a_2}\eps^2 k}\label{e:hkbound1}.
		\end{align}
		\item[(ii)] There exists $C=C(T)\in\R_+$ such that,  for all $t\leq \eps^{-4}T$, $k,\ell,j\in\N$, for all $v\in[0,1]$, 	
		\begin{equation}\label{e:hkbound2}
			|\p_t^\eps(k,j)-\p_t^\eps(\ell,j)|\leq C\left(1\wedge\frac{1}{\sqrt{t^{1+v}}}\right)|k-\ell|^v.
		\end{equation}
		Moreover, for all $b>0$, there exists $C=C(T,b)\in\R_+$ such that for all $t\leq \eps^{-4}T$, $k,\ell\in\N$, for all $v\in[0,1]$,
		\begin{equation}\label{e:hkbound4}
			|\nabla^\pm\p_t^\eps(k,\ell)|\leq C\left(1\wedge\frac{1}{\sqrt{t^{1+v}}}\right)e^{-b|k-\ell|(1\wedge t^{-1/2})},
		\end{equation} 
		and consequently, for all $a\geq 0$, there exists $C=C(T,a)\in\R_+$ such that, for all $k\in\N$,
		\begin{equation}\label{e:hkbound5}
			\sum_{\ell=0}^\infty \left|\nabla^\pm\p_t^\eps(k,\ell)\right|e^{a\eps^2 \ell}e^{a|k-\ell|(1\wedge t^{-1/2})}\leq Ce^{a\eps^2k}t^{-1/2}.
		\end{equation}
		\item[(iii)] For all $t\geq s\geq 0$, $k,\ell\in\N$,
		\begin{equation}\label{e:hkbound3}
			\p_s^\eps(k,\ell)\leq e^{t-s}\p_t^\eps(k,\ell).
		\end{equation}
	\end{enumerate}

\end{lemma}

\begin{proof}
	A number of these bounds follow from those established for instance in \cite{parekh2017kpz}. Therein, the author considers the discrete heat kernel with Robin boundary condition with parameter $\mu_A=1-A\eps^2$ (note that $\eps$ for us corresponds to $\sqrt{\eps}$ therein). In particular, the monotonicity in $\mu$ of $\p^R$ allows us to use directly the bounds established in \cite{parekh2017kpz}. More precisely: 
	\begin{itemize}
		\item 
		\eqref{e:hkbound1} follows from Corollary 3.3 in \cite{parekh2017kpz};
		\item 
		\eqref{e:hkbound2} and \eqref{e:hkbound4} follow from similar bounds that hold for the standard heat kernel $p_t$ on the whole line $\Z$, and from \eqref{eq:imagesmethodrepresentation}, as in the proof of Proposition 3.2 in \cite{parekh2017kpz}. In fact, the monotonicity in $\mu$ of \eqref{eq:imagesmethodrepresentation} implies a monotonicity in the upper bounds used in the proof of \cite{parekh2017kpz}, so that we can use the estimates therein as an upper bound;

		\item \eqref{e:hkbound3} is an immediate consequence of \eqref{e:representationdiscretetime}. \end{itemize}
\end{proof}

\begin{lemma}[Key cancellation and consequences] \label{lem:cons} We have the following estimates:
	\begin{enumerate}
		\item[(i)] For any $k,\ell\in\N$,
		\begin{equation}\label{e:keycancel}
			\sum_{j=0}^\infty\int_0^\infty \nabla^+\p^\eps_t(k,j)\nabla^+\p^\eps_t(\ell,j)dt=\mathbf{1}_{\ell=k}.
		\end{equation}
		\item[(ii)] For any $a>0$, there exist $\eps_0=\eps_0(a,T)>0,c  =c(a,T)\in(0,1)$ such that for all $\eps<\eps_0$, $k\in\N^*$,
		\begin{equation}\label{e:hkboundcancel1}
			\sum_{j=1}^\infty\int_0^{\eps^{-4}T}\left|\nabla^+\p^\eps_r(k,j)\nabla^-\p^\eps_r(k,j)\right|e^{a\eps^2|k-j|}dr\leq c.
		\end{equation}
		\item[(iii)]   For any $a>0$, there exist $\eps_0=\eps_0(a,T)>0,C  =C(a,T)\in\R_+$ such that for all $\eps<\eps_0,k\in\N^*,t\leq T$,
		\begin{equation}\label{e:hkboundcancel2}
			\sum_{j=1}^\infty\int_0^{\eps^{-4}t}\left|\nabla^+\p^\eps_r(k,j)\nabla^-\p^\eps_r(k,j)\right|e^{a\eps^2|k-j|}\frac{1}{\sqrt{\eps^{-4}t-r}}dr\leq C\eps^2.
		\end{equation}
	\end{enumerate}
\end{lemma}
\begin{proof}
	The last two bounds \eqref{e:hkboundcancel1} and \eqref{e:hkboundcancel2} follow from the first ``key cancellation'' identity \eqref{e:keycancel}, together with the previous bounds \eqref{e:hkbound4} and \eqref{e:hkbound5}, by the same arguments as in \cite[Proof of Corollary 5.4]{corwin2016open}.
	
	Let us show \eqref{e:keycancel}, which can be proved in a more elementary way than is done in \cite{corwin2016open} (where it is Proposition 5.1). Indeed, note that $\sum_{j=0}^\infty \p^\eps_t(k,j)\p^\eps_t(\ell,j)=\p^\eps_{2t}(k,\ell)$ by symmetry of $\p^\eps_t$ and the Markov property of the random walk described after Definition~\ref{d:heatkernel}. Moreover, $\int_0^\infty\p^\eps_{2t}(k,\ell)dt=\frac{1}{2}G(k,\ell)$ is the Green function associated with that random walk, \textit{i.e.}~the expected number of times the random walk started from $k$ goes through $\ell$ before being killed. Elementary computations therefore yield
	\begin{multline}
		{2}\sum_{j=0}^\infty\int_0^\infty \nabla^+\p^\eps_t(k,j)\nabla^+\p^\eps_t(\ell,j)dt \\ =G(k+1,\ell+1)+G(k,\ell)-G(k+1,\ell)-G(k,\ell+1).
	\end{multline}
	By the interpretation of $G$ in terms of number of visits, it is clear that $G(k+1,\ell+1)=G(k,\ell+1)$ if $k>\ell$ and $G(k+1,\ell)=G(k,\ell)$ if $k\geq \ell$. Moreover, by first-step analysis, $G(k+1,k+1)=1+\frac{1}{2}\left[G(k,k+1)+G(k+2,k+1)\right]=1+\frac{1}{2}\left[G(k,k+1)+G(k+1,k+1)\right]$, and therefore $G(k+1,k+1)=2+G(k,k+1)$. Then \eqref{e:keycancel} follows.
\end{proof}

\subsection{\oo{Long-time and proto-Hölder} estimates with exact computations}\label{sec:moment}

In this section, we give a long time estimate for the heat kernel $\p_t^R$, as well as a Hölder bound for the convolution of $p^\eps_t$ with the empty initial condition $Z_0(k)=\mu^k$.

\begin{proposition}\label{prop:boundinitialcondition}
	There exists a constant $C>0$  such that for any  $\eps\in (0,1)$, $t>0$, $\alpha\in [0,\mm{1})$   and $k,\ell \in \N$,  we have
	\begin{equation}
		\p_t^R(k,\ell) \leqslant  \frac{C}{\sqrt{t}},  \qquad \text{ for any $\mu \in (0,1)$}
		\label{eq:boundheatkernel}
	\end{equation}
	and 
	\begin{align}
		\big|  (\p_t^\varepsilon \ast \varepsilon^{-2}Z_0)(k)-(\p_t^\varepsilon \ast \varepsilon^{-2}Z_0)(\ell)\big|  \leqslant C \left( \eps^2\vert k-\ell\vert\right)^{\alpha}(\eps^4 t)^{-1-\alpha/2}. \vphantom{\bigg(}
		\label{eq:boundIC2}
	\end{align}
	
\end{proposition}

\begin{remark} Note that the constant $C$ in \eqref{eq:boundheatkernel} is universal (\mm{although likely not optimal}), and does not depend on $t$ nor on the terminal time $T$.  This estimate is called \emph{long-time estimate} in \cite{parekh2017kpz}, see Proposition 3.6 therein. Here, we provide a different proof and obtain an even better estimate.
\end{remark}

\begin{proof}
	Let us first establish explicit formulas for the quantities to bound in \eqref{eq:boundheatkernel} and \eqref{eq:boundIC2}.  For $t\geqslant 0$ and $k\in \Z$,  we have the integral representation of the heat kernel 
	\begin{equation}
		p_t(k) = \frac{1}{2\I\pi}\oint  e^{\frac{1}{2}\left(\xi+\xi^{-1}-2\right) t} \;\xi^k \frac{d \xi}{\xi},
		\label{eq:heatkernelintegral}
	\end{equation}
	where the contour is a positively oriented circle around $0$. 
	From \eqref{eq:heatkernelintegral} and \eqref{eq:imagesmethodrepresentation}, we deduce 
	\begin{equation}
		\p_t^R(k,\ell) = \frac{1}{2\I\pi}\oint   e^{\frac{1}{2}\left(\xi+\xi^{-1}-2\right) t} \xi^k \left( \xi^{-\ell}+ \xi^{\ell+1} \frac{\mu-\xi}{1-\mu \xi}
		\right) \frac{d \xi}{\xi},
		\label{eq:heatkernelRobinintegral}
	\end{equation}
	where the contour is a positively oriented circle around $0$ with radius smaller than $\mu^{-1}$. Hence, we have that 
	\begin{equation}
		\sum_{\ell\ge 0} \p_t^R(k,\ell) \mu^\ell = \mathbb E\left[ Z_t(k) \right]= \frac{1}{2\I\pi}\oint   e^{\frac{1}{2}\left(\xi+\xi^{-1}-2\right) t} \xi^k \frac{(1-\mu^2)(1-\xi^2)}{(1-\mu/\xi)(1-\mu\xi)^2}  \frac{d \xi}{\xi}
		\label{eq:formulafirstmoment}
	\end{equation}
	where the contour is a positively oriented circle around $0$ with radius comprised between $\mu$ and $\mu^{-1}$, and
	\begin{equation}
		\sum_{j\ge 0} \p_t^R(k,j) \mu^j -   \sum_{j\ge 0} \p_t^R(\ell,j) \mu^j = \frac{1}{2\I\pi}\oint   e^{\frac{1}{2}\left(\xi+\xi^{-1}-2\right) t} (\xi^k-\xi^\ell) \frac{(1-\mu^2)(1-\xi^2)}{(1-\mu/\xi)(1-\mu\xi)^2}  \frac{d \xi}{\xi}.
		\label{eq:differenceIC}
	\end{equation}
	Now we estimate the integrals above. Since $\mu\in (0,1)$, we may assume that the contour (in each formula above) is a circle of radius $1$. Now, since $\xi$ has modulus $1$, we have that $\vert \xi^{k-\ell} \vert =\vert \xi^{k+\ell+1}\vert =1$ and for any $\mu\in (0,1)$, $\vert \mu - \xi\vert =\vert \mu -\bar \xi\vert =\vert \mu\xi-1\vert $ so that $\big| \frac{\mu-\xi}{1-\mu \xi} \big| =1$. 
	Using the change of variables $\xi=e^{\I\theta }$ in \eqref{eq:heatkernelRobinintegral}, we get 
	\begin{equation*}
		\p_t^R(k,\ell) \leqslant \frac{2}{2\pi}\int_{-\pi}^{\pi}  e^{\frac{t}{2}\Re\left[e^{\I \theta}+e^{-\I\theta}-2\right]} d\theta. 
	\end{equation*}
	Then, we use the estimate $\cos(x)-1 \leqslant \frac{-x^2}{5}$, valid for $x\in (-\pi, \pi)$, so that 
	\begin{equation} \Re[e^{\I \theta}+e^{-\I\theta}-2 ] =2\cos( \theta)-2 \leqslant \frac{-2 \theta^2}{5}.
		\label{eq:eigenvalue}
	\end{equation}
	Hence, setting $t=\eps^{-4}T$, we obtain 
	\begin{equation*}
		\p_t^R(k,\ell) \leqslant \frac{\eps^2}{\pi} \int_{-\pi}^{\pi} e^{\frac{-\theta^2 t}{5}} d\theta.
	\end{equation*}
	Finally, using the change of variables $\theta=\tilde \theta t^{-1/2}$, 
	\begin{equation*}
		\p_t^R(k,\ell) \leqslant \frac{1}{\pi \sqrt{t}}\int_{\mathbb R} e^{\frac{-\tilde \theta^2}{5}}d\tilde \theta
	\end{equation*} 
	which proves \eqref{eq:boundheatkernel} with $C=\frac{1}{\pi}\int_{\mathbb R} e^{\frac{-z^2}{5}} d z$. 
	
	Now we assume that $\mu=e^{-\eps}$. 
	Using the change of variables $\xi=e^{\I\theta \eps^{2}}$, and setting  $\eps^{2}k=X$, \oo{$\eps^{2}\ell=Y$, \eqref{eq:differenceIC} }becomes 
	\begin{multline}
		\eps^{-2}  \sum_{j\ge 0} \oo{(\p_t^\eps(k,j)-\p^\eps_t(\ell,j))} \mu^j =\\ \frac{1}{2\pi}\int_{-\eps^{-2}\pi}^{\eps^{-2}\pi}   e^{\frac{1}{2}\left(e^{\I\eps^2 \theta}+e^{-\I\eps^2\theta}-2\right) t}\oo{ \left(e^{\I \theta  X}-e^{\I\theta Y}\right) }\frac{(1-e^{-2\eps})(1-e^{2\I\eps^2\theta})}{(1-e^{-\eps-\I\eps^2\theta})(1-e^{-\eps+\I\eps^2\theta})^2}  d\theta. 
		\label{eq:expectationsolution}
	\end{multline}
	Using $ x-\frac{x^2}{2} \leqslant 1-e^{-x} \leqslant x $ for $x>0$, and $0 \leqslant 1-\cos(x) \leqslant \frac{x^2}{2}$ for $x\in (-\pi, \pi)$,  we have 
	\begin{equation*} \left\vert \frac{(1-e^{-2\eps})(1-e^{2\I\eps^2\theta})}{(1-e^{-\eps-\I\eps^2\theta})(1-e^{-\eps+\I\eps^2\theta})^2}  \right\vert = \frac{(1-e^{-2\eps})\sqrt{2(1-\cos(2\eps^2\theta))}}{\left(1+e^{-2\eps}-2 e^{-\eps} \cos(\eps^2\theta)\right)^{3/2}}  \leqslant \frac{4\eps^3\vert \theta \vert}{\left(\eps-\eps^2/2\right)^3}= \frac{4\vert \theta\vert }{(1-\eps/2)^3}.
	\end{equation*}
	Thus, using additionally the estimate \eqref{eq:eigenvalue}  and setting $t=\eps^{-4}\oo{\tau}$,  we obtain the bound 
	\begin{equation*}
		\left|\eps^{-2} \sum_{j\ge 0} \oo{(\p_t^\eps(k,j)-\p^\eps_t(\ell,j))} \mu^j \right| \leqslant  \frac{1}{2\pi} \int_{-\eps^{-2}\pi}^{\eps^{-2}\pi} e^{\frac{-\theta^2 \oo{\tau}}{5}} \frac{4\vert\theta\vert }{(1-\eps/2)^3} \oo{ \left|e^{\I \theta  X}-e^{\I\theta Y}\right| }d\theta.
	\end{equation*}
	The change of variables $\theta=\tilde \theta \oo{\tau}^{-1/2}$ now yields (for  $\eps\in (0,1)$) 
	\begin{equation}
		\left|\eps^{-2} \sum_{j\ge 0} \oo{(\p_t^\eps(k,j)-\p^\eps_t(\ell,j))} \mu^j \right|  \leqslant \frac{1}{\oo{\tau}} \frac{16}{\pi}\int_{\R} \vert \tilde \theta \vert e^{\frac{-\tilde \theta^2}{5}}\left\vert e^{\I X\tilde \theta \oo{\tau}^{-1/2} }- e^{\I Y\tilde \theta \oo{\tau}^{-1/2} }\right\vert d\tilde \theta.\label{eq:bounddifference}
	\end{equation}
	Using the change of variables $u=2\theta$ and choosing a constant $C$ s.t.~$\frac{16}{\pi}\vert \tilde \theta \vert e^{\frac{-\tilde \theta^2}{5}} \leqslant C e^{\frac{-u^2}{2}}$,
	\begin{align*}
		\left|\eps^{-2} \sum_{j\ge 0} \oo{(\p_t^\eps(k,j)-\p^\eps_t(\ell,j))} \mu^j \right| &\leqslant \frac{C}{\oo{\tau}} \int_{\R}   e^{\frac{-u^2}{2}}  \left\vert e^{2\I X u \oo{\tau}^{-1/2} }- e^{2\I Yu \oo{\tau}^{-1/2} }\right\vert du ,\\ 
		&\leqslant \frac{C}{\oo{\tau}} \int_{\R}   e^{\frac{-u^2}{2}}  \left\vert 1- e^{2\I (X-Y)u \oo{\tau}^{-1/2} }\right\vert du, \\ 
		&{\leqslant\frac{4C}{\oo{\tau}}\int_{\R_+}e^{-u^2/2}\min\left(u|X-Y|\oo{\tau}^{-1/2},1\right)du}\\
		&\leqslant{\frac{4C}{\oo{\tau}}\min\left\{|X-Y|\oo{\tau}^{-1/2}\int_{\R_+}ue^{-u^2/2}du,\int_{\R_+}e^{-u^2/2}du\right\}}\\
		&\leqslant {\frac{C'}{\oo{\tau}} \min\big\{ \vert X-Y\vert  \oo{\tau}^{-1/2}, 1 \big\}},
	\end{align*}
	{where $C'=4C\sqrt{\pi/2}$.} Thus, \eqref{eq:boundIC2} holds for any $\alpha\in (0,1]$. 
\end{proof} 

\subsection{H\"older property for the heat kernels} \label{ssec:holder}
	
	\mm{We prove here a H\"older property for the heat kernels which will be used to get suitable H\"older estimates for the scaled  process $\mathcal{Z}_t^\varepsilon$ defined in \eqref{eq:scaled1} (see Section \ref{subs:neareq} below). Recall that  $p_t$ stands for the kernel of the continuous time symmetric simple random walk on $\Z$. }
	
\begin{lemma}\label{l:Holder-heat-kernel}Fix $0<\delta<T$. There exists a constant $C=C(\delta,T)$ s.t. for $\eps$ small enough, 
\begin{enumerate}
 \item  for any $\tau\in[\delta,T], x,y\in\R$, letting $t=\eps^{-4}\tau$, $k=\lfloor\eps^{-2}x\rfloor,\ell=\lfloor\eps^{-2}y\rfloor$,
 \begin{equation}
	\sum_{j\in \Z} \vert p_t(k-j)-p_t(\ell-j) \vert \leqslant \frac{C\vert x-y\vert}{\sqrt{\tau}}.
	\label{eq:diffheatkernel}	
	\end{equation}
	\item  for any $\tau\in[\delta,T], x,y\in\R_+$, letting $t=\eps^{-4}\tau$, $k=\lfloor\eps^{-2}x\rfloor,\ell=\lfloor\eps^{-2}y\rfloor$
	\begin{equation}
	 \sum_{j\in \N} \vert \p^\eps_t(k,j)-\p^\eps_t(\ell,j) \vert \leqslant \frac{C\vert x-y\vert}{\sqrt{\tau}}.
	\label{eq:diffheatkerneleps}
	\end{equation}

\end{enumerate}

\end{lemma}

\begin{proof}
 Let us first prove \eqref{eq:diffheatkernel}. Let $X_t^{(k)}$ and $X_t^{(\ell)}$  denote continuous time simple random walks starting from $k$ and $\ell$ respectively. We may write for any $k,\ell,j\in\Z$
\begin{equation}
	p_t(k-j)-p_t(\ell-j) = \mathbb E^{(k,\ell)}\left[ \mathds{1}_{X_t^{(k)}=j} - \mathds{1}_{X_t^{(\ell)}=j} \right],
	\label{eq:diffheatkernel2}
\end{equation}
where the expectation $\mathbb E^{(k,\ell)}$ is associated to a probability distribution $\mathbb P^{(k,\ell)}$ of the couple of random walks $(X_t^{(k)}, X_t^{(\ell)})$. The identity \eqref{eq:diffheatkernel2} is true as long as the marginal distribution of each walk under $\mathbb P^{(k,\ell)}$ is the continuous time random walk with kernel $p_t$, but the two random walks are not necessarily independent. Let us first choose $\mathbb P^{(k,\ell)}=\mathbb P_{\mathbf{Coal}}^{(k,\ell)}$ to be the probability distribution of two coalescing random walks with kernel $p_t$ starting from $k$ and $\ell$ respectively.  

We then obtain 
\begin{equation}
\sum_{j=0}^{+\infty} \vert p_t(k-j)-p_t(\ell-j) \vert \leqslant \mathbb E_{\mathbf{Coal}}^{(k,\ell)} \left[\sum_{j=0}^{+\infty} \left\vert \mathds{1}_{X_t^{(k)}=j} - \mathds{1}_{X_t^{(\ell)}=j} \right\vert  \right].
\end{equation}
The sum inside the expectation is non-zero only if the two walks have not coalesced before time $t$. 
Let us denote $\mathbf{Coal}_t=\{X_t^{(k)}=X_t^{(\ell)}\}$ the coalescing event, and by $\mathbf{Coal}^{\complement}_t$ its complement. We have
\begin{align*}
\sum_{j=0}^{+\infty} \vert p_t(k-j)-p_t(\ell-j) \vert& \leqslant \mathbb E_{\mathbf{Coal}}^{(k,\ell)} \left[\mathbf{1}_{\mathbf{Coal}^{\complement}_t}\sum_{j=0}^{+\infty} \left( \mathds{1}_{X_t^{(k)}=j} + \mathds{1}_{X_t^{(\ell)}=j} \right)  \right]\\
&\leqslant 2\mathbb P_{\mathbf{Coal}}^{(k,\ell)}	\left( \mathbf{Coal}^{\complement}_t \right).
\end{align*}
Let us denote by $\mathbf{Intersect}_t^{\complement}\mm{=\{\forall\,s\in [0,t],\; X_s^{(k)} \neq X_s^{(\ell)}\}}$ the event that two walks do not intersect on $[0,t]$. 
We have \[\mathbb P_{\mathbf{Coal}}^{(k,\ell)}	\left( \mathbf{Coal}^{\complement}_t \right) =  \mathbb P_{\mathbf{Ind}}^{(k,\ell)}	\left( \mathbf{Intersect}_t^{\complement} \right),\] where $\mathbb P_{\mathbf{Ind}}^{(k,\ell)}$ is the probability distribution of two independent continuous time random walks with kernel $p_t$ starting from $k$ and $\ell$. At this point, we have arrived at 
\begin{equation}
	\sum_{j=0}^{+\infty} \vert p_t(k-j)-p_t(\ell-j) \vert \leqslant 2 \mathbb P_{\mathbf{Ind}}^{(k,\ell)}	\left( \mathbf{Intersect}_t^{\complement} \right). 
\end{equation}
We now estimate the latter probability. We have
\begin{align*}
 \mathbb P_{\mathbf{Ind}}^{(k,\ell)}	\left( \mathbf{Intersect}_t^{\complement} \right)&\leqslant \mathbb P_{\mathbf{Ind}}^{(k,\ell)}	\left(\min_{0\leqslant s\leqslant  t}|X^{(k)}_s-X^{(\ell)}_s|\leq 0\right)\\
 &\leqslant \mathbb P\left(\max_{0\leq s\leq t}\lbrace  S_s\rbrace < \vert k-\ell\vert  \right),
\end{align*}
where $S$ is the simple symmetric continuous time random walk (jumping by $+1$ and $-1$ both with rate $1$) started from $0$.

 By the reflection principle for simple random walks, for any $n\in\N\setminus\{0\}$
 \begin{equation*}
  \mathbb P\left(\max_{0\leq s\leq t}\lbrace  S_s\rbrace =n \right)=\mathbb P(S_t=n)+\mathbb P(S_t=n+1).
 \end{equation*}
Therefore,
\begin{align*}
 \mathbb P\left( \max_{0\leq s\leq t}\lbrace  S_s\rbrace < n \right) &= \mathbb P\left(\max_{0\leq s\leq t}\lbrace  S_s\rbrace =0\right)+\mathbb P(0<S_t<n)+\mathbb P(1<S_t<n+1)\\
 &\leqslant
  2\P(0\leqslant S_t<n).
\end{align*}
In particular
\begin{align*}
 \sum_{j=0}^{+\infty} \vert p_t(k-j)-p_t(\ell-j) \vert \leqslant 4\P\left(0\leqslant \frac{S_t}{\sqrt{t}}<\frac{|k-\ell|}{\sqrt{t}}\right).
\end{align*}
Then, \eqref{eq:diffheatkernel}  follows from the Central Limit Theorem and Dini's Theorem.

 We now show that \eqref{eq:diffheatkernel} implies \eqref{eq:diffheatkerneleps}. Recall \eqref{eq:imagesmethodrepresentation}, which yields
 \begin{multline}
\sum_{j=0}^{+\infty} \vert \p^\eps_t(k,j)-\p^\eps_t(\ell,j) \vert \leqslant  
\sum_{j=0}^{+\infty} \vert p_t(k-j)-p_t(\ell-j)\vert  + \mu \sum_{j=0}^{+\infty} \vert p_t(k+j+1)-p_t(\ell+j+1) \vert \\+ \vert 1-\mu^{-2}\vert\sum_{i=2}^{\infty} \mu^i \sum_{j=0}^{+\infty} \vert p_t(k+j+i)-p_t(\ell+j+i) \vert.
\end{multline}
In the right hand side, we may replace summations over $\N$ by summations over $\mathbb Z$. We get
\begin{multline}
	\sum_{j=0}^{+\infty} \vert \p^\eps_t(k,j)-\p^\eps_t(\ell,j) \vert \leqslant   
	\sum_{j\in \Z} \vert p_t(k-j)-p_t(\ell-j) \vert +  \mu \sum_{j\in\Z} \vert p_t(k+j+1)-p_t(\ell+j+1) \vert \\+ \vert 1-\mu^{-2}\vert \sum_{i=2}^{\infty} \mu^i \sum_{j\in \Z} \vert p_t(k+j+i)-p_t(\ell+j+i) \vert,
\end{multline}
so that reindexing and using \eqref{eq:diffheatkernel} in each of the three terms, with $k=\lfloor\eps^{-2}x\rfloor,$ and $\ell=\lfloor\eps^{-2}y\rfloor$, we get 
\begin{align*}
	\sum_{j=0}^{+\infty} \vert \p^\eps_t(k,j)-\p^\eps_t(\ell,j) \vert &\leqslant  
	\frac{C\vert x-y\vert}{\sqrt{\tau}} +  \mu \frac{C\vert x-y\vert}{\sqrt{\tau}} + \vert 1-\mu^{-2}\vert \sum_{k=2}^{\infty} \mu^k \frac{C\vert x-y\vert}{\sqrt{\tau}}\\
	&\leqslant \frac{C\vert x-y\vert}{\sqrt{\tau}} \left(1+\mu +\frac{1- \mu^2}{1-\mu} \right)\\
	&\leqslant \frac{5C\vert x-y\vert}{\sqrt{\tau}}
\end{align*}
for $\varepsilon$ small enough.  
\end{proof}

\subsection{Estimate on the killing probability} \label{ssec:killing}
Recall the interpretation of the heat kernel in terms of random walks killed at the boundary, presented at the beginning of Section~\ref{sec:heatkernel}. We estimate here the probability that the random walk started from $k$ is killed before time $t$.
\begin{lemma}\label{lem:killing}
For any $a>0$ and $T>0$, there exists $C=C(a,T)\in\R_+$ such that for any $k\in\N$ and  $t\in[1,\eps^{-4}T]$,
 \begin{equation}\label{eq:killingestimate}
  \bigg|1-\sum_{j\in\N}\p^\eps_t(k,j)\bigg|\leq C\exp\left(-a \frac{\eps^2k}{\sqrt{\eps^4t}}\right).
 \end{equation}
\end{lemma}
\begin{proof}
Recall that $\p_t^\eps(k,j)$ is given by \eqref{eq:heatkernelRobinintegral} with $\mu=e^{-\eps}$. Summing over $j$, we obtain 
\begin{align*}
	\sum_{j\in \mathbb N} 	\p_t^\eps(k,j) &=  \frac{1}{2\I\pi}\oint_{\mathcal C_R}   e^{\frac{1}{2}\left(\xi+\xi^{-1}-2\right) t} \xi^k \sum_{j\in \mathbb N} \xi^{-j}  \frac{d \xi}{\xi}\nonumber  +  \frac{1}{2\I\pi}\oint_{\mathcal C_R}   e^{\frac{1}{2}\left(\xi+\xi^{-1}-2\right) t} \xi^k \sum_{j\in \mathbb N} \xi^{j+1} \frac{\mu-\xi}{1-\mu \xi}
	\frac{d \xi}{\xi},\\
	&= \frac{1}{2\I\pi}\oint_{\mathcal C_{>1}}   e^{\frac{1}{2}\left(\xi+\xi^{-1}-2\right) t} \frac{\xi^k}{1-1/\xi }   \frac{d \xi}{\xi} +  \frac{1}{2\I\pi}\oint_{\mathcal C_{<1}}   e^{\frac{1}{2}\left(\xi+\xi^{-1}-2\right) t} \frac{\xi^{k+1}}{1-\xi} \sum_{j\in \mathbb N} \xi^{j+1} \frac{\mu-\xi}{1-\mu \xi}
	\frac{d \xi}{\xi},
\end{align*}
where in the first equality, we simply brought the summation inside the integral, and in the second equality, we compute the sums after having deformed the first contour to be a circle of radius larger than $1$, and the second contour to be a circle of radius smaller than $1$.

Computing the residue at $\xi=1$, the first integral is 
\[\frac{1}{2\I\pi}\oint_{\mathcal C_{>1}}   e^{\frac{1}{2}\left(\xi+\xi^{-1}-2\right) t} \frac{\xi^k}{1-1/\xi }   \frac{d \xi}{\xi} = 1+\frac{1}{2\I\pi}\oint_{\mathcal C_{<1}}   e^{\frac{1}{2}\left(\xi+\xi^{-1}-2\right) t} \frac{\xi^k+1}{\xi-1 }   \frac{d \xi}{\xi}\]
so that 
\begin{equation}
	\sum_{j\in \mathbb N} 	\p_t^\eps(k,j) = 1 + \frac{1}{2\I\pi}\oint_{\mathcal C_{<1}}   e^{\frac{1}{2}\left(\xi+\xi^{-1}-2\right) t} \frac{\xi^{k+1}}{\xi-1 }  \left(1- \frac{\mu-\xi}{1-\mu\xi}\right) \frac{d \xi}{\xi}.
\end{equation}
Rearranging the formula, 
\begin{equation}
	1-\sum_{j\in \mathbb N} 	\p_t^\eps(k,j) = \frac{1}{2\I\pi}\oint_{\mathcal C_{<1}}   e^{\frac{1}{2}\left(\xi+\xi^{-1}-2\right) t} \frac{\xi^{k+1}}{\xi-1 } \frac{(\mu-1)(\xi+1)}{1-\mu\xi}  \frac{d \xi}{\xi}.
\end{equation}
Now, we scale and define $t=\eps^{-4}\tau, \tau\in (0,T]$ , $x=\eps^{2}k$, and  $$\xi=e^{\I \eps^2  (\theta+\I a)/\sqrt{\tau}}$$
where the constant $a>0$ is arbitrary. Under these scalings, 
\[ \mathfrak{Re}\left[\frac{1}{2}\left(\xi+\xi^{-1}-2\right) t\right] = \frac{1}{2}(a^2-\theta^2)\tau  +o(\eps). \]
To obtain a bound uniformly in $\eps$, let us write 
\[\mathfrak{Re}\left[\frac{1}{2}\left(\xi+\xi^{-1}-2\right) t\right] =\left(\cosh(a \eps^2 /\sqrt{\tau}] \cos(\theta \eps^2/\sqrt{\tau})  - 1\right) \eps^{-4}\tau.\]
If $t\geq 1$, $a\eps^2/\sqrt{\tau}\leq a$, and we can find constant $\underline\kappa, \overline\kappa$ such that for all $y\in [-a,a]$, we have $1+\underline\kappa y^2 \leq \cosh(y)\leq 1+\overline\kappa y^2$. 	Moreover,  for all $y\in (-\pi,\pi)$,    $1-y^2/2 \leq \cos(y)\leq 1- y^2/5$. Hence, there exists a constant $C=C(a, T)$ such that for $t\in [1, \eps^{-4}T]$, 
\[ \mathfrak{Re}\left[\frac{1}{2}\left(\xi+\xi^{-1}-2\right) t\right] \leq  C -\theta^2/C.\]
Under these scalings, we also have $\vert \xi^k \vert \leq e^{-a \frac{x}{\sqrt{\tau}}}$. Hence, we can find a constant $C=C(a,T)$
such that 
\begin{equation}
	\bigg\vert 1-\sum_{j\in \mathbb N} 	\p_t^\eps(k,j) \bigg\vert \leq \frac{C}{2\pi} \int_{\mathbb R} d\theta e^{C-\theta^2/C} e^{-a \frac{x}{\sqrt{\tau}}} 
\end{equation}
Finally, we conclude that for any $a>0$, there exists a constant $C(a,T)$ such that for $t\in [1, \eps^{-4}T]$, 
\begin{equation}
	\bigg\vert 1-\sum_{j\in \mathbb N} 	\p_t^\eps(k,j) \bigg\vert \leq C(a,T) e^{-a \frac{x}{\sqrt{\tau}}} 
\end{equation}
which is equivalent to the statement of the lemma. 
\end{proof}

\subsection{Two useful lemmas} \label{ssec:useful}
We conclude this section by collecting two lemmas that will be essential in the proofs of the next sections.

\begin{lemma}[\oo{\cite{dembo2016weakly,corwin2016open}}] \label{l:iteration}
	For any $n\in\N^*$, there exists $C=C(n)<\infty$ such that, for any $F$ bounded {(deterministic) function} on $\R_+\times \N$, for any $t,t'\in\R_+$ such that $t'\ge t+1$,
	\begin{equation}
		\bigg\|{ \int_t^{t'}}\sum_{\ell=0}^\infty F(s,\ell)dM_s(\ell)\bigg\|_{2n}^2\leq C\eps^2{ \int_t^{t'}}\sum_{\ell=0}^\infty \bar F(s,\ell) ^2\big\|Z_s(\ell)\big\|^2_{2n}ds, \label{eq:it1}
	\end{equation}
	where $\displaystyle\bar F(s,\ell):=\sup_{\substack{|s'-s|<1 \\ s'\geq 0}}|F(s',\ell)|$.
	In particular, when $F(s,\ell)=\p^\eps_s(k,\ell)$ for some $k\in\N$, we have  \begin{equation}\label{eq:it2}\bar{\p^\eps_s}(k,\ell)^2\leq Cs^{-1/2}\p^\eps_{s+1}(k,\ell).\end{equation}
\end{lemma}
\begin{proof}[Proof of Lemma \ref{l:iteration}]
The first estimate \eqref{eq:it1} is proved in \cite[Lemma 3.1]{dembo2016weakly} or \cite[Lemma 4.18]{corwin2016open},where one can see that  the particular choice of boundary conditions does not play any role. Moreover, from \eqref{e:hkbound3}, we have, for any $s \geqslant 0$,
\[ \sup_{|s-s'|\le 1} \p^\eps_{s'}(k,\ell) \le e^2 \p_{s+1}^\eps(k,\ell).\] Therefore, from \eqref{eq:boundheatkernel} it follows that $\bar{\p^\eps_s}(k,\ell)^2\leq e^4 (\p_{s+1}^\eps(k,\ell))^2 \le C s^{-1/2} \p_{s+1}^\eps(k,\ell).$ 
\end{proof}

\begin{lemma}[\oo{\cite{bertini1997stochastic}}]
		\label{l:oriane} If $\bar\eps\in (0,1)$ is such that, for all $\eps\in(0,\bar\eps),$ $\nu\leq \eps$ and $|\log(\eps)|\eps\leq 1$, for any $t\geq 0$, $n\in\N$, there exists $C=C(n)\in\R_+$ such that for all $k\in\N$,
		\[\E \bigg[ \sup_{s\in[0,1]} \big|Z_{t+s}(k)-Z_t(k)\big|^{2n} \bigg] \le C(n)\varepsilon^{2n}\;\big\|Z_t(k)\big\|_{2n}^{2n}. \]
		In particular, assume $T>0$, $n\in \N$, and assume that there exists $C_0=C_0(T,n)>0$ and $\varepsilon_0>0$ such that, for all $\varepsilon \in (0,\varepsilon_0)$, all $k \in \N$ and $t \in [0,\varepsilon^{-4}T]$, 
		\begin{equation} \label{eq:hypmoment} \big\| Z_t(k) \big\|_{2n} \le C_0. \end{equation}
		Then, there exists $C_1=C_1(T,n) >0$ such that, for all $\varepsilon \in (0,\oo{\eps_0\wedge\bar\eps})$,
		\[  \sup_{t\in [0,\varepsilon^{-4}T]} \sup_{k\in\N} \E \bigg[ \sup_{s\in[0,1]} \big|Z_{t+s}(k)-Z_t(k)\big|^{2n} \bigg] \le C_1 \varepsilon^{2n}. \]
		\end{lemma}

     \begin{proof} This proof is contained in (4.63)--(4.65) of \cite{bertini1997stochastic}. For completeness, we reproduce it in Appendix~\ref{a:bertini-giacomin}. 
	\end{proof}

\section{Near-equilibrium initial condition: proof of Theorem \ref{t:cvneareq}} \label{sec:neareq}

In this section we assume that the initial condition  is near-equilibrium, and we prove Theorem \ref{t:cvneareq}. We are thus interested in the rescaled process \oo{$\scZ_t(x):=Z_{\varepsilon^{-4}t}(\varepsilon^{-2}x)$} defined  in \eqref{eq:scaled} for any $x\in \varepsilon^2\N$ and then extended to $\R_+$ by linear interpolation.
Together with the uniqueness of the solution to the martingale problem \eqref{e:contmartpb}--\eqref{e:contmartpb2}, the next two propositions prove the desired result stated in Theorem \ref{t:cvneareq}.

\begin{proposition}[Tightness]\label{p:tightness}
	Under the assumptions of Theorem~\ref{t:cvneareq}, the sequence of processes $\left(\mathscr{Z}^\eps_s\right)_{s\in[0,T]}$ is tight in the space $D([0,T],C(\R_+))$. Moreover, any limit point belongs to $C([0,T],C(\R_+))$.
\end{proposition}

\begin{proposition}[Identification of limit points]\label{p:limitpoints}
	Under the assumptions of Theorem~\ref{t:cvneareq}, any limit point $(\mathscr{Z}_s)_{s\in[0,T]}$ of $\left(\mathscr{Z}^\eps_s\right)_{s\in[0,T]}$ in the space $D([0,T],C(\R_+))$ satisfies the continuous martingale problem \eqref{e:contmartpb}--\eqref{e:contmartpb2}.
\end{proposition}

The rest of the section is devoted to the proof of Proposition \ref{p:tightness} and Proposition \ref{p:limitpoints}. 


\subsection{Tightness: proof of Proposition~\ref{p:tightness}}\label{s:tightnessneareq}
The \oo{proof relies on} the following lemma.  

\begin{lemma}\label{l:Ztbounds-cont}
	Fix $T>0$ and assume that the initial condition $\mathscr{Z}_0^\eps \in C(\R_+)$  is near-equilibrium (Definition \ref{de:near-eq-cont}). Then for all $n\in\N$ and $\alpha\in[0,\frac12)$, there exists $C=C(\alpha,n,T),a\in\R_+$ such that for all $\eps$ \mm{small enough}, for any $x,y\in\R_+$, $s,s'\in[0,T]$,
	\begin{align}
		\|\scZ_s(x)\|_{2n}&\leq Ce^{au},\label{e:momentZt}\\
		\|\scZ_s(x)-\scZ_s(y)\|_{2n}&\leq C|x-y|^\alpha e^{a(x+y)}, \qquad \mm{\text{if } s\ge \oo{\eps^4}}\label{e:momentdiffZt}\\
		\|\scZ_s(x)-\scZ_{s'}(x)\|_{2n}&\leq C \left(\eps^{2\alpha}\vee |s-s'|^{\alpha/2}\right)e^{2ax}.\label{e:momenttempsZt}
	\end{align}
\end{lemma}
These estimates, together with Arzela-Ascoli's Theorem, imply Proposition~\ref{p:tightness} (see \cite[Chapter 3]{billingsley1968convergence}).  This implication is detailed in \cite{bertini1997stochastic}, and hinted at in later papers such as \cite{corwin2016open,parekh2017kpz}. Note that, in constrast with Lemmata 4.2 and 4.3 in \cite{bertini1997stochastic}, we have a small restriction on the value of $s$ in \eqref{e:momentdiffZt}, which is a consequence of the use of Lemma~\ref{l:iteration} below. Since the question was asked by an anonymous referee, let us explain why this does not affect the proof of tightness (as is already the case in \textit{e.g.}\@ \cite{corwin2016open,parekh2017kpz}). The strategy of \cite{bertini1997stochastic} is to use Kolmogorov's Lemma to deduce from Lemma~\ref{l:Ztbounds-cont} the existence of a Hölder version of the process. However, since our microscopic process has jumps, it cannot be Hölder in time. Therefore \cite{bertini1997stochastic} consider $\bar\scZ$, which is a version of the interpolation of $\scZ$ at times in $\eps^4\N$ with the desired regularity; for this step our restriction on $s$ is no obstacle. Then it remains to control the error between $\bar\scZ$ and $\scZ$: this uses Lemma~\ref{l:oriane} and only requires the moment bound \eqref{e:momentZt} from Lemma~\ref{l:Ztbounds-cont} (see the proof of Lemma 4.7 in \cite{bertini1997stochastic}).

In order to prove Lemma \ref{l:Ztbounds-cont}, the starting point is to write the Duhamel equation for $Z$ which is a consequence of Lemma~\ref{lem:quadratic}: for any $t\ge 0$,
	\begin{equation}\label{eq:ZDuhamel}
		Z_t(k)=\sum_{\ell=0}^\infty \p^\eps_t(k,\ell)Z_0(\ell)+\int_0^t\sum_{\ell=0}^\infty \p^\eps_{t-s}(k,\ell)dM_s(\ell).
	\end{equation}


	 We will then use \oo{two} main ingredients:
	\begin{enumerate}
		\item[(1)] the bounds on the heat kernel $\p^\eps$ which are proved in Lemma \ref{lem:bounds} and Lemma \ref{lem:cons} ; 
		\item[(2)] an $L^{2n}$-bound on the integral terms contained in the second term in \eqref{eq:ZDuhamel}, similar to the one stated in \cite[Lemma 5.3]{parekh2017kpz}. This uses Lemma \ref{l:iteration} and allows the iteration procedure described in the proof below.
	\end{enumerate}

\begin{proof}[Proof of  Lemma \ref{l:Ztbounds-cont}]

The proof is inspired by \cite[Proof of Proposition 5.4]{parekh2017kpz}. We highlight here its main steps and the adjustments we need to make, notably the fact that we need the killing estimate \eqref{eq:killingestimate} to prove the Hölder properties.
 We denote the microscopic time variables by $t=\eps^{-4} s, t'=\eps^{-4}s'$.
  
 The first bound \eqref{e:momentZt} in Lemma \ref{l:Ztbounds-cont} is obtained by an iteration argument using Duhamel formula \eqref{eq:ZDuhamel}. The $L^{2n}$ norm of the first term is bounded by $Ce^{ak}$ thanks to the triangular inequality, together with \eqref{e:momentZ0} and \eqref{e:hkbound1} with $a_1=0, a_2=a$. 
 
 Then, for $t\geq 1$, we can apply Lemma~\ref{l:iteration} to the second term. After the same manipulations as in \cite[Proof of Proposition 5.4]{parekh2017kpz}, using again \eqref{e:hkbound1} and defining  
 \[
  [Z_t]_{2n}:=\sup_{k\in\N}e^{-a\eps^2 k}\|Z_t(k)\|_{2n},
 \]
 we obtain, when $t\ge 1$,
 \[
  [Z_t]_{2n}^2\leq C+C\eps^2\int_0^t ds (t-s)^{-1/2}[Z_s]_{2n}.
 \]
 Using Lemma~\ref{l:oriane} shows that this holds in fact also for $t\leqslant 1$. We can now perform the same iteration as in \cite[Proof of Proposition 5.4]{parekh2017kpz} to conclude the proof of \eqref{e:momentZt}.

For the second bound \eqref{e:momentdiffZt}, write as a consequence of \eqref{eq:ZDuhamel}
\begin{equation}\label{e:Ztdiff1}
	Z_t(k)-Z_t(\ell)=\sum_{j=0}^\infty \left[\p^\eps_t(k,j)-\p^\eps_t(\ell,j)\right]Z_0(j)+\int_0^t\sum_{j=0}^\infty \left[\p^\eps_{t-r}(k,j)-\p^\eps_{t-r}(\ell,j)\right]dM_r(j).
\end{equation}
To control the first sum, we can proceed as in \cite{parekh2017kpz} and extend $Z_0$ into a function $\widetilde Z_0$ over $\Z$ by imposing that $j\mapsto Z_0(j-1)-\mu Z_0(j)$ is odd. Note that $Z_0$ now depends implicitly on $\eps$. Then it is easy to check  that $\tilde Z_0$ still satisfies \eqref{e:momentdiffZ0} (possibly changing $C$), and $\sum_{j=0}^\infty \p^\eps_t(k,j)Z_0(j)=\sum_{j\in\Z}p_t(k-j)\widetilde Z_0(j)$. This allows to bound the $L^{2n}$-norm of the first term by 
\[
 \sum_{j\in\Z}p_t(k-j)\|\widetilde Z_0(j)-\widetilde Z_0(j+\ell-k)\|_{2n}\leqslant C(\eps^2|k-\ell|)^\alpha e^{a\eps^2(k+\ell)},
\]
using \eqref{e:momentdiffZ0} and the analog of \eqref{e:hkbound1} for the standard heat kernel.

To bound the $L^{2n}$--norm of the second term, we use again Lemma~\ref{l:iteration} with $F(r,j)=\p^\eps_{t-r}(k,j)-\p^\eps_{t-r}(\ell,j)$. By \eqref{e:hkbound2} and \eqref{e:hkbound3}, similarly as in \cite{parekh2017kpz}, \mm{for all $v \in (0,1)$,}
\[
 \bar F(r,j)^2\leq C(t-r)^{-(1+\mm{v})/2}|k-\ell|^\mm{v}\left(\p^\eps_{t-r+1}(k,j)+\p^\eps_{t-r+1}(\ell,j)\right).
\]
Then, using \eqref{e:hkbound1} and \eqref{e:momentZt} and following the same steps as \cite{parekh2017kpz}, we get that the $L^{2n}$--norm of the second term in \eqref{e:Ztdiff1} is bounded by $C(\eps^2|k-\ell|)^{\mm{v}/2}e^{a\eps^2(k+\ell)}$.  \mm{Choosing $v=2\alpha$, for any $\alpha \in (0,\frac12)$, we finally get \eqref{e:momentdiffZt}}.

For the last bound \eqref{e:momenttempsZt}, we need a new ingredient with respect to \cite{parekh2017kpz}'s proof. We start this time from the Duhamel equation between times $s$ and $s'$ (again a consequence of Lemma~\ref{lem:quadratic}): assume first that $t-t' \ge 1$, then \begin{multline}
 Z_t(k)-Z_{t'}(k)=\sum_{j=0}^\infty \p^\eps_{t-t'}(k,j)\left(Z_{t'}(j)-Z_{t'}(k)\right)+\bigg(\sum_{j=0}^\infty\p^\eps_{t-t'}(k,j)-1\bigg)Z_{t'}(k)\\+\int_{t'}^t\sum_{j=0}^\infty \p^\eps_{t-r}(k,j)dM_r(j). \label{eq:mult}
\end{multline} 
Using \eqref{e:momentdiffZt} and \eqref{e:hkbound1} with $a_1=1, a_2=a$, the $L^{2n}$--norm of the first term is bounded as in \cite{parekh2017kpz}, by
\[
 C\eps^{2\alpha}(t-t')^{\alpha/2}e^{2a\eps^2 k}.
\]
\mm{The last term in \eqref{eq:mult} is bounded \oo{by the same quantity (up to the constant),} like in \cite{parekh2017kpz}, thanks to Lemma \ref{l:iteration} with $F(r,j)=\p^\eps_{t-r}(k,j)$ together with \eqref{eq:it2} and  \eqref{e:momentZt}.}

The second term is the reason why we proved the killing estimate \eqref{eq:killingestimate}: combined with \eqref{e:momentZt}, it gives that the $L^{2n}$--norm of the second term is bounded by $Ce^{-\eps^2 k/\sqrt{\eps^4(t-t')}}$. Let us denote the macroscopic variable $x=\eps^2 k$ and recall that $t=\eps^{-4}s$ and $t'=\eps^{-4}s$.

 Let us now split cases. If $x \geqslant \left(s-s'\right)^{\sqrt{\alpha/2}}$,  the above estimates combined yield (adjusting the constant)
\[
 \big\|\eps^{-2}(Z_t(k)-Z_{t'}(k))\big\|_{2n}\leqslant  C \Big(  \varepsilon^{2\alpha} \oo{ |t-t'|^{\alpha/2}} + e^{-(\eps^4(t-t'))^{-\left(1/2-\sqrt{\alpha/2}\right)}}\Big).
\]
Since $\sqrt{\alpha/2}<1/2$, this gives \eqref{e:momenttempsZt} in this first case. If instead $x<\left(s-s'\right)^{\sqrt{\alpha/2}}$, let us define $k':=\lceil\eps^{-2}(s-s')^{\sqrt{\alpha/2}}\rceil.$ By the triangle inequality,
\begin{align*}
 \big\|\eps^{-2}(Z_t(k)-Z_{t'}(k))\big\|_{2n}\leqslant&\; \big\|\eps^{-2}(Z_t(k)-Z_{t}(k'))\big\|_{2n}+\big\|\eps^{-2}(Z_t(k')-Z_{t'}(k'))\big\|_{2n}\\
 &\qquad+\ \big\|\eps^{-2}(Z_{t'}(k')-Z_{t'}(k))\big\|_{2n} \vphantom{\bigg(}\\
 \leqslant& \; C\varepsilon^{2\alpha} \oo{ |t-t'|^{\alpha/2}}+2C(\eps^2k')^{\sqrt{\alpha/2}},
\end{align*}
 where we have used \eqref{e:momenttempsZt} applied to $k'$ (which we just proved holds) and \eqref{e:momentdiffZt} twice, with exponent $\sqrt{\alpha/2}<1/2$ (since {$0 < k \leq k'$}, $|k-k'|\leqslant k'$). It now remains to replace $k'$ by its expression. We get 
 \[
  \eps^2 k'\leq \eps^2+\left(s-s'\right)^{\sqrt{\alpha/2}}.
 \]
 Since $|t-t'|\geqslant 1$, $\eps^2\leq\left(\eps^4(t-t')\right)^{1/2}\leqslant C\left(\eps^4(t-t')\right)^{\sqrt{\alpha/2}}$, where the constant is chosen so that $u^{1/2}\leq C u^{\sqrt{\alpha/2}}$ for all $u\in [0,T]$.
 
  Putting these estimates together concludes the proof for $t-t'\ge 1$. The case $t-t'<1$ can be handled using Lemma~\ref{l:oriane}. \end{proof}

\subsection{Identification of limit points: proof of Proposition~\ref{p:limitpoints}}
\label{s:limitpoints}
\subsubsection{Discrete martingale problem}

Let us denote by $\Delta^\eps:=\Delta^{\mu}$ the discrete Laplacian with boundary condition defined in \eqref{eq:deltamu} with $\mu=e^{-\eps}$.
We also introduce, for any $\phi,\psi:\R_+\to\R$ which are square summable, the following notation:
\begin{equation}
	(\psi, \phi)_\eps:=\eps^2\sum_{k=0}^\infty\phi(\eps^2k)\psi( k).
\end{equation}
\mm{We will prove that the limit points of the rescaled process satisfy in the limit the martingale problem stated in Definition \ref{de:solution}.  Let us take $\phi \in \oo{\mathcal C_c^\infty(\R)}$.}
From Lemma  \ref{lem:quadratic}, 
\begin{align}
	N^\eps_t(\phi):=(Z_{\eps^{-4}t},\phi)_\eps-(Z_0,\phi)_\eps-\frac{1}{2}\int_0^{\eps^{-4} t} (\Delta ^\eps Z_s,\phi)_\eps \; ds
\end{align} is a martingale. Let us compute
\begin{align*}
	\eps^{-2}(\Delta ^\eps Z_s,\phi)_\eps 
	&=\sum_{k=0}^\infty Z_s(k)\Big[\phi(\eps^2(k+1))+\phi(\eps^2(k-1))-2\phi(\eps^2k)\Big] +Z_s(-1)\phi(0)-Z_s(0)\phi(-\eps^2) \vphantom{\sum_{k=0}^\infty}\\
	&=\sum_{k=0}^\infty Z_s(k)\eps^4\phi''(\eps^2k)  \\ & \quad +\sum_{k=0}^\infty Z_s(k)\left[\Delta\phi(\eps^2\cdot)(k)-\eps^4\phi''(\eps^2k)\right]+\sqrt{\frac{q}{p}}Z_s(0)\phi(0)-Z_s(0)\phi(-\eps^2),\\
	&=\eps^4\sum_{k=0}^\infty \scZ_{\eps^4s}(\eps^2k)\phi''(\eps^2k)  \\ & \quad + \sum_{x=0}^\infty \scZ_{\eps^4s}(\eps^2k)\left[\Delta\phi(\eps^2\cdot)(k)-\eps^4\phi''(\eps^2k)\right] +  \scZ_{\eps^4s}(0)\left[\sqrt{\frac{q}{p}}\phi(0)-\phi(-\eps^2)\right].
\end{align*}
Therefore $N_t^\eps(\phi)$ can be rewritten as
\begin{align}
	N_t^\eps(\phi)=&\;\eps^2\sum_{k=0}^\infty \phi(\eps^2k)\scZ_{t}(\eps^2 k)-\eps^2\sum_{k=0}^\infty \phi(\eps^2k)\scZ_{0}(\eps^2 k)-\frac{1}{2}\int_0^t\eps^2\sum_{k=0}^\infty\scZ_s(\eps^2k)\phi''(\eps^2k)ds\notag\\&-\frac{1}{2}\eps^{-2}\left[\sqrt{\frac{q}{p}}\phi(0)-\phi(-\eps^2)\right]\int_0^t\scZ_s(0)ds+\mathfrak{E}_1(\eps),
	\label{eq:martingaleepsilon}
\end{align}
where the error term is
\begin{equation}
	\mathfrak{E}_1(\eps):=-\frac{1}{2}\eps^{-2}\sum_{k=0}^\infty\left[\Delta^\eps \phi(\eps^2\cdot)(k)-\eps^4\phi''(\eps^2k)\right]\int_0^t\scZ_s(\eps^2k)ds.
\end{equation}
\mm{Now,} let us fix $\psi:\R\rightarrow\R_+$ smooth, compactly supported, such that $\psi(0)=1$ and $\psi'(0)=0$, and define
\begin{equation}\label{e:defphieps}\phi_\eps=\phi+\eps\phi'(0)\psi,\end{equation} so that $\phi_\eps(0)=\eps \phi_\eps'(0)$. From the previous computation \eqref{eq:martingaleepsilon}, the following is a martingale:
\begin{multline*} 
	N_t^\eps(\phi_\eps)=\eps^2\sum_{k=0}^\infty \phi(\eps^2k)\scZ_{t}(\eps^2 k)  - \eps^2\sum_{k=0}^\infty \phi(\eps^2k)\scZ_{0}(\eps^2 k)  \\  - \frac{\eps^2}{2}\int_0^t\sum_{k=0}^\infty\scZ_s(\eps^2k)\phi''(\eps^2k)ds+R_1+R_2+R_3 \mm{+R_4},
	\label{eq:martingaleepsilon2}
\end{multline*}
where 
\begin{align}
	R_1&:=-\frac{1}{2}\eps^{-2}\sum_{k=0}^\infty\left[\Delta^\eps\phi_\eps(\eps^2\cdot)(k)-\eps^4\phi_\eps''(\eps^2k)\right]\int_0^t\scZ_s(\eps^2k)ds\\
	R_2&:=-\frac{1}{2}\eps^{-2}\left[e^{-\eps}\phi_\eps(0)-\phi_\eps(-\eps^2)\right]\int_0^t\scZ_s(0)ds,\\
	R_3&:=\eps^3\phi'(0)\sum_{k=0}^\infty\psi(\eps^2 k)\left[\scZ_t(\eps^2k)-\scZ_0(\eps^2 k)\right]  \\
	 \mm{R_4} & := \mm{- \frac{\eps^3}{2} \phi'(0) \int_0^t \sum_{k=0}^\infty \scZ_s(\eps^2k) \psi''(\eps^2 k)ds.}
  \end{align}
Moreover, the quadratic variation of $N_t^\eps(\phi_\eps)$ is given (see Lemma~\ref{lem:quadratic} and Lemma \ref{lem:quadeps}) by:
\begin{align}
	[N^\eps(\phi_\eps)]_t=\; &\eps^4\sum_{k=1}^\infty\phi^2_\eps(\eps^2k)\int_0^{\eps^{-4}t}\left[\eps^2Z_s(k)^2+\nabla^+Z_s(k)\nabla^-Z_s(k)+o(\eps^2)Z_s(k)^2\right]ds\\
	&+\eps^4\phi^2_\eps(0)\int_0^{\eps^{-4}t}\left[\eps^2Z_s(0)^2\oo{-}\eps Z_s(0)\nabla^+Z_s(0)+o(\eps^2)Z_s(0)^2\right]ds.
\end{align}
In order to conclude, we prove in the next section that $R_1,R_2,R_3,\oo{R_4}$ vanish in probability as $\eps\rightarrow 0$, which then establishes that $N_t(\phi)$ (defined in \eqref{e:contmartpb}) is a martingale for any $\mathscr{Z}$ limit point of $\scZ$. Then, we verify \eqref{e:contmartpb2}.

\subsubsection{Proof of Proposition \ref{p:limitpoints}}
In order to prove that $R_1,R_2,R_3,\oo{R_4}$ vanish in probability, we estimate their $\|\cdot\|_n$-norms. 
Let us start with $R_2$, where we will see why we chose $\phi_\eps$ as in \eqref{e:defphieps}.

\begin{enumerate}
	\item[($R_2$)] For any $n\in\N$, by Lemma~\ref{l:Ztbounds-cont},
	\begin{align}
		\|R_2\|_n&\leq C\eps^{-2}t\left|e^{-\eps}\eps\phi'(0)-\phi(-\eps^2)-\eps\phi'(0)\psi(-\eps^2)\right|\\
		&\leq C\eps^{-2}t\left|(1-\eps)\eps\phi'(0)+\eps^2\phi'(0)-\eps\phi'(0)+o(\eps^2)\right|=o(1).
	\end{align}
	Therefore $R_2$ vanishes in any $L^n$, $n\in\N$.
	
	\medskip
	
	\item[($R_3$, \mm{$R_4$})] Moreover, the same holds for $R_3$ \mm{and $R_4$}, again thanks to Lemma~\ref{l:Ztbounds-cont}, and because \mm{both} series  $\eps^2\sum_{k=0}^\infty \psi(\eps^2k)e^{2a\eps^2k}$ \mm{and $\eps^2\sum_{k=0}^\infty \psi''(\eps^2k)e^{2a\eps^2k}$} converge to, respectively,  $\int_{\R_+}\psi(u)e^{2au}du$ \mm{and $\int_{\R_+}\psi''(u)e^{2au}du$} \mm{which are both finite}.
	\medskip
	
	\item[($R_1$)] Let us now consider $R_1$. By Lemma~\ref{l:Ztbounds-cont}, it is enough to show that, for all $k\in\N^*$, 
	\[\left|\Delta^\eps\phi_\eps(\eps^2\cdot)(k)-\eps^4\phi_\eps''(\eps^2k)\right|=o(\eps^4),\] and \[
	\left|\Delta^\eps\phi_\eps(\eps^2\cdot)(0)-\eps^4\phi_\eps''(0)\right|=o(\eps^2),\]
	which can be checked by {Taylor expansion, using that $|\phi'''|$ is bounded} thanks to the compactness of the support of $\phi$.
\end{enumerate}

It remains to show that any limit point $\mathscr{Z}$ of $\scZ$ satisfies \eqref{e:contmartpb2}. Let us rewrite 
\begin{align}
	[N^\eps(\phi_\eps)]_t&=\int_0^t\eps^2\sum_{k=0}^\infty\phi^2(\eps^2k)\scZ_s(\eps^2k)^2ds+R_1'+R_2'+R_3'+R_4',
\end{align}
where
\begin{align}
	R_1'&=\int_0^t\eps^2\sum_{k=0}^\infty[\phi_\eps^2(\eps^2k)-\phi^2(\eps^2k)]\scZ_s(\eps^2k)^2ds,\\
	R'_2&=\eps^4\int_0^{\eps^{-4}t}\sum_{k=1}^\infty\phi_\eps^2(\eps^2k)\nabla^+Z_s(k)\nabla^-Z_s(k)ds,\\
	R'_3&=-\eps^5\int_0^{\eps^{-4}t}\phi_\eps^2(0)Z_s(0)\nabla^+Z_s(0)ds,\\
	R'_4&=o(1)\int_0^t\eps^2\sum_{k=0}^\infty\phi^2(\eps^2k)\scZ_s(\eps^2k)^2ds.
\end{align}
We have to show that $R_1',R_2',R_3',R_4'$ {vanish in $L^2$} as $\eps\rightarrow 0$. 

\begin{enumerate}
	\item[($R_1'$)] For the first term, the result follows from the bound $\|\phi_\eps^2-\phi^2\|_\infty=\mathcal{O}(\eps)$, the fact that $\phi,\phi_\varepsilon$ have compact support,  and from \eqref{e:momentZt}. 
	\medskip
	
	\item[($R_3'$ and $R_4'$)] Both terms $R_3'$ and $R_4'$ are also controlled with \eqref{e:momentZt} and \eqref{e:momentdiffZt}. 
	\medskip
	
	\item[($R_2'$)] 
	Similarly to \cite{corwin2016open} and \cite[Proof of Theorem 5.7]{parekh2017kpz}, we split $R'_2$ in two parts: write $R'_2=r_1+r_2$, where
	\begin{align}
		r_1&=\eps^4\int_0^{\eps^{-3}{\wedge \eps^{-4}t}}\sum_{k=1}^\infty\phi_\eps^2(\eps^2k)\nabla^+Z_s(k)\nabla^-Z_s(k)ds\\
		r_2&=\eps^4\int_{\eps^{-3}{\wedge \eps^{-4}t}}^{\eps^{-4}t}\sum_{k=1}^\infty\phi_\eps^2(\eps^2k)\nabla^+Z_s(k)\nabla^-Z_s(k)ds.
	\end{align}
	By \eqref{e:momentdiffZt}, $\|\nabla^\pm Z_s(k)\|_2\leq C\eps^{2\alpha}e^{2a\eps^2k}$ for any $\alpha<\frac12$ and some $C,a>0$. Therefore, we can bound 
	\begin{align}
		\|r_1\|_2\leq C\eps^4\eps^{-3}\eps^{4\alpha}\sum_{k=1}^\infty \phi_\eps^2(\eps^2k)e^{4a\eps^2k}\leq C'\eps^{4\alpha-1},
	\end{align}
	which goes to $0$ if we choose for instance $\alpha=1/3$. On the other hand, { if $\eps^{-3}\leq \eps^{-4}t$ {(else, $r_2=0$)}}, $\mathbb{E}[r_2^2]$ is equal to 
	\begin{multline*}
		2\eps^8\int_{\eps^{-3}}^{\eps^{-4}t}\int_{\eps^{-3}}^s\sum_{k,\ell=1}^\infty \phi_\eps^2(\eps^2k)\phi_\eps^2(\eps^2\ell)\mathbb{E}\big[\nabla^+Z_s(k)\nabla^-Z_s(k)\nabla^+Z_r(\ell)\nabla^-Z_r(\ell)\big]dr ds\\
		=2\eps^8\int_{\eps^{-3}}^{\eps^{-4}t}\int_{\eps^{-3}}^s\sum_{k,\ell=1}^\infty \phi_\eps^2(\eps^2k)\phi_\eps^2(\eps^2\ell)\mathbb{E}\big[\nabla^+Z_r(\ell)\nabla^-Z_r(\ell)U(k,r,s)\big]dr ds,
	\end{multline*}
	where 
	\begin{equation}
		U(k,r,s):=\mathbb{E}\Big[\nabla^+Z_s(k)\nabla^-Z_s(k)|\mathcal F_r\Big], \qquad \mathcal F_r:=\sigma\left({Z_{r'}(j); j\in\N, r'\leq r}\right).
	\end{equation}
	Therefore, 
	\begin{align}
		\mathbb{E}[r_2^2]&\leq C\eps^8\eps^{4\alpha}\int_{\eps^{-3}}^{\eps^{-4}t}\int_{\eps^{-3}}^s\sum_{k,\ell=1}^\infty \phi_\eps^2(\eps^2k)\phi_\eps^2(\eps^2\ell)e^{4a\eps^2\ell}\mathbb{E}\big[|U(k,r,s)|\big]drds\\
		&\leq C\eps^{6+4\alpha}\int_{\eps^{-3}}^{\eps^{-4}t}\int_{\eps^{-3}}^s\sum_{k=1}^\infty \phi_\eps^2(\eps^2k)\mathbb{E}\big[|U(k,r,s)|\big]drds,
	\end{align}
	{where the value of $C$ changed between the two lines.} We will prove in Lemma \ref{l:U} below an estimate on $\mathbb E\big[|U(k,r,s)|\big]$  which shows that 
	\begin{align}
		\mathbb{E}[r_2^2]&\leq C\eps^{4+6\alpha}\int_{\eps^{-3}}^{\eps^{-4}t}\int_{\eps^{-3}}^s\frac{1}{\sqrt{s-r}}dr ds\leq C\eps^{1+6\alpha}.
	\end{align}
\end{enumerate}
Therefore, $R'_i$ goes to $0$ in $L^2$ for $i=1,...,4$, and both properties \eqref{e:contmartpb} and \eqref{e:contmartpb2} are satisfied, which  conclude the proof of Proposition \ref{p:limitpoints}. We now prove the last needed estimate:
\begin{lemma}\label{l:U}
	For any $\alpha<\frac12$, there exist $a,C=C(\alpha,t)>0$ such that, for any $k\in\N^*$ and if  $\eps^{-3}\leq r\leq s\leq \eps^{-4}t$, 
	\begin{equation}
		\mathbb E\big[|U(k,r,s)|\big]\leq C\eps^{2\alpha}\frac{e^{a\eps^2k}}{\sqrt{s-r}}.
	\end{equation}
\end{lemma}

\begin{proof}[Proof of Lemma~\ref{l:U}]
	As in \cite{bertini1997stochastic,corwin2016open,parekh2017kpz}, we write
 	\begin{multline}
		U(k,r,s)  =  \nabla^-I_s(k)\nabla^+I_s(k)
		+\nabla^-I_s(k)\nabla^+N_r^s(k)+\nabla^-N_r^s(k)\nabla^+I_s(k) \\ 
		 +\nabla^-N_r^s(k)\nabla^+N_r^s(k) 
		 +  \mathbb E\left[ \int_r^s\sum_{\ell=0}^\infty K_{s-\tau}(k,\ell)d\langle M(\ell)\rangle_\tau \bigg| \mathcal F_r \right], \label{e:UK}
	\end{multline}
	where 
	\begin{align*}
		I_s(k)&=\sum_{\ell=0}^\infty \p^\eps_s(k,\ell)Z_0(\ell)\\N_r^s(k)&=\int_0^r\sum_{\ell=0}^\infty \p^\eps_{s-\tau}(k,\ell)dM_\tau(\ell),\\ K_t(k,\ell)&=\nabla^+\p_t^\eps(k,\ell)\nabla^-\p_t^\eps(k,\ell).
	\end{align*}
	By \eqref{e:hkbound5}, \eqref{e:momentZ0} and Lemma~\ref{lem:quadratic}, we have (see \textit{e.g.}~\cite[Proof of Lemma 5.7]{corwin2016open})
	\begin{equation}
		\max\left\{\mathbb E\left[\left(\nabla^\pm I_s(k)\right)^2\right], \mathbb E\left[\left(\nabla^\pm N_r^s(k)\right)^2\right]\right\}\leq C\eps^{2\alpha}\frac{e^{a\eps^2k}}{\sqrt{s-r}}.
	\end{equation}
	Moreover, \oo{using Lemma~\ref{lem:quadeps}}, the expectation term in \eqref{e:UK} can be split into
	\begin{align*}
		(\eps^2+o(\eps^2)) & \sum_{\ell=0}^\infty\int_r^sK_{s-\tau}(k,\ell)\mathbb{E}[(Z_\tau(\ell))^2|\mathcal F_r]d\tau -\eps\int_r^sK_{s-\tau}(k,0)\mathbb{E}[Z_\tau(0)\nabla^+Z_\tau(0)|\mathcal F_r]d\tau\\\oo{+}&\sum_{\ell=1}^\infty\int_r^sK_{s-\tau}(k,\ell)\mathbb{E}[\nabla^-Z_\tau(\ell)\nabla^+Z_\tau(\ell)|\mathcal F_r]d\tau.
	\end{align*}
	If we can bound the first two terms by $C\eps^{2\alpha}e^{a\eps^2k}/\sqrt{s-r}$, using \eqref{e:hkboundcancel1}, \eqref{e:hkboundcancel2}, the same iterative procedure as described in \cite{corwin2016open} yields the desired result. 
	
	Let us consider the second term (which is new with respect to the case treated in \cite{corwin2016open}). By \eqref{e:momentZt}, \eqref{e:momentdiffZt} and \eqref{e:hkbound4}, we have for any $\alpha<\frac12, b\geq 0,\nu\leq 1$
	\begin{align*}
		\left|\eps\int_r^sK_{s-\tau}(k,0)\mathbb{E}\big[Z_\tau(0)\nabla^+Z_\tau(0)|\mathcal F_r\big]d\tau\right|&\leq C\eps^{1+4\alpha}e^{-2b|k|}\int_r^s\left(1\wedge(s-\tau)^{1+\nu}\right)d\tau\\ &=C\eps^{1+4\alpha}e^{-2b|k|}(s-r)^{-\nu}.
	\end{align*}
	It thus remains to check that 
	\begin{align}
		\eps^2\sum_{\ell=0}^\infty\int_u^sK_{s-\tau}(k,\ell)\mathbb{E}[(Z_\tau(\ell))^2|\mathcal F_u]d\tau\leq C\eps^{2\alpha}\frac{e^{a\eps^2k}}{\sqrt{s-u}}, 
	\end{align}
	which is exactly Lemma 5.8 in \cite{corwin2016open}. In turn, this inequality is a consequence of bounds \eqref{e:hkbound4}, \eqref{e:hkbound5} and Lemma~\ref{l:Ztbounds-cont} (see \cite{corwin2016open}).
\end{proof}

\section{Delta prime initial condition: proof of Theorem \ref{thm:conv}}\label{s:deltaprime}
In this section we come back to our initial problem and assume that the initial condition is empty so that  $Z_0(k) = \mu^{k}$, where we recall that $\mu=\sqrt{q/p}$. Therefore, the initial condition is not near-equilibrium, and we  now use  the different scaling introduced in \eqref{eq:scaled1}, namely 
\begin{equation*}
	\rZ_t(x)=\eps^{-2} Z_{\eps^{-4}t}(\eps^{-2}x), \qquad x \in \R_+.
\end{equation*}
Let us now prove Theorem \ref{thm:conv}.

Recall that we already know from  Section \ref{sec:proofshe} existence and uniqueness of the solution to  the SHE (Definition \ref{de:solution}) with Dirichlet boundary condition and initial condition $\mathcal Z_{\rm ini}=-2\delta_0'$.  Also, we have obtained in Section \ref{sec:heatkernel} explicit estimates involving the discrete Dirichlet heat kernel, relying on exact computations {and random walk arguments}. 

\mm{In this section, some exact computations inspired by the recent paper \cite{barraquand2022markov} and} the heat kernel estimates from Section \ref{sec:heatkernel} \mm{will be crucial. More precisely, we start by} showing in Sections~\ref{subs:momentsZ} and \ref{subs:neareq} that after time $\delta>0$, $\rZ_{t}(\cdot)$ is near-equilibrium in the sense of Definition \ref{de:near-eq-cont}. This allows us to follow the same strategy as in Section~\ref{sec:neareq} to obtain tightness in the space $D([\delta,T],C(\mathbb R_+))$ for any $0<\delta<T$, and show that the limit points satisfy the SHE (Proposition~\ref{prop:tight2}). Then, following \cite{parekh2017kpz}, we will  show in Section \ref{sec:limitZ} that there exists a limit point in the space $D((0,T],C(\mathbb R_+))$, which is solution to the SHE with Dirichlet boundary condition.  We also prove in Section \ref{sec:limitZ} second moment bounds satisfied by this limit point, using \cite{barraquand2022markov}. 
We finally determine the initial condition in Section \ref{sec:conclusion}, using the moment bounds  from Section \ref{sec:limitZ}.

\subsection{Moment estimate for $\mathcal{Z}^\eps$} \label{subs:momentsZ}
\begin{proposition}\label{p:momentZ}
	Fix $\delta>0$ and  $T\geqslant \delta$. Let $n \in \N$. There exists $C =C(\delta, T,n)\in\R_+$ and $\eps_0$ such that, for all $\varepsilon \in (0,\eps_0)$, $k \in \N$, and $t \in [\eps^{-4}\delta,\varepsilon^{-4}T]$:
	\begin{align}
		\|\varepsilon^{-2} Z_t(k)\|_n & \leqslant C .\label{eq:Zp} 
	\end{align}
	\end{proposition}

 When $t=\eps^{-4}\tau$, $\|\varepsilon^{-2} Z_t(k)\|_n$ cannot be bounded by $C/\sqrt{\tau}$ as was the case in \cite{parekh2017kpz}, but rather by $C/\tau$. Because the latter bound is not integrable in the variable $\tau$ near $0$, we cannot adapt the method from \cite[Proposition 6.2]{parekh2017kpz}, based on the martingale decomposition, to prove such an estimate. Hence, we will employ another method, based on exact moment formulas from \cite{barraquand2022markov}. 

More precisely, the bound \eqref{eq:Zp} can be obtained using the same arguments as in \cite[Proposition 4.6]{barraquand2022markov}. However, this result is stated there for another boundary condition (leading to Neumann type boundary condition in the $\eps\to 0$ limit). It is claimed in \mm{\cite[Section 4.3]{barraquand2022markov}} that the same arguments apply as well in the present case. Since explanations given in \cite{barraquand2022markov} do not contain much details, we provide a complete proof of the bound \eqref{eq:Zp}. First, in Section~\ref{sss:integral_formula} we present an exact integral formula for the moments of $Z$ (see \eqref{eq:momentZ}). From there, in Section \ref{subs:proof} we deduce Proposition~\ref{p:momentZ}.

\subsubsection{Exact integral formula}\label{sss:integral_formula}
Mixed moments of $Z_t(k)$ are computed in \cite{barraquand2022markov}, in the form of certain finite sums of explicit contour integrals. The structure of those sums is somewhat intricate, and we will need to introduce a few pieces of notation before stating the result. First of all, recall our notation for the asymmetry ratio $\mu^2=q/p$, which we scale as $\mu^2=e^{-2\eps}$ (the asymmetry ratio was denoted by the letter $q$ in \cite{barraquand2022markov} but this choice would be confusing in the present paper). 

Fix $n\in \N$. Moments of order $n$ will be written as a sum over integer partitions $\lambda$ of $n$, \emph{i.e.}~nonincreasing sequences of integers  $(\lambda_1, \lambda_2, \dots, \lambda_{\ell(\lambda)})$ such that  $\sum_{i=1}^{\ell(\lambda)}\lambda_i =n$. The fact that a partition $\lambda$ is a partition of the integer $n$ is denoted $\lambda\vdash n$.  The number $\ell(\lambda)$ is called the length of the partition $\lambda$. 

For a given partition $\lambda\vdash n$,  we will further consider diagrams formed by  $\ell(\lambda)$ lines of numbers separated by arrows, of the following form:  
\begin{eqnarray*}
	i_1 &\rightarrowplus i_2\rightarrowplus \dots \rightarrowplus i_{r_1-1} \rightarrowplus &i_{r_1}\leftarrowminus i_{r_1+1} \leftarrowplus \dots\leftarrowplus i_{\lambda_1}\\
	j_1 &\rightarrowplus j_2\rightarrowplus \dots \rightarrowplus j_{r_2-1} \rightarrowplus &j_{r_2}\leftarrowminus j_{r_2+1} \leftarrowplus \dots\leftarrowplus j_{\lambda_2} \\
	&&	\vdots   
\end{eqnarray*}
The diagrams are such that 
\begin{itemize}
	\item all numbers $i_1,i_2,\dots, j_1,j_2,\dots$ 
	are distinct elements of $\lbrace 1, \dots,n\rbrace$ (so that every integer in $\lbrace 1, \dots,n\rbrace$ appears exactly once in the diagram); 
	\item arrows go from a larger integer to a smaller one;
	\item on each line, arrows change from pointing right to pointing left at most once. Equivalently, the sequence of numbers is decreasing, until a certain minimum, and then it is increasing. We denote by $r_i$ the index of the minimal number on the $i$th line, around which arrows change orientation. The numbers  $r_1,r_2, \dots$ are denoted $\mu_1,\mu_2, \dots$ in \cite{barraquand2022markov} but we changed the notation here to avoid ambiguity with the asymmetry parameter $\mu=\sqrt{q/p}$;
	\item all arrows bear a plus sign, except for the first arrow pointing left in each row which bears a minus sign. It may also happen that on some row, $r_i=\lambda_i$, i.e. the minimal number is in the rightmost position, in which case all arrows point right and bear a plus sign. 
\end{itemize}
 We denote by $S(\lambda)$ the set of all such diagrams and we will use the letter $I$ to denote one such diagram. We refer to \cite[Section \mm{2.5.2}]{barraquand2022markov} for details about these definitions and examples of diagrams. The precise structure of arrows described above will not matter that much in the arguments below, but is is important to note that $S(\lambda)$ is a finite set for any $\lambda$.

 Now, given a function $f$ of $n$ complex variables $z_1, \dots, z_n$,  we will define a procedure that  associates any diagram $I\in S(\lambda), \lambda\vdash n$ with a certain multivariate residue of the function $f$. For a diagram $I\in S(\lambda)$ with $\lambda\vdash n$, we will compute residues in as many variables as the number of arrows in the diagram $I$ (that is $n$ minus the number of lines in the diagram), according to the following rules:   
\begin{itemize}
	\item for each arrow of the form $a\rightarrowplus b$ or $b\leftarrowplus a$ we compute the residue in the variable $z_a$ at the point $\mu^2 z_b$. Assuming that $f(z_a,z_b)$ is a rational function in $z_a$ with a simple pole at $\mu^2z_b$, which will be the case below, the residue of $f$ in the variable $z_a$ at the point $\mu^2z_b$ is simply given by $\mathrm{Sub}_{z_a\to \mu^2z_b}\left\lbrace (z_a-\mu^2z_b)f(z_a,z_b)\right\rbrace$, where $\mathrm{Sub}$ denotes the substitution of a variable by a certain expression, i.e. for any rational function $g(z)$,  $\mathrm{Sub}_{z\to w}\left\lbrace g(z)\right\rbrace=g(w)$; 
	\item for each arrow of the form $a\leftarrowminus b$ we compute the residue in the variable $z_b$ at $1/z_a$. Again, assuming that the pole is simple, which will be the case below, the residue in the variable $z_b$ at the point $1/z_a$ of a rational function $f(z_a,z_b)$ is simply given by $\mathrm{Sub}_{z_b\to 1/z_a}\left\lbrace (z_b-1/z_a)f(z_a,z_b)\right\rbrace$).
\end{itemize} 
We denote such a  multi-residue  by $\Res{I} \left\lbrace  f(\vec z)\right\rbrace$. Again, we refer to \cite[Section \mm{2.5.2}]{barraquand2022markov} for many details about this procedure and justifications that it is well-defined. After computing all those residues, the quantity $\Res{I} \left\lbrace  f(\vec z)\right\rbrace$ only depends on  the variables $z_j$ for each $j$ being the minimal number on one of the lines of the diagram $I$. In other terms, $\Res{I} \left\lbrace  f(\vec z)\right\rbrace$ is a function of the variables $z_{i_{r_1}}, z_{j_{r_2}}, \dots$ Eventually, there are as many variables left as the number of lines in the diagram, that is the length of the partition $\lambda$. 

To be more concrete, let us consider the partition $\lambda=(n)$ and the diagram $I$ formed by a single line with only arrows bearing a plus sign
$$I\;\;=\;\; n \rightarrowplus n-1 \rightarrowplus \dots \rightarrowplus 2 \rightarrowplus 1 $$
Then, for a rational function $f$ of variables $z_1, \dots, z_n$, in order to compute $\Res{I} \left\lbrace  f(\vec z)\right\rbrace$, we will first multiply the function $f$ by the products 
\begin{equation}
	(z_n-\mu^2z_{n-1})(z_{n-1}-\mu^2z_{n-2})\dots (z_2-\mu^2z_1) 
	\label{eq:products}
\end{equation}
and then evaluate the rational function obtained at the point 
$$ (z_1, \dots, z_n)= (z_1, \mu^{2}z_1, \dots, \mu^{2(n-1)}z_1)$$
This indeed corresponds to the usual notion of complex residue provided that the denominator of the rational function $f$ is divisible exactly once by each of the factors appearing in \eqref{eq:products}. Let us  also consider the slightly more general situation where for some $1\leq j\leq n-1$, 
$$I\;\;=\;\; j \rightarrowplus j-1 \rightarrowplus \dots \rightarrowplus 2 \rightarrowplus 1\leftarrowminus j+1 \leftarrowplus \dots\leftarrowplus n $$
In this case, in order to compute $\Res{I} \left\lbrace  f(\vec z)\right\rbrace$, we multiply the function $f$ by the products 
\begin{equation}
	(z_j-\mu^2z_{j-1})(z_{j-1}-\mu^2z_{j-2})\dots (z_2-\mu^2z_1) (z_{j+1}-1/z_1)(z_{j+2}-\mu^2z_{j+1})\dots (z_n-\mu^2z_{n-1})
	\label{eq:products2}
\end{equation}
and then evaluate the rational function obtained at the point 
$$ (z_1, \dots, z_n)= (z_1, \mu^{2}z_1, \dots, \mu^{2(j-1)} z_1, 1/z_1, \mu^2/z_1, \mu^{2(n-j-2)}/z_1)$$ 
Again, it corresponds to the usual notion of complex residue in all the cases considered below. 
We will consider sums over diagrams of contour integrals of the form 
\begin{multline}
	\mathcal I_n[\phi(\vec x, \vec z)]_{\mathcal C}=\mu^{n(n-1)}  \sum_{\lambda \vdash n} \frac{(-1)^{n-\ell(\lambda)}}{m_1!m_2!\dots}\sum_{I\in S(\lambda)} \oint_{\mathcal C} \frac{dz_{i_{r_1}}}{2\I\pi}  \oint_{\mathcal C} \frac{dz_{j_{r_2}}}{2\I\pi} \dots \; \\ \Res{I} \left\lbrace  \prod_{i<j} \frac{z_i-z_j}{\mu^2 z_i-z_j} \frac{\mu^{-2}-z_iz_j}{1-z_iz_j}\prod_{j=1}^n\frac{1}{z_j}\;\;\phi(\vec x, \vec z)\right\rbrace,
	\label{eq:notationcalI}
\end{multline} 
where $\mathcal C$ will be a certain integration contour in the complex plane and the function $\phi:\mathbb Z^n\times \mathbb C^n\to \mathbb C$ will be holomorphic in a neighborhood of $\mathcal C^n$, for all $\vec x$.

We are now able to state
\cite[Theorem 2.13]{barraquand2022markov}, which, translated into the notations of the present paper, states that for $0\leq  k_1 <\dots <k_n$, 
\begin{equation} \mathbb E\left[\prod_{i=1}^n Z_t(k_i) \right]=\mathcal I_n\left[ \prod_{j=1}^nG_{k_j}(z_j)  \right]_{\mathcal C}
	\label{eq:momentZ}
\end{equation} 
where $\mathcal C$ is the positively oriented circle of radius $1/\mu$ around $0$ (the formula holds as long as the radius is comprised between $1$ and $1/\mu^2$), and 
\begin{equation} G_k(z) = \frac{p-q z^2}{p-p z}  \exp\left( \frac{(\sqrt{p}-\sqrt{q})^2(\sqrt{p}+\sqrt{q} z)^2  t}{(1-z)(p-q z)}\right) \left( \frac{\sqrt{pq}(1-z)}{p-q z} \right)^{k+1}.
	\label{eq:defGx}
\end{equation}

\subsubsection{Proof of Proposition \ref{p:momentZ}}
\label{subs:proof}
In this section, we prove \eqref{eq:Zp}. By the definition of $Z_t(k)$, we have $Z_t(k)/Z_t(k+1)\leq \mu^{-1}$, so we can write
$$ \mathbb E\left[(\eps^{-2} Z_t(k))^n\right] \leq C\oo{(n,\eps_0)} \,\mathbb E\left[ \eps^{-2n} Z_t(k)Z_t(k+1)\dots Z_t(k+n-1)\right].$$
The right hand side can be evaluated using \eqref{eq:momentZ}. 

We will scale $p$ and $q$ as in \eqref{eq:pq},  let $t=\eps^{-4}\tau$, and  
 estimate each of the factors appearing in \eqref{eq:momentZ}. Since we are taking many residues in \eqref{eq:momentZ} we not only need to estimate the functions when the variables $z_i$ vary over the contour  $\mathcal C$ but also when the variables vary in $\mu^2\mathcal C \mathcal, \dots \mu^{2(n-1)} \mathcal C$ and $\mu^2/\mathcal C, \dots \mu^{2(n-1)}/\mathcal C$. 

We observe that when $z$ varies over a circle around $0$ with radius at most $(1+\mu^{-2})/2$, the modulus of $\big| \frac{\sqrt{pq}(1-z)}{p-q z} \big|$ is maximal when $z$ has minimal real part. In particular, if $z$ varies over a circle of radius comprised between $\mu^{-1}$ and $\mu^{2n-1}$,
\begin{equation}
	\left\vert  \frac{\sqrt{pq}(1-z)}{p-q z} \right\vert 
	\leqslant 1.
	\label{eq:boundx}
\end{equation}
To estimate the other factors, it will be convenient to use the change of variables $z=-1/\mu \mu^{2w}$.
The rational factors in the formula are easily bounded by a constant. More precisely, we have   (we correct a power of $\eps$ which is missing in  \cite[(4.12)]{barraquand2022markov})
\begin{equation}
	\left\vert \Res{I} \left\lbrace  \prod_{i<j} \frac{z_i-z_j}{\mu^2 z_i-z_j} \frac{\mu^{-2}-z_iz_j}{1-z_iz_j}\right\rbrace \right\vert <C \eps^{n-\ell},
	\label{eq:boundRes}
\end{equation} 
where $\ell$ is the number of lines in the diagram $I$. The power of $\eps$ will exactly cancel factors of $\eps$ coming from the Jacobian arising in \eqref{eq:notationcalI} after the change of variables.
  We can find a constant $C$ so that 
\begin{equation}
	\left\vert \frac{p-q z^2}{p-p z} \right\vert \leq C\eps\vert w\vert.
	\label{eq:boundalone}
\end{equation}
Then, to estimate the exponential factor, we need to control the real part of the function 
\begin{equation}  \frac{(\sqrt{p}+\sqrt{q} z)^2  \eps^{-2}}{(1-z)(p-q z)} = \frac{(1-e^{-2\eps w})^2\eps^{-2}}{(1+e^{-\eps(2w-1)})(1+e^{-\eps(2w+1)})}=:H(w).
	\label{eq:defh}
\end{equation}
Under the change of variables that we are considering, the contour $\mathcal C$ for $z$ becomes $\I[\frac{-\pi}{2} \eps^{-1}, \frac{\pi}{2} \eps^{-1}]$ for $w$. Since the formula in \eqref{eq:momentZ} involves taking residues, we not only need to estimate $\Re[H(w)]$ along the contour $\mathcal C$, but also estimate  expressions of the form $\Re[H(w_1)]+\dots+ \Re[H(w_k)]$, where $\vec w$ is of the form $(w+k-1, w+k-2, \dots, w)$ or of the form  $(w+k-1, \dots, w, 1-w, 2-w, \dots, \ell-w)$ for some integers $k,\ell\geq 0$ such that $k+\ell\leq n$, and $w\in \I\R$. 
\begin{lemma}[{\cite[Lemma 4.7]{barraquand2022markov}}]
	Fix positive integers $k,\ell\leq n$ such that $k+\ell\leq n$. There exist a constant $C$ and $ \eps_0$ so that uniformly for $\eps\in (0, \eps_0)$, for all $y\in [\frac{-\pi}{2} \eps^{-1}, \frac{\pi}{2} \eps^{-1}]$, 
	\begin{equation}
		\sum_{i=0}^{k-1}\Re[H(\I y+i)]\leq  C (C-y^2); 
		\label{eq:lemma}
	\end{equation}
	and for a sequence $(w_1, \dots, w_{k+\ell})$ of the form $(w+k-1, \dots, w, 1-w, 2-w, \dots, \ell-w)$ with $w=\I y$, 
	\begin{equation}
		\sum_{i=1}^{k+\ell}\Re[H(w_i)]\leq  C (C-y^2).
		\label{eq:lemma2}
	\end{equation}
	\label{lem:quadraticdecay}
\end{lemma}

Using Lemma \ref{lem:quadraticdecay} and the bounds \eqref{eq:boundx},  \eqref{eq:boundRes} and \eqref{eq:boundalone} above, the integrand of each term in the sum \eqref{eq:momentZ} can be dominated by a function of the form $Ce^{\tau C(C-y^2)}$ so that we have obtained 
\begin{equation*}
	 \mathbb E\left[(\eps^{-2} Z_t(k))^n\right] \leq \\  \sum_{\lambda \vdash n} \frac{1}{m_1!m_2!\dots}\sum_{I\in S(\lambda)} \int_{-\eps^{-1}\pi/2}^{\eps^{-1}\pi/2} \frac{dy_1}{2\pi} \dots  \int_{-\eps^{-1}\pi/2}^{\eps^{-1}\pi/2}\frac{dy_{\ell(\lambda)}}{2\pi} C \prod_{i=1}^{\ell(\lambda)} e^{C\tau (C-y_i^2)}.
\end{equation*}
Since $\tau \in [\delta,T]$, the above finite sum of integrals can be bounded by a constant (depending on $n, \delta, T$). This concludes the proof of \eqref{eq:Zp}.

\subsection{H\"older estimates for $\mathcal{Z}^\eps$}
\label{subs:neareq}

We now  prove the following H\"older estimates which, together with Proposition \ref{p:momentZ}, will be essential to prove the forthcoming tightness result.

\begin{proposition}\label{prop:HolderZdelta'}
	Fix $\delta>0$ and  $T\geqslant \delta$. Let $n \in \N$. For any $\alpha\in[0,\frac12)$ there exists $C =C(\alpha,\delta, T,n)\in\R_+$ such that, for \mm{$\eps$ small enough}, for any $k,k' \in \N$, and $t,t' \in [\eps^{-4}\delta,\varepsilon^{-4}T]$:
	\begin{align}
		\|\varepsilon^{-2} (Z_t(k)-Z_{t}(k'))\|_{2n} & \leqslant C(\eps^2|k-k'|)^{\alpha} . \label{eq:Zxxprime}
	\end{align}
and 
\begin{equation}
	\|\varepsilon^{-2} (Z_t(k)-Z_{t'}(k)) \|_{2n}  \leqslant  C  \; \varepsilon^{2\alpha} (1 \vee |t-t'|^{\alpha/2}).   \label{eq:Zttprime}
\end{equation}
\end{proposition}
\begin{proof}
\oo{With the moment estimate \eqref{eq:Zp} at hand, the proof is very similar in spirit to that of Lemma~\ref{l:Ztbounds-cont}, but we need to control more tightly $\overline{F}$ when $F$ is the difference of two heat kernels (Lemma~\ref{l:integralF} below).}

First, let us choose $\eps$ small enough such that $\eps^{-4} \delta \ge 1$. We start the proof of \eqref{eq:Zxxprime} by \oo{writing Duhamel's formula between times $\eps^{-4}\delta$ and $t$: 
\begin{equation}
 Z_t(k)=\sum_{j=0}^\infty \p^\eps_t(k,j)Z_{\eps^{-4}\delta}(j)+\int_{\eps^{-4}\delta}^t\sum_{j=0}^\infty \p^\eps_{t-s}(k,j)dM_s(j).
\end{equation}
 We get, using Lemma~\ref{l:iteration} and Proposition~\ref{p:momentZ},}
\begin{align*}
		\big\|\varepsilon^{-2}(Z_t(k)-Z_t(k'))\big\|_{\mm{2n}} &\leqslant \big|  (\p_t^\varepsilon \ast \varepsilon^{-2}Z_{\oo{\eps^{-4}}\delta})(k)-(\p_t^\varepsilon \ast \varepsilon^{-2}Z_{\oo{\eps^{-4}}\delta})(k')\big|  \\
		&\qquad+ \bigg\|\varepsilon^{-2} \int_{\oo{\eps^{-4}}\delta}^t \sum_{j\in\N} \big(\p_{t-s}^\varepsilon (k,j)-\p_{t-s}^\varepsilon(k',j)\big) d M_s(j)\bigg\|_{\mm{2n}}\\
		&\leq C\sum_{j\in\N}\left|\p^\eps_t(k,j)-\p^\eps_t(k',j)\right| +  C\Bigg(\eps^2\int_{{\eps^{-4}\delta}}^t\sum_{j\in\N}\bar F(s,j)^2 ds\Bigg)^{1/2},
	\end{align*}
where $F(s,j)=\p^\eps_\oo{t-s}(k, j)-\p^\eps_\oo{t-s}(k',j)$. Now, 
 for $\eps^2|k-k'|\geqslant 1$, \eqref{eq:Zxxprime} is a consequence of the moment bound \eqref{eq:momentZ}, while for $\eps^2|k-k'|<1$, \eqref{eq:Zxxprime} is a consequence of \eqref{eq:diffheatkerneleps} (which bounds the first term) and the following lemma.

\begin{lemma}\label{l:integralF} Let $t\ge \eps^{-4}\delta\ge 1$ and let $F$ defined as above. Then, there exists a universal constant $C\in\R_+$ such that, if $k\neq k'$, 
\begin{align} 
 \eps^2\int_{\eps^{-4}\delta}^t\sum_{j\in\N}\bar F(s,j)^2 ds&\leq C\eps^2|k-k'|\left(1+\ln(t/|k-k'|)\vee 0\right).\label{e:integralFbar}
\end{align}
\end{lemma}
\noindent We postpone the proof of this lemma to the end of the section. 

\medskip

Let us now turn to \eqref{eq:Zttprime}. \mm{Once again, when $|t-t'| \ge 1$, from Duhamel's formula} we write the martingale decomposition between times $t'$ and $t$: 
\begin{equation} Z_t(k)  = \sum_{j\geqslant 0} \p_{t-t'}^\varepsilon(k,j)Z_{t'}(j) + \int_{t'}^{t} \sum_{j\geqslant 0} \p_{t-s}^\varepsilon (k,j) d M_s(j). \label{eq:Zmart}
\end{equation}
Thus, we get
\begin{equation} 
	\big\|\varepsilon^{-2}(Z_t(k)-Z_{t'}(k))\big\|_{2n}  \leqslant \big\|  (\p_{t-t'}^\varepsilon \ast \varepsilon^{-2}Z_{t'})(k)- \varepsilon^{-2}Z_{t'}(k)\big\|_{2n} + \bigg\|\varepsilon^{-2} \int_{t'}^t \sum_{j\geqslant 0} \p_{t-s}^\varepsilon (k,j) d M_{s}(j)\bigg\|_{2n}.
	\label{eq:estimeedeuxtemps}
\end{equation} 
We denote the terms on the right-hand side above by $I_1$ and $I_2$ respectively. \oo{As in Section~\ref{s:tightnessneareq}, we have, using \eqref{eq:Zp}, \eqref{eq:Zxxprime} and \eqref{eq:killingestimate}:}
\begin{align*}
	I_1 & = \bigg\|  \varepsilon^{-2} \mm{\bigg(}\sum_{j\geqslant 0} \p_{t-t'}^\varepsilon(k,j)Z_{t'}(j)-Z_{t'}(k) \mm{\bigg)}\bigg\|_{2n} \\
	& \leqslant  \bigg\|  \varepsilon^{-2} \sum_{j\geqslant 0} \p_{t-t'}^\varepsilon(k,j)\big(Z_{t'}(j)-Z_{t'}(k)\big)\bigg\|_{2n} + \bigg|\sum_{j\geqslant 0} \p^\varepsilon_{t-t'}(k,j)-1\bigg| \; \big\|\varepsilon^{-2} Z_{t'}(k)\big\|_{2n}\\
	& \leqslant  C   \varepsilon^{2\alpha} \oo{|t-t'|^{\alpha/2}} + C\exp\bigg\{-\frac{\eps^2 k}{\sqrt{\eps^4(t-t')}}\bigg\}   .
\end{align*}
	 In the last line, we have used the same proof as for 
	\cite[Eq. (34)]{parekh2017kpz}: since $r^\alpha\leq e^r$ for $\alpha\leqslant 1$,
	\begin{align}
		\sum_{j\geqslant 0} \p_{t-t'}^\varepsilon(k,j)  |k-j|^\alpha &\leqslant\sum_{j\geqslant 0} \p_{t-t'}^\varepsilon(k,j)  e^{|k-j|}\nonumber\\
		&\leqslant \sum_{j\geqslant 0} \p_{t-t'}^\varepsilon(k,j) \oo{ |t-t'|^{\alpha/2}}e^{|k-j|(1\wedge |t-t'|^{-\alpha/2})}\nonumber\\
		&\leqslant C \oo{ |t-t'|^{\alpha/2 }},
		\label{eq:Holderheatkernel} 
	\end{align}
	using \eqref{e:hkbound1} with $a_1=1$, $a_2=0$.

As in Section~\ref{s:tightnessneareq}, $I_2$ is bounded using \eqref{eq:Zp} and both estimates \eqref{eq:it1} and \eqref{eq:it2} from Lemma \ref{l:iteration}:
\[ I_2 \leqslant C  \varepsilon(t-t')^{1/4} \leqslant C'  \varepsilon^{2\alpha} |t-t'|^{\alpha/2}. \]
As in Section~\ref{s:tightnessneareq} we now split cases and deduce \eqref{eq:Zttprime} for $\eps^2k \geqslant \left(\eps^4(t-t')\right)^{\sqrt{\alpha/2}}$. 
If instead $\eps^2k<\left(\eps^4(t-t')\right)^{\sqrt{\alpha/2}}$, we define $k':=\lceil\eps^{-2}(\eps^4(t-t'))^{\sqrt{\alpha/2}}\rceil$ and use the triangle inequality, \eqref{eq:Zxxprime} between $k$ and $k'$ and \eqref{eq:Zttprime} for $k'$ to prove \eqref{eq:Zttprime} for $k$.
\end{proof}

\mm{It remains to prove Lemma \ref{l:integralF}.}

\begin{proof}[Proof of Lemma~\ref{l:integralF}] Recall $F(s,j)=\p^\eps_s(k, j)-\p^\eps_s(k',j)$.
 We first prove that
 \begin{equation}\label{e:integralF2}
  \eps^2\int_0^t\sum_{j\in\N} F(s,j)^2 ds\leq \frac{20}{\pi}\eps^2|k-k'|\left(1+\ln(t/|k-k'|)\vee 0\right) .
 \end{equation}
Using the semigroup property, we can evaluate
\begin{align}
 \sum_{j\in\N} F(s,j)^2&=\sum_{j\in\N}(\p_{t-s}^\eps(k,j)-\p^\eps_{t-s}(k',j))^2\\
 &=\p^\eps_{2(t-s)}(k,k)+\p^\eps_{2(t-s)}(k',k')-2\p^\eps_{2(t-s)}(k,k').
\end{align}
Using \eqref{eq:heatkernelRobinintegral}, this can be expressed as 
\begin{multline}
 \frac{1}{2i\pi}\oint \left[\xi^k\Big(\xi^{-k}+\xi^{k+1}\frac{\mu-\xi}{1-\mu\xi}\Big)+\xi^{k'}\Big(\xi^{-k'}+\xi^{k'+1}\frac{\mu-\xi'}{1-\mu\xi'}\Big)-2\xi^k\Big(\xi^{-k'}+\xi^{k'+1}\frac{\mu-\xi}{1-\mu\xi}\Big)\right]\\
 \times e^{(\xi+\xi^{-1}-2)(t-s)}\frac{d\xi}{\xi}\\
 =  \frac{1}{2i\pi}\oint e^{(\xi+\xi^{-1}-2)(t-s)}\left[2(1-\xi^{k-k'})+\frac{\mu-\xi}{1-\mu\xi}\xi^{2k'+1}(1-\xi^{k-k'})^2\right]\frac{d\xi}{\xi}.
\end{multline}
This is a non-negative quantity, which can be bounded by 
\begin{align}
\frac{2}{\pi}\oint e^{(\xi+\xi^{-1}-2)(t-s)}\left|1-\xi^{k-k'}\right|\frac{d\xi}{\xi} &=\frac{4}{\pi}\int_{-\pi}^\pi e^{2(\cos\theta -1)(t-s)}\left|\sin\left(\frac{k-k'}{2}\theta\right)\right|d\theta .
\end{align}
Integrating over $s$, we get
\begin{align}
\eps^2\int_0^t\sum_{j\in\N} F(s,j)^2 ds\leq \frac{4}{\pi}\eps^2 \int_0^\pi\left|\sin\left(\frac{k-k'}{2}\theta\right)\right|\frac{e^{\mm{2}(\cos\theta-1)t}-1}{\mm{2}(\cos\theta -1)}d\theta.
\label{eq:boundwithsine}
\end{align}
 Let us denote $I$ the last integral and rewrite $\eps^{2}(x-x')=2X$, $t=\eps^{-4}\tau$. Then,
\begin{align}
 I=\int_0^\pi|\sin(\eps^{-2}X\theta)|\frac{1-e^{(\cos\theta-1)\eps^{-4}\tau}}{1-\cos\theta }d\theta.
\end{align}
On $[0,\pi]$, $-\theta^2/2\leq\cos\theta -1\leq -\theta^2/5$, so that
\begin{align}
 I&\overset{{\color{white}\theta'=\theta\eps^{-2}X}}{\leq} 5\int_0^\pi|\sin(\eps^{-2}X\theta)|\frac{1-e^{-\eps^{-4}\tau\theta^2/2}}{\theta^2}d\theta\\
 &\overset{\theta'=\theta\eps^{-2}X}\leq 5\eps^{-2}X\int_0^{\pi\eps^{-2}X}|\sin\theta'|\frac{1-e^{-\tau\theta'^2/(2X^2)}}{\theta'^2}d\theta'\\
 &\overset{|\sin\theta|\leq\theta\wedge 1}{\leq} 5\eps^{-2}X\left[\int_0^1\frac{1-e^{-\tau\theta^2/(2X^2)}}{\theta}d\theta+\int_1^\infty\frac{1-e^{-\tau\theta^2/(2X^2)}}{\theta^2}d\theta\right].
\end{align}
The second integral is less that $\int_1^\infty\theta^{-2}d\theta=1$. The first integral can be computed as 
\begin{equation}
 \int_0^\frac{\sqrt{\tau}}{X\sqrt{2}}\frac{1-e^{-\theta^2}}{\theta}d\theta\leq 1+\int_1^{\frac{\sqrt{\tau}}{X\sqrt{2}}\vee 1}\theta^{-1}d\theta=1+\ln\left(\frac{\sqrt{\tau}}{X\sqrt{2}}\right)\vee 0.
\end{equation}
\eqref{e:integralF2} follows. We now turn to the proof of \eqref{e:integralFbar}. We may treat separately the integration over $[0,t-1]$ and the integration over $[t-1,t]$. Let us first focus on on 
\begin{multline}
	\varepsilon^2\int_0^{t-1} \sum_{j \in \N}\bigg( \sup_{\vert s'-s\vert \leqslant 1} \left\lbrace \left\vert  \p^\eps_{t-s'}(k,j)-\p^\eps_{t-s'}(k',j)\right\vert  \right\rbrace  \bigg)^2ds \\
	= \varepsilon^2\int_0^{t-1} \sum_{j \in \N}\left( \sup_{0\leqslant u\leqslant 2} \left\lbrace \left\vert   \p^\eps_{r+u}(k,j)-\p^\eps_{r+u}(k',j) \right\vert \right\rbrace  \right)^2dr
	\label{eq:quantitytobound2}
\end{multline}
where in the right hand side, we have used the change of variables $r=t-s-1$. 
We may write 
\begin{align}
	\sup_{0\leqslant u\leqslant 2} \Big\lbrace \left\vert \p^\eps_{r+u}(k,j)-\p^\eps_{r+u}(k',j) \right\vert  \Big\rbrace  &\leq \sup_{0\leqslant u\leqslant 2} \left\lbrace  \sum_{\ell=0}^{+\infty} \left\vert \p^\eps_{r}(k,\ell)-\p^\eps_{r}(k',\ell)\right\vert \p^\eps_u(\ell,j)  \right\rbrace\\
	&\leqslant  \sum_{\ell=0}^{+\infty} \vert \p^\eps_{r}(k,\ell)-\p^\eps_{r}(k',\ell)\vert \sup_{0\leqslant u\leqslant 2} \left\lbrace \p^\eps_u(\ell,j)  \right\rbrace.
\end{align}
There exists a function $f:\Z\to \mathbb R_+$ with $\sum_{n\in \N} f(n) <+\infty$ such that 
\begin{equation}
	 \sup_{0\leqslant u\leqslant 2} \left\lbrace \p^\eps_u(\ell,j) \right\rbrace \leqslant f( \ell-j).
\end{equation}
Indeed, it suffices to bound $\p^\eps_u(\ell,j)$ by the probability that a Poisson variable with parameter $2$ is larger than $\vert \ell-j\vert$. It will be convenient to write $f(k)=C \mathbf P(k)$ where $\mathbf P$ is a probability distribution. We will denote by $Z$ a random variable with distribution $\mathbf P$ and write $\mathbf E$ the associated expectation. 

We have 
\begin{align}
		\sup_{0\leqslant u\leqslant 2} \lbrace  \left|\p^\eps_{r+u}(k,j)-\p^\eps_{r+u}(k',j)  \right|\rbrace &\leqslant C \sum_{\tilde\ell=-j}^{+\infty} \vert \p^\eps_{r}(k,j+\tilde\ell)-\p^\eps_{r}(k',j+\tilde\ell)\vert \mathbf P(\tilde\ell)\\
		&\leqslant C \sum_{\tilde \ell=-\infty}^{+\infty} \vert \p^\eps_{r}(k,j+\tilde \ell)-\p^\eps_{r}(k',j+\tilde \ell)\vert \mathbf P(\tilde \ell)\\
		&\leqslant C \mathbf E\Big[  \vert \p^\eps_{r}(k,j+Z)-\p^\eps_{r}(k',j+Z)\vert \Big],\label{e:boundpbyZ}
\end{align} 
where in the second line, we have extended the summation over all $\mathbb Z$ using the convention that $\p^\eps_s(k,j)=\p^\eps_s(k',j)=0$ when $j<0$. 
Now, using Jensen's inequality and the linearity of the expectation, we have 
\begin{equation}
\eqref{eq:quantitytobound2} \leqslant C^2 \mathbf E\Bigg[  \varepsilon^2 \int_0^{t-1} \sum_{j\in \Z}\left(   \p^\eps_{r}(k,j+Z)-\p^\eps_{r}(k',j+Z)   \right)^2ds  \Bigg], 
\end{equation}
where, again, we have extended the sum over $j\in \Z$ for notational convenience. Inside the expectation, we may use \eqref{e:integralF2} and we conclude that 
 \begin{equation}
 	\eqref{eq:quantitytobound2} \leqslant \frac{20C^2}{\pi}\eps^2|k-k'|\left(1+\ln(t/|k-k'|)\vee 0\right) .
 \end{equation}
 It remains to deal with the integral from $t-1$ to $t$:
 \begin{multline}
	\varepsilon^2\int_{t-1}^{t} \sum_{j \in \N}\left( \sup_{\vert s'-s\vert \leqslant 1} \left\lbrace \left\vert  \p^\eps_{t-s'}(k,j)-\p^\eps_{t-s'}(k',j)\right\vert  \right\rbrace  \right)^2ds \\
	\leq \varepsilon^2 \sum_{j \in \N}\left( \sup_{0\leqslant u\leqslant 2} \left\lbrace \left\vert   \p^\eps_{u}(k,j)-\p^\eps_{u}(k',j) \right\vert \right\rbrace  \right)^2\leq C\eps^2,
	\label{eq:quantitytobound3}
\end{multline}
where we used \eqref{e:boundpbyZ} with $s=2$ for the last bound. 
\end{proof}

\subsection{Construction and properties of the limit point} \label{sec:limitZ}

We are now able to obtain tightness of $\{\rZ_t\}$, and identify the law of any of its limit point, as follows:

\begin{proposition}[Tightness]\label{prop:tight2}
	For any $0<\delta \leqslant \tau$, the laws of $\{\rZ\}$ are tight on the Skorokhod space $D([\delta,\tau],C(\R_+))$, and moreover, any limit point $\P$ is an element of \break $C([\delta,\tau],C(\R_+))$. 
	
	For any $\theta\in[\delta,\tau]$, define $\cL_\theta:C([\delta,\tau],C(\R_+)) \to C(\R_+)$ as the \emph{evaluation map} at time $\theta$. Then, the process $\{\cL_{\theta+\delta}\; ; \; \theta \in [0,\tau-\delta]\}$ has the same distribution under $\P$ as the solution of the stochastic heat equation \eqref{eq:she} as defined in Definition \ref{de:solution}, with initial condition $\mathscr{Z}_{\rm ini}$ whose distribution is the same as the one of $\mathcal{L}_{\delta}$ under $\P$.
\end{proposition}

\begin{proof}
	The argument is exactly the same as in \cite[Proof of Corollary 6.3]{parekh2017kpz}. The tightness property is based on Propositions \ref{p:momentZ} and \ref{prop:HolderZdelta'}, together with Arzela-Ascoli's Theorem. Continuity of limit points follow from Proposition~\ref{prop:HolderZdelta'} and Kolmogorov's continuity criterion. Finally, the identification of limit points follows from the same arguments as in Section~\ref{s:limitpoints}, replacing $\scZ$ by $\rZ$ in \eqref{eq:martingaleepsilon} (\oo{Propositions~\ref{p:momentZ}} and \ref{prop:HolderZdelta'} replacing Lemma~\ref{l:Ztbounds-cont} in the control of the error terms).
\end{proof}

The next step consists in defining a limit point in $C((0,+\infty),C(\R_+))$ (Lemma \ref{lem:limit2} below) and identifying the initial condition. This means showing that the Duhamel form of the SHE \eqref{eq:Duhamelform} is satisfied for all $t>0$, with $\mc Z_{\mathrm{ini}}=-2\delta_0'$,  that is, 
\begin{equation}
	\mc Z_t(x) = dP^{\rm Dir}_t(x,0)   + \int_0^t ds \int_{\mathbb R_+}dy P^{\rm Dir}_{t-s}(x,y) \mc Z_s(y)\xi(s,y). 
	\label{eq:Duhamelformdelta'}
\end{equation}

\begin{lemma}\label{lem:limit2}Let $\bQ^\eps$ denote the law of $\rZ$ on $C((0,+\infty),C(\R_+))$. Then, there exists a measure $\bQ$ on $C((0,+\infty),C(\R_+))$ which is a limit point of the sequence $\{\bQ^\eps\}_\eps$ on $D((0,+\infty),C(\R_+))$. 
\end{lemma}
\begin{proof}
	This can be proved exactly as in \cite[Lemma 6.5]{parekh2017kpz} thanks to Kolmogorov's extension Theorem. 
\end{proof}

The identification of the initial condition follows a general argument, given in  \cite{parekh2017kpz}, based on estimates on the first two moments given in the next two lemmas, which however are specific to the $\delta_0'$ initial condition and use arguments different from \cite{parekh2017kpz}. We start with some exact computation for the first moment. 
\begin{lemma} \label{lem:limitmoment} Let  $\{\mathcal{Z}_t\}$ be distributed according to the measure $\mathbb Q$ defined in Lemma \ref{lem:limit2}.  
	For any $t> 0$,  and $x\in \mathbb R_+$, we have 
	\begin{equation}
		\lim_{\eps\to 0 }\mathbb E\left[  \rZ_{t}(x) \right] = \mathbb E\left[\mathcal Z_t(x)\right]  = d P_t^{\rm Dir}(x,0)  =\frac{1}{2 \pi} \int_{-\infty}^{+\infty} e^{\frac{-t}{2}\theta^2 +\mathbf i\theta x}(-4 \mathbf i \theta)d\theta,
		\label{eq:expectationinitialdata}
	\end{equation}
	where we recall that $d P_t^{\rm Dir}(x,0)$ is defined in \eqref{eq:defdP}. 
	\label{lem:expectationinitialdata}
\end{lemma}
\begin{proof}
	Using the same steps as in the proof of Proposition \ref{prop:boundinitialcondition} -- see \eqref{eq:expectationsolution} in particular,
	\begin{align}
		\mathbb E\left[  \rZ_{t}(x) \right] &= 	\eps^{-2}  \sum_{\ell\ge 0} \p_{t\eps^{-4}}^	\eps(\lfloor\eps^{-2}x\rfloor,\ell) \mu^\ell \notag\\ & = \frac{1}{2\pi}\int_{-\eps^{-2}\pi}^{\eps^{-2}\pi}   e^{\frac{1}{2}\left(e^{\I\eps^2 \theta}+e^{-\I\eps^2\theta}-2\right) \eps^{-4}t} e^{\I \theta  x} \frac{(1-e^{-2\eps})(1-e^{2\I\eps^2\theta})}{(1-e^{-\eps-\I\eps^2\theta})(1-e^{-\eps+\I\eps^2\theta})^2}  d\theta. 
		\label{eq:expectationsolution2}
	\end{align}
	Using the estimate \eqref{eq:eigenvalue}, we see that the exponential term in the integrand of \eqref{eq:expectationsolution2} is dominated by $e^{-2\theta^2 t/5}$, so that we may apply dominated convergence to take the $\epsilon\to 0$ limit. It is easy to see that the integrand converges pointwise to 
	$$ e^{-\frac{\theta^2t}{2}+\I \theta x}(-4\I\theta),$$
	so that 
	$$  \lim_{\eps\to 0} \mathbb E\left[  \rZ_{t}(x) \right] = \int_{-\infty}^{+\infty} e^{-\frac{\theta^2t}{2}+\I \theta x}(-4\I\theta)d\theta.$$
	This Gaussian type of  integral can be computed explicitly as $2\sqrt{\frac{2}{\pi }} \frac{x}{t^{3/2}} e^{-\frac{x^2}{2 t}}$, which, by the definition of the Dirichlet heat kernel from \eqref{eq:PtDir}, equals $2\partial_y P_t^{\rm Dir}(x,0)$. Furthermore, we have $\mathbb E\left[ ( \rZ_{t}(x))^2 \right]\leqslant (C/t)^2$ by \eqref{eq:Zp}, so that the sequence is uniformly integrable. Hence we may exchange the limit with the expectation, and obtain \eqref{eq:expectationinitialdata}.
\end{proof}
We can conclude from Lemma \ref{lem:expectationinitialdata} that if $\rZ_t(x)$ does converge as $\eps\to 0$ to some $\mathcal Z_t(x)$ solving the stochastic heat equation \eqref{eq:she} with some deterministic initial condition $\mc Z_{\mathrm{ini}}$, then for all $t>0$, 
\begin{equation}\label{eq:id-ini} \int_{\mathbb R_+}dy P_t^{\rm Dir}(x,y) \mc Z_{\mathrm{ini}}(y)  = dP_t^{\rm Dir}(x,0).\end{equation}
This suggests that $\mc Z_{\mathrm{ini}} =-2\delta'$. 
However, in order to rigorously identify the initial condition, we need to establish that \eqref{eq:Duhamelformdelta'} holds, and for that we will also need a second moment estimate. 
\begin{lemma}\label{lem:limit2bis}
	Fix $T>0$ and consider $\{\mathcal{Z}_t\}_{t\in{(0,T]}}$ distributed according to the measure $\mathbb Q$ defined in Lemma \ref{lem:limit2} (restricted to the time interval $(0,T]$).  There exists a constant $C=C(T)$ such that, for any $x \in\R_+$ and $0<t\leq T$, 
	\begin{align}
		\big\| \mathcal{Z}_t(x)\big\|_2^2 &  \leq C \left(dP_t^{\rm Dir}(x,0)\right)^2  \label{eq:firsttoprovecorrected} \\
		\big\|\mathcal{Z}_t(x)-dP_t^{\rm Dir}(x,0)\big\|_2^2 & \leq  C \sqrt{t} \left(dP_t^{\rm Dir}(x,0)\right)^2  , \label{eq:secondtoprovecorrected}
	\end{align} 
	where $\|F\|_2:=\sqrt{\int |F|^2 d\bQ}$.
\end{lemma}
\begin{remark}
	The bounds \eqref{eq:firsttoprovecorrected} and \eqref{eq:secondtoprovecorrected} are different from their analogues in  \cite[Lemma 6.6]{parekh2017kpz} in the case of the Neumann boundary condition with narrow wedge initial condition. The proof of \cite[Lemma 6.6]{parekh2017kpz} uses the martingale decomposition \eqref{eq:Zmart} and \mm{an iterating argument}. These ingredients would allow to mimic at the microscopic level the construction of the solution to the SHE (similar to our Section~\ref{sec:proofshe}). In our case, we did not manage to obtain the discrete analog of \eqref{eq:boundGt} which would have been needed to complete the proof, because this required much sharper heat kernel estimates. \mm{Instead, } we opted to use the exact computation of $\lim\mathbb E\left[\mathcal Z_t^{\eps}(x)^2\right]$ as $\eps\to 0$ which was done in the separate article \cite{barraquand2022markov}.
\end{remark}
\begin{proof}
	We use \cite[Prop. 4.8]{barraquand2022markov}, recorded below as \eqref{eq:momentKPZDirichlet}. For $x_1\leq x_2$, 
	\begin{equation}
		\mathbb E[\mathcal Z_t(x_1)\mathcal Z_t(x_2)] = 4^2\!\! \int_{\I \mathbb R} \frac{dw_1}{2\I\pi}\int_{\I \mathbb R+1+\eta} \!\!\frac{dw_2}{2\I\pi} \frac{w_1-w_2}{w_1-w_2+1}\frac{w_1+w_2}{w_1+w_2-1} \prod_{i=1}^2 w_ie^{\frac{t w_i^2}{2}-u_iw_i},
		\label{eq:nestedintegral}
	\end{equation}
	where $\eta>0$ is any positive real number, and the expectation is taken with respect to the measure $\bQ$. 
	More precisely, \cite[Prop. 4.8]{barraquand2022markov} in the case $n=2$ states that under the measure $\bQ^{\eps}$ from Lemma \ref{lem:limit2},  the limit as $\eps\to 0$ of $\mathbb E[\mathcal Z_t^{\eps}(x_1)\mathcal Z_t^{\eps}(x_2)]$ is given by the right-hand side of \eqref{eq:nestedintegral}. Moreover, Proposition~\ref{p:momentZ} with Cauchy-Schwarz inequality shows that 
	higher moments  are uniformly bounded as $\eps\to 0$. This  implies that the sequence of random variables $\mathcal Z_t^{\eps}(x_1)\mathcal Z_t^{\eps}(x_2)$ is uniformly integrable as $\eps\to 0$. Hence, the weak convergence to $\mathcal Z_t(x_1)\mathcal Z_t(x_2)$  implied by Proposition \ref{prop:tight2} holds as well in $L^1$, so that \eqref{eq:nestedintegral} holds.

	The evaluation of this double Gaussian integral is not trivial. To simplify its estimation, it is convenient to use the shorthand notation
	\begin{equation}
		D(t,x) = 4\int_{\I \mathbb R} \frac{dw}{2\I\pi} e^{\frac{t w^2}{2}-x w } w=4\int_{\I \mathbb R+1+\eta}\frac{dw}{2\I\pi} e^{\frac{t w^2}{2}-x w }w = dP_t^{\rm Dir}(x,0).
		\label{eq:shorthandnotation}
	\end{equation}
	The second equality follows from Cauchy's theorem\footnote{More precisely, we may apply the Cauchy theorem on the rectangle formed by the points $-\I R$, $\I R$, $\I R+1+\eta$, $-\I R+1+\eta$. Since the function $w\mapsto w e^{\frac{t w^2}{2}-xw }$ has no residues inside the rectangle, the integral over the rectangle equal zero. Moreover, due to the exponential decay of the integrand as the imaginary part increases, we see that the difference between the integration along the segment $[-\I R, \I R]$ or the segment $[-\I R+1+\eta, \I R+1+\eta]$ goes to zero as $R$ goes to infnity. This proves that the two integral formulas in \eqref{eq:shorthandnotation} are equal.} and the lack of residues between the vertical lines $\I\mathbb R$ and $\I\mathbb R+1+\eta$.
	
	We will use that for  $\mathfrak{R}[w_1- w_2 +1]<0$ and $\mathfrak{R}[w_1+ w_2 -1]>0$, 
	$$ \frac{1}{w_1- w_2+1} = -\int_{0}^{\infty}d\lambda e^{-\lambda(w_2-w_1-1)} \text{ and }\frac{1}{w_1+ w_2-1} = \int_{0}^{\infty}d\mu e^{-\mu(w_1+w_2-1)},$$
	Plugging this into \eqref{eq:nestedintegral} and using \eqref{eq:shorthandnotation}, we obtain 
	\begin{equation*}
		\mathbb E[\mathcal Z_t(x_1)\mathcal Z_t(x_2)] =\int_{0}^{\infty} d\lambda \int_{0}^{\infty} d\mu \; e^{\lambda+\mu }\partial_{\lambda}\partial_{\mu} D(t,x_1-\lambda+\mu)D(t,x_2+\lambda+\mu), 
	\end{equation*}
	so that from the exact explicit expression for $dP_t^{\rm Dir}(u,0)$ in  \eqref{eq:defdP}, we get 
	\begin{equation*}
		\frac{\Vert\mathcal Z_t(x)\Vert_2^2}{dP_t^{\rm Dir}(x,0)^2} =\frac{1}{4}\int_{0}^{\infty} d\lambda \int_{0}^{\infty} d\mu \; e^{\lambda+\mu }\partial_{\lambda}\partial_{\mu} \left( \frac{(x+\lambda)^2-\mu^2}{x^2}e^{-\frac{\lambda^2+\mu^2+2\lambda x}{t}}\right).
	\end{equation*}
	Computing the derivatives in $\lambda$ and $\mu$, the integral can be evaluated using Matematica, which yields the explicit formula 
	\begin{multline}
		\frac{\Vert\mathcal Z_t(x)\Vert_2^2}{dP_t^{\rm Dir}(x,0)^2}	=   
		1+ \frac{t}{2x}    \\  + \frac{e^{-x}\sqrt{\pi t} e^{t/4}}{4x^2}  \left(e^x
			\left(\mathrm{erf}\left(\frac{\sqrt{t}}{2}\right)+1\right) \left(t (x-1)+2 x^2\right)+t
			e^{\frac{x^2}{t}} \left(\mathrm{erf}\left(\frac{t-2 x}{2 \sqrt{t}}\right)+1\right)\right).
		\label{eq:fonctiontu}
	\end{multline} 
	We see that the  maximum of this function is attained as $x\to 0$, so that 
	\begin{equation}
		0\leq \frac{\Vert\mathcal Z_t(x)\Vert_2^2}{dP_t^{\rm Dir}(x,0)^2}\leq \frac{1}{8} \left(\sqrt{\pi } e^{t/4} \sqrt{t} (t+6)\left( \mathrm{erf}\bigg(\frac{\sqrt{t}}{2}\bigg)+1\right)+2
		(t+4)\right).
		\label{eq:complicatedbound}
	\end{equation}
	It can be checked that this expression behaves as  $1+\frac{3\sqrt{\pi}}{4}\sqrt{t}+o(\sqrt{t})$ as $t\to 0$, so that  on an interval $[0,T]$, \eqref{eq:complicatedbound} is bounded by $1+C \sqrt{t}$ where the constant $C$ depend on $T$. This immediately implies \eqref{eq:firsttoprovecorrected}, and, using Lemma \ref{lem:expectationinitialdata}, it also implies \eqref{eq:secondtoprovecorrected}. 
\end{proof}

\subsection{Conclusion} \label{sec:conclusion}
Provided with Proposition \ref{prop:tight2} and Lemma \ref{lem:limit2}, one can follow the argument of \cite[Lemma 6.7 and Theorem 6.8]{parekh2017kpz}, and one obtains that $\{\mathcal{Z}_s^\eps\}_{s\in{(0,T]}}$ converges as $\eps \to 0$ to a solution of the stochastic heat equation \eqref{eq:she} on the time interval $[0,T]$ with Dirichlet boundary condition, in the sense of weak convergence of probability measures on the path space $D({(0,T]},C(\R_+))$ endowed with the Skorokhod topology.

It remains to formally check that the initial condition is $\mathcal Z_{\rm ini}=-2\delta_0'$ as we have claimed. We need to check that the limit of $\{\mathcal{Z}_s^\eps\}_{s\in{(0,T]}}$, denoted $\{\mathcal{Z}_s\}_{s\in{(0,T]}}$, satisfies \eqref{eq:Duhamelformdelta'}. Since $\mathcal{Z}_s$ is a solution of \eqref{eq:SHE} on the space $D({(0,T]},C(\R_+))$, we already know that for any $0<s<t<T$,  
\begin{equation*}
	\mathcal Z_t(x) = P_{t-s}^{\rm Dir}(x,\cdot)\ast \mathcal Z_s(\cdot) + \int_s^t d\tau \int_{\mathbb R_+} dy   P_{t-\tau}^{\rm Dir}(x,y) \mathcal Z_{\tau}(y) \xi(\tau,y). 
\end{equation*}
Hence, following \cite[Section 6]{parekh2017kpz}, we may write 
\begin{multline} 
	\left\Vert \mathcal Z_t(x) -dP_t^{\rm Dir}(x,0) - \int_0^t d\tau \int_{\mathbb R_+} dy   P_{t-\tau}^{\rm Dir}(x,y) \mathcal Z_{\tau}(y) \xi(\tau,y) \right\Vert_2 \\ \leq  
	\left\Vert P_{t-s}^{\rm Dir}(x,\cdot)\ast \mathcal Z_s(\cdot)  - dP_t^{\rm Dir}(x,0) \right\Vert_2 +\left\Vert \int_0^s d\tau \int_{\mathbb R_+} dy   P_{t-\tau}^{\rm Dir}(x,y) \mathcal Z_{\tau}(y) \xi(\tau,y) \right\Vert_2. 
	\label{eq:twothingstobound}
\end{multline} 
This inequality holds for any $0<s<t$, so that in order to show that the LHS of \eqref{eq:twothingstobound} equal zero, it suffices to show that the RHS vanishes as $s\to 0$. Using the semi-group property,   
\begin{align*}
	\left\Vert P_{t-s}^{\rm Dir}(x,\cdot)\ast \mathcal Z_s(\cdot)  - dP_t^{\rm Dir}(x,0) \right\Vert_2 &\leq   \left\Vert   \int_{\mathbb R_+} dy   P_{t-s}^{\rm Dir}(x,y) \left( \mathcal Z_{s}(y) - dP^{\rm Dir}_s(y,0)\right)\right\Vert_2\\
	&\leq  \int_{\mathbb R_+} dy   P_{t-s}^{\rm Dir}(x,y)  \left\Vert  \mathcal Z_{s}(y) - dP^{\rm Dir}_s(y,0)\right\Vert_2 \\
	&\leq  \int_{\mathbb R_+} dy   P_{t-s}^{\rm Dir}(x,y)  s^{1/4 }dP^{\rm Dir}_s(y,0) \\
	&\leq s^{1/4 } dP^{\rm Dir}_t({x},0)
\end{align*}
where in the third inequality we have used Lemma \ref{lem:limit2bis}. Hence we have obtained that 
$$ \lim_{s\to 0} \left\Vert P_{t-s}^{\rm Dir}(x,\cdot)\ast \mathcal Z_s(\cdot)  - dP_t^{\rm Dir}(x,0) \right\Vert_2 =0. $$
Now we turn to the second term to bound. By It\^o isometry, 
\begin{align*}
	\left\Vert \int_0^s d\tau \int_{\mathbb R_+} dy   P_{t-\tau}^{\rm Dir}(x,y) \mathcal Z_{\tau}(y) \xi(\tau,y) \right\Vert_2^2 &= \int_0^s d\tau \int_{\mathbb R_+} dy P_{t-\tau}^{\rm Dir}(x,y)^2 \mathbb E\left[\mathcal Z_{\tau}(y)^2 \right]\\
	&\leq C \int_0^s d\tau \int_{\mathbb R_+} dy P_{t-\tau}^{\rm Dir}(x,y)^2 {dP_{\tau}^{\rm Dir}}(y,0)^2\\
	&= C \int_0^s d\tau {G_t(\tau,x)} \\
	&\leq C \sqrt{\frac{t}{t-s}}\int_0^s d\tau \frac{1}{\sqrt{\tau}} \,{dP_{\tau}^{\rm Dir}}(x,0)^2\\
	&\leq C\sqrt{\frac{st}{t-s}}\,{dP_{\tau}^{\rm Dir}}(x,0)^2,
\end{align*} 
where in the first inequality we have used Lemma \ref{lem:limit2bis}, and in the second inequality we have used 
Lemma \ref{lem:asymptoticsGt} (recall that the function ${G_t(\tau,x)}$ is defined in \eqref{eq:defGt}). We conclude that 
$$ \lim_{s\to 0}\left\Vert \int_0^s d\tau \int_{\mathbb R_+} dy   P_{t-\tau}^{\rm Dir}(x,y) \mathcal Z_{\tau}(y) \xi(\tau,y) \right\Vert_2 = 0,$$
so that \eqref{eq:Duhamelformdelta'} is satisfied for all $0<t\leq T$. This concludes the proof of Theorem \ref{thm:conv}.

\appendix

\section{Microscopic Cole-Hopf transform} \label{app:comp}

Here we explain our choice of parameters \eqref{eq:choice} which permits to obtain the discrete stochastic heat equation as in \eqref{eq:discDelta}.
We observe that $Z_t(x)$ is affected only by exchanging values $\sigma(k)$, $\sigma(k+1)$. Therefore, one can check that
\[ \mathcal{L}Z_t(k) =\Big[p\sigma_t(k)(1-\sigma_t(k+1))(e^{2\lambda}-1)+q\sigma_t(k+1)(1-\sigma_t(k))(e^{-2\lambda}-1)\Big]Z_t(k)\] for any $k>0$ and moreover
\[\mathcal LZ_t(0)=p(1-\sigma_t(1))(e^{2\lambda}-1)Z_t(0).\]
Besides, the discrete Laplacian acts as:
\[\Delta Z_t(k)=\big[e^{-\lambda(2\sigma_t(k+1)-1)}+e^{\lambda(2\sigma_t(k)-1)}-2\big]{Z_t}(k)\] for any $k>0$.
By identification we obtain the following conditions
\begin{eqnarray}
	\nu&=&D(e^{\lambda}+e^{-\lambda}-2)\\
	\nu&=&D(2e^\lambda-2)-p(e^{2\lambda}-1)\\
	\nu&=&D(2e^{-\lambda}-2)-q(e^{-2\lambda}-1).
\end{eqnarray}
which, after resolution, give
\begin{equation}\label{e:choixpar}
	\lambda=\frac{1}{2}\log\frac{q}{p},\quad D=\sqrt{pq},\quad \nu=q+p-2\sqrt{pq}.
\end{equation}
Finally, we need to define ${Z_t}(-1)$ such that, at the boundary point $k=0$, we get $(\nu {Z_t}(0)+\mathcal L{Z_t}(0))=D\Delta {Z_t}(0)$. This gives
\begin{eqnarray}
	\nu&=&D\left(\sqrt{\frac{p}{q}}-2+\frac{{Z_t}(-1)}{{Z_t}(0)}\right)\quad\text{if ${\sigma_t}(1)=1$}\\
	\nu+q-p&=&D\left(\sqrt{\frac{q}{p}}-2+\frac{{Z_t}(-1)}{{Z_t}(0)}\right)\quad\text{if ${\sigma_t}(1)=0$}.
\end{eqnarray}
With the choice of $D$ made in  \eqref{e:choixpar}, the last two conditions are in fact the same, and read
\begin{equation}
	{Z_t}(-1)=\mu\; {Z_t}(0) \qquad \text{with } \mu=\sqrt{\frac{q}{p}}.
\end{equation}
\section{Proof of Lemma \ref{l:oriane} \cite{bertini1997stochastic}}\label{a:bertini-giacomin}
	Let us fix $k\in\N$ and $t \in [0,\varepsilon^{-4}T]$. We denote by $(\mathcal{F}_s)_{s\in\R_+}$ the natural filtration associated to the process $(Z_s)_{s\in\R_+}$. Recall that $Z_t(k)=e^{-\varepsilon h_t(k)+\nu t}$. We first bound
	\begin{equation}\label{e:estimoriane1} \E \bigg[ \sup_{s\in[0,1]} \big|Z_{t+s}(k)-Z_t(k)\big|^{2n}  \bigg] \le \E \bigg[ \big|Z_t(k)\big|^{2n}  \sup_{s\in[0,1]} \big| e^{-\varepsilon(h_{t+s}(k)-h_t(k))+\nu s }-1 \big|^{2n}\bigg]. \end{equation} Let us denote by $\mathcal{Q}$ the number of jumps of the height at site $k$ between $t$ and $t+1$, which satisfies: for all $s\in[0,1]$, 
	\[  -2\mathcal{Q} \le h_{t+s}(k)-h_t(k) \le 2 \mathcal{Q}. \] Therefore, for all $s\in[0,1]$
	\[\big| e^{-\varepsilon(h_{t+s}(k)-h_t(k))+\nu s }-1 \big| \le e^{2\varepsilon \mathcal{Q} + \nu}-1. \] Inserting this into \eqref{e:estimoriane1} we get
	\begin{align}
		\E \bigg[ \sup_{s\in[0,1]} \big|Z_{t+s}(k)-Z_t(k)\big|^{2n}  \bigg] & = \E\bigg[ \E \bigg[ \sup_{s\in[0,1]} \big|Z_{t+s}(k)-Z_t(k)\big|^{2n}   \Big| \mathcal{F}_t\bigg]\bigg]  \notag\\
		& \le  \E\bigg[   \big|Z_t(k)\big|^{2n} \;  \E \Big[\big| e^{2\varepsilon \mathcal{Q} + \nu}-1\big|^{2n}   \Big| \mathcal{F}_t\Big]\bigg]. \label{eq:supbound}
	\end{align}
Note that $\mathcal{Q}$ is dominated by a Poisson random variable of parameter $2$ which is independent of $\mathcal{F}_t$. Therefore, \oo{with $M=\lfloor|\log(\eps)|\rfloor$,
\begin{align}
	 \E \Big[\big| e^{2\varepsilon \mathcal{Q} + \nu}-1\big|^{2n}   \Big| \mathcal{F}_t\Big] & \le \sum_{\ell=0}^{+\infty} \frac{2^\ell}{\ell!} \big( e^{2\varepsilon \ell + \nu}-1\big)^{2n}  \notag \\
 & \leqslant\sum_{\ell=0}^{M-1}\frac{2^\ell}{\ell!}(e^{\nu+2\eps \ell }-1)^{2n}+\sum_{\ell=M}^{\infty}\frac{2^\ell}{\ell!}(e^{\nu+2\eps \ell }-1)^{2n}\\
  &\leq \sum_{\ell=0}^{M-1}\frac{2^\ell}{\ell!}(10\eps \ell)^{2n}+e^{n\nu}\frac{(e^{6\eps n})^{M}}{M!}e^{e^{6\eps n}}\\
    &\leq(10\eps)^{2n}\sum_{\ell=0}^\infty \ell^{2n}\frac{2^\ell}{\ell!}+C(n)\frac{e^{6\eps nM}}{M!}\\
    &=O(\eps^{2n}),
\end{align}
by the choice of $M$.
Injecting this estimate inside \eqref{eq:supbound} yields the result.}

\section{Proof of the H\"{o}lder estimate in the $\eps\to 0$ limit for Delta prime initial condition}\label{app:Holderlimit}

\mm{In this appendix, we find interesting to show by an independent argument that the H\"older estimate \eqref{eq:Zxxprime} indeed holds in the $\eps\to 0$ limit, even if the following estimate is not sufficient to prove the tightness result.}

 Let us prove that for $t\in [\delta,T]$,  $\alpha\in (0,\frac12]$, and $n$ even, 
\begin{equation}
	\sup_{x,y\in \mathbb R_+}\left\lbrace \frac{\lim_{\eps\to 0} \mathbb E\left[(\mathcal Z^{\eps}_t(x)-\mathcal Z^{\eps}_t(y))^n\right]}{\vert x-y\vert^{\alpha n}} \right\rbrace <C\label{eq:boundholderlimit}
\end{equation}
where $C$ is a constant depending on $n, \delta, T, \alpha$. 
Given the bound \eqref{eq:Zp} (or more precisely, its $\eps\to 0$ limit), it suffices to consider the case $\alpha=1/2$, or any $\alpha\geq 1/2$, and we may also restrict the supremum to $x,y$ being close to each other.

From \cite[Prop. 4.8]{barraquand2022markov}, we have, for any $0\leq x_1\leq \dots \leq x_n$, 
\begin{multline}
 \lim_{\eps\to 0} \mathbb E\left[\prod_{i=1}^n \mathcal Z^{\eps}_t(x_i)\right]\\= 4^n\int_{r_1+\I\R} \frac{dw_1}{2\I\pi} \dots \int_{r_n+\I\R} \frac{dw_n}{2\I\pi} \prod_{i<j}\frac{w_i-w_j}{w_i-w_j+1}\frac{w_i+w_j}{w_i+w_j-1} \prod_{i=1}^n w_i e^{\frac{t w_i^2}{2}-x_i w_i},
	\label{eq:momentKPZDirichlet}
\end{multline}  
where $0=r_1<r_2-1<\dots <r_n-n+1$.  
Hence, we may write 
\begin{multline}
\lim_{\eps\to 0} 	\mathbb E\left[(\mathcal Z^{\eps}_t(x)-\mathcal Z^{\eps}_t(y))^n\right] = \\
	4^n\int_{r_1+\I\R} \frac{dw_1}{2\I\pi} \dots \int_{r_n+\I\R} \frac{dw_w}{2\I\pi} \prod_{i<j}\frac{w_i-w_j}{w_i-w_j+1}\frac{w_i+w_j}{w_i+w_j-1} \prod_{i=1}^n w_i e^{\frac{t w_i^2}{2}-xw_i}F_{y-x}(w_1, \dots, w_n)
	\label{eq:diffpowern}
\end{multline}
where the function $F_{\eta}(\vec w)$ is defined by 
$$ F_{\eta}(\vec w) = 1-\binom{n}{1}e^{-\eta w_n}+\binom{n}{2} e^{-\eta(w_n+w_{n-1})}-\dots +(-1)^n \binom{n}{n}e^{-\eta (w_1+\dots+w_n)}.$$
We need to consider the  fine behaviour of the function $F_{\eta}(\vec w)$ as $\eta$ goes to zero. Of course, if we could show that  $F_{\eta}(\vec w)$ is bounded as $\eta\to 0$ by $\vert \eta\vert^{n/2}$ times a bounded function of $\vec w$, it would not be too difficult to prove the bound \eqref{eq:boundholderlimit}. But the function $F_{\eta}(\vec w)$ does not satisfy such an inequality. It is an analytic function of the variable $\eta$, whose Taylor coefficients at $\eta=0$ can be readily computed for small values of $n$ and they are all non-zero except for the first one. However, we will show that the contribution of the first $\lceil n/2\rceil$ of those Taylor coefficients has zero contribution to the integral \eqref{eq:diffpowern}. Let us consider the (unique) decomposition 
\begin{equation}
	F_{\eta}(\vec w) = \sum_{i=0}^{\lceil n/2\rceil-1} \eta^i F(\vec w;i) \;\; + \; \eta^{\lceil n/2\rceil} H_{\eta}(\vec w), 
	\label{eq:Taylordecomposition}
\end{equation}
where the Taylor coefficients $F(\vec w;i)$ are polynomials in the variables $w_i$ and the function $H_{\eta}(\vec w)$ is analytic in $\eta$ with no singularity at zero. 
It is convenient to introduce the functional 
\begin{multline} I_n[F,x]:= \\4^n\int_{r_1+\I\R} \frac{dw_1}{2\I\pi} \dots \int_{r_n+\I\R} \frac{dw_w}{2\I\pi} \prod_{i<j}\frac{w_i-w_j}{w_i-w_j+1}\frac{w_i+w_j}{w_i+w_j-1} \prod_{i=1}^n w_i e^{\frac{t w_i^2}{2}-xw_i}F(w_1, \dots, w_n).
\label{eq:defIn}
\end{multline}
so that \eqref{eq:diffpowern}  is nothing but $I_n[F_{y-x},x]$. Using \eqref{eq:Taylordecomposition}, we have the decomposition
$$I_n[F_{\eta},x] = \sum_{i=0}^{\lceil n/2\rceil-1} \eta^i I_n[F(\vec w;i),x] \;\; + \; \eta^{\lceil n/2\rceil} I_n[H_{\eta}, x].$$
From the definition of $F_{\eta}(\vec w)$ we know that $H_{\eta}(\vec w)$ is holomorphic in each of the variables $w_i$ and grows at most as a polynomial of degree $\lceil n/2\rceil$ as the variables $w_i$ vary over the vertical contours $r_i+\I\R$ (and $H_{\eta}(\vec w)$ can be estimated using Taylor's theorem with integral remainder, which is valid for complex-valued functions). Hence, due to the Gaussian decay of the term $e^{tw_i^2/2}$ and the fact that $t\in [\delta,T]$, we find that the modulus of $I_n[H_{\eta},x]$ is bounded by a constant (uniformly for $x\in \mathbb R_+$ and $\eta$ in a compact set). Then, assuming that for all $0\leq i\leq \lceil n/2\rceil -1$, $I_n[F(\vec w;i),x]=0$, we conclude that 
$$ \vert I_n[F_{\eta},x] \vert \leq C   \vert \eta\vert^{\lceil n/2\rceil} $$
so that \eqref{eq:boundholderlimit} holds.

Now we will use an algebraic argument to explain why, indeed, for all $0\leq i\leq \lceil n/2\rceil -1$, $I_n[F(\vec w;i),x]=0$. The Taylor coefficients are given explicitly by 
\begin{equation}
	F(\vec w;i) = \frac{(-1)^i}{i!}\sum_{k=0}^n \binom{n}{k}(-1)^k \bigg( \sum_{j=n-k+1}^n w_j \bigg)^i.
\end{equation}
\begin{lemma}
Fix $n\in \mathbb N$. For all $0\leq i\leq \lceil n/2\rceil -1$, we may write the polynomial $	F(\vec w;i)$ as 
\begin{equation}
		F(\vec w;i) = \sum_{j=1}^{n-1} (w_j-w_{j+1}+1)f_j(\vec w; i)
		\label{eq:decompositionarithmetic}
\end{equation}
where the polynomials $f_j(\vec w; i)$ are symmetric with respect to exchanging $w_j$ and $w_{j+1}$. 
\label{lem:decomposition}
\end{lemma} 
The statement is trivial when $i=0$ since in that case, $F(\vec w;0)=0$. Hence, for $n=1$ and $n=2$, the lemma is trivial. The first non trivial case arises for $n=3$ and $i=1$. In that case, we find that 
$$ F(\vec w;1) = w_1-2w_2+w_3 $$ which can indeed be decomposed as 
$$ F(\vec w;1) = (w_1-w_2+1) -  (w_2-w_3+1),$$
so that the polynomials $f_1$ and $f_2$ can be chosen as  $f_1(\vec w;1)=1$ and $f_2(\vec w;1)=-1$. 
Before proving Lemma \ref{lem:decomposition} in general, let us see how it can be used. 

Fix $n\in \mathbb N$ and  $0\leq i\leq \lceil n/2\rceil -1$. Using the decomposition \eqref{eq:decompositionarithmetic}, in order to prove that $I_n[F(\vec w;i),x]=0$, it suffices to prove that  for all $1\leq j\leq n-1$, 
$$I_n[(w_j-w_{j+1}+1)f_j(\vec w; i),x]=0.$$
The factor $(w_j-w_{j+1}+1)$ cancels the same factor present in the denominator in the integral \eqref{eq:defIn}. Hence, the integrand has no singularity at $w_{j+1}=w_j+1$ anymore. This implies that we may freely deform the contour of the variable $w_{j+1}$, originally chosen as $r_{j+1}+\I \R$ with $r_{j+1}>r_j+1$, to become the same as the contour for $w_j$, that is $r_j+\I\R$. Once the variables $w_j$ and $w_{j+1}$ are integrated along the same contour, we see by a symmetry argument that the integral vanishes: the integrand may be written as $(w_j-w_{j+1})$ multiplied by a function which is symmetric with respect to exchanging $w_j$ and $w_{j+1}$ (hence the importance of the assumption on $f_j$ in Lemma \ref{lem:decomposition}). The integral of such function must be zero when the two integration domains are the same. Thus, we have shown that all $I_n[(w_j-w_{j+1}+1)f_j(\vec w; i),x]=0$, which implies that $I_n[F(\vec w;i),x]=0$. It only remains to prove Lemma \ref{lem:decomposition}. 

\begin{proof}[Proof of Lemma \ref{lem:decomposition}] 
	It is convenient to reorder the variables and prove  the following, slightly stronger, statement (below $\alpha$ is any real number):

For all $n\geq 2i+1$, the polynomial	
	\begin{equation} P_i(\vec w) := \sum_{k=0}^n \binom{n}{k}(-1)^k \bigg( \sum_{j=1}^k w_j \bigg)^i.
		\label{eq:defPi}
		\end{equation}
can be decomposed as 
\begin{equation}
	 P_i(\vec w) = \sum_{j=1}^{n-1} (w_{j+1}-w_{j}+\alpha)p_j(\vec w; i)
	\label{eq:decompositionP}
\end{equation}
where the polynomials $p_j(\vec w; i)$ are symmetric with respect to exchanging $w_j$ and $w_{j+1}$.

We will prove this statement by induction on $i$. 
Assume first that $i=0$ and consider any $n\geq 1$. 
Then, as we have seen, 
$$ P_0(\vec w) = \sum_{j=0}^n \binom{n}{j}(-1)^j=0 $$
so that $P_0(\vec w)$  trivially satisfies a decomposition of the form \eqref{eq:decompositionP}.


%

Let us now assume that $i>1$ and consider any $n\geq 2i+1$. It is convenient to use the notation 
$S_k:= \sum_{j=1}^k w_j$. We will use the integration by parts formula 
\begin{equation}
	\sum_{k=0}^n a_kb_k = -\sum_{k=1}^{n}\bigg( (b_k-b_{k-1})\sum_{j=0}^{k-1}a_j \bigg)+ b_n\sum_{j=0}^n a_j.
	\label{eq:discreteIPP}
\end{equation}
We obtain 
\begin{align*}
P_i(\vec w) = \sum_{k=1}^n \binom{n}{k}(-1)^{k} S_k^{i-1} S_k 
 &=- \sum_{k=1}^{n} w_k\left( \sum_{j=0}^{k-1}\binom{n}{j}(-1)^j S_j^{i-1}  \right) + S_n P_{i-1}(\vec w) \\
 &= - \sum_{k=0}^{n-1} w_{k+1}\left( \sum_{j=0}^{k}\binom{n}{j}(-1)^j S_j^{i-1}  \right) + S_n P_{i-1}(\vec w)\\
 &= \sum_{k=1}^{n-1}(w_{k+1}-w_k)f_k+ w_n f_n + S_n P_{i-1}(\vec w)
\end{align*}
where in the first line, we have used the integration by parts \eqref{eq:discreteIPP} (with $a_k=\binom{n}{k}(-1)^{k} S_k^{i-1}$ and $b_k=S_k$), in the second line we have used a simple change of variables, and in the last line we have used again the identity \eqref{eq:discreteIPP}, letting
\begin{equation}
	f_k:= \sum_{\ell=0}^{k-1}\sum_{j=0}^{\ell} \binom{n}{j}(-1)^j S_j^{i-1}= \sum_{j=0}^{k-1} (k-j)\binom{n}{j}(-1)^j S_j^{i-1}.
	\label{eq:deffk}
\end{equation} 
The sequence $f_k$ satisfies two convenient properties: first, 
\begin{align*}
	f_n=\sum_{j=0}^{n-1} (n-j)\binom{n}{j}(-1)^j S_j^{i-1} =n \sum_{j=0}^{n-1} \binom{n-1}{j}(-1)^j S_j^{i-1} =nP_{i-1}(w_1, \dots, w_{n-1}).
\end{align*}
Second, 
\begin{align*}
	 \sum_{k=1}^{n-1}f_k = \sum_{k=1}^{n-1} \sum_{j=0}^{k-1} (k-j)\binom{n}{j}(-1)^j S_j^{i-1} 
	 &=  \sum_{j=0}^{n-2} \frac{(n-j)(n-j-1)}{2} \binom{n}{j}(-1)^j S_j^{i-1}\\
	 &= \frac{n(n-1)}{2} \sum_{j=0}^{n-2} \binom{n-2}{j}(-1)^j S_j^{i-1}\\
	 &= \frac{n(n-1)}{2} P_{i-1}(w_1, \dots, w_{n-2}),
\end{align*}
where in the first line we have just exchanged the order of summation and  in the second line we have simplified the binomial coefficient.

To summarize, we arrive at the identity 
\begin{multline}
	P_i(w_1, \dots, w_n) = S_nP_{i-1}(w_1, \dots, w_n) + w_n n  P_{i-1}(w_1, \dots, w_{n-1}) \\ - \alpha \frac{n(n-1)}{2} P_{i-1}(w_1, \dots, w_{n-2}) + \sum_{k=1}^{n-1}(w_{k+1}-w_k+\alpha)f_k .
	\label{eq:formulaPi}
\end{multline} 
Since the $S_j$'s only depend on the variables $w_1, \dots, w_j$, it is clear from \eqref{eq:deffk} that $f_k$ only depends  on the variables $w_1, \dots, w_{k-1}$, hence it is symmetric with respect to exchanging $w_k$ and $w_{k+1}$. This implies that the sum $\sum_{k=1}^{n-1}(w_{k+1}-w_k+\alpha)f_k$ is a decomposition of the desired form \eqref{eq:decompositionP}. 

For $n\geq 2i+1$, we have $n-2\geq 2(i-1)+1$ so that using the induction hypothesis, we know that  $P_{i-1}(w_1, \dots, w_{n-2})$ 
admits a decomposition of the form \eqref{eq:decompositionP}. 
The induction hypothesis also implies that $S_nP_{i-1}(w_1, \dots, w_n)$ and $ w_n P_{i-1}(w_1, \dots, w_{n-1})$ satisfy decompositions of the form \eqref{eq:decompositionP}, since $S_n$ is symmetric in all variables and $w_n$ is symmetric with respect to exchanging $w_j$ and $w_{j+1}$ for all $j<n-1$. 

Thus, we have shown that $P_i(w_1, \dots, w_n)$ is the sum of polynomials satisfying a decomposition of the form \eqref{eq:decompositionP}, so that $P_i(w_1, \dots, w_n)$ admits such a decomposition as well, which concludes the proof of \eqref{eq:boundholderlimit}. 
\end{proof}





\begin{thebibliography}{10}
	
	\bibitem{alberts2014intermediate}
	T.~Alberts, K.~Khanin, and J.~Quastel, \emph{The intermediate disorder regime
		for directed polymers in dimension $1+ 1$}, Ann. Probab. \textbf{42} (2014),
	no.~3, 1212--1256.
	
	\bibitem{amir2011probability}
	G.~Amir, I.~Corwin, and J.~Quastel, \emph{Probability distribution of the free
		energy of the continuum directed random polymer in 1 + 1 dimensions}, Comm.
	Pure Appl. Math. \textbf{64} (2011), no.~4, 466--537.
	
	\bibitem{ayyer2020stationary}
	A.~Ayyer, S.~Goldstein, J.~L. Lebowitz, and E.~R. Speer, \emph{Stationary
		states of the one-dimensional facilitated asymmetric exclusion process}, J.
	Math. Phys. \textbf{63} (2022).
	
	\bibitem{baik2018FTASEP}
	J.~Baik, G.~Barraquand, I.~Corwin, and T.~Suidan, \emph{Facilitated exclusion
		process}, Computation and Combinatorics in Dynamics, Stochastics and Control
	(Cham) (Elena Celledoni, Giulia Di~Nunno, Kurusch Ebrahimi-Fard, and
	Hans~Zanna Munthe-Kaas, eds.), Springer International Publishing, 2018,
	pp.~1--35.
	
	\bibitem{baik2018pfaffian}
	\bysame, \emph{{P}faffian {S}chur processes and last passage percolation in a
		half-quadrant}, Ann. Probab. \textbf{46} (2018), no.~6, 3015--3089.
	
	\bibitem{baik2001algebraic}
	J.~Baik and E.~M. Rains, \emph{Algebraic aspects of increasing subsequences},
	Duke Math. J. \textbf{109} (2001), no.~1, 1--65.
	
	\bibitem{barraquand2018half}
	G.~Barraquand, A.~Borodin, and I.~Corwin, \emph{Half-space {M}acdonald
		processes}, Forum Math Pi (2020).
	
	\bibitem{barraquand2018stochastic}
	G.~Barraquand, A.~Borodin, I.~Corwin, and M.~Wheeler, \emph{Stochastic
		six-vertex model in a half-quadrant and half-line open asymmetric simple
		exclusion process}, Duke Math. J. \textbf{167} (2018), no.~13, 2457--2529.
	
	\bibitem{barraquand2022markov}
	G.~Barraquand and I.~Corwin, \emph{Markov duality and {B}ethe ansatz formula
		for half-line open {ASEP}}, Probab. Math. Phys. \textbf{5} (2024), no.~1,
        89--129.
	
	\bibitem{barraquand2022stationary}
	\bysame, \emph{Stationary measures for the log-gamma polymer and {KPZ} equation
		in half-space}, Ann. Probab. \textbf{51} (2023), no.~5, 1830--1869.
	
	\bibitem{basu2009FTASEP}
	U.~Basu and P.~K. Mohanty, \emph{Active--absorbing-state phase transition
		beyond directed percolation: A class of exactly solvable models}, Phys. Rev.
	E \textbf{79} (2009), 041143.
	
	\bibitem{bertini1997stochastic}
	L.~Bertini and G.~Giacomin, \emph{Stochastic {B}urgers and {KPZ} equations from
		particle systems}, Comm. Math. Phys. \textbf{183} (1997), no.~3, 571--607.
	
	\bibitem{billingsley1968convergence}
	P.~Billingsley, \emph{Convergence of probability measures}, John Wiley \& Sons,
	1968.
	
	\bibitem{blondel2020FEP}
	O.~Blondel, C.~Erignoux, M.~Sasada, and M.~Simon, \emph{{Hydrodynamic limit for
			a facilitated exclusion process}}, Ann. Inst. Henri Poincar\'e, Probab. Stat.
	\textbf{56} (2020), no.~1, 667 -- 714.
	
	\bibitem{blondel2021FEP}
	O.~Blondel, C.~Erignoux, and M.~Simon, \emph{Stefan problem for a nonergodic
		facilitated exclusion process}, Probab. Math. Phys. \textbf{2} (2021), no.~1,
	127--178.
	
	\bibitem{blondel2016}
	O.~Blondel, P.~Gon\c{c}alves, and M.~Simon, \emph{Convergence to the stochastic
		{B}urgers equation from a degenerate microscopic dynamics}, Electron. J.
	Probab. \textbf{21} (2016), 25 pp.
	
	\bibitem{borodin2016directed}
	A.~Borodin, A.~Bufetov, and I.~Corwin, \emph{Directed random polymers via
		nested contour integrals}, Ann. Phys. \textbf{368} (2016), 191--247.
	
	\bibitem{chen2018limiting}
	D.~Chen and L.~Zhao, \emph{The limiting behavior of the {FTASEP} with product
		{B}ernoulli initial distribution}, arXiv:1808.10612.
	
	\bibitem{corwin2020stochastic}
	I.~Corwin, P.~Ghosal, H.~Shen, and L.~Tsai, \emph{Stochastic {PDE} limit of the
		six vertex model}, Commu. Math. Phys. \textbf{375} (2020), no.~3, 1945--2038.
	
	\bibitem{corwin2016open}
	I.~Corwin and H.~Shen, \emph{Open {ASEP} in the weakly asymmetric regime},
	Comm. Pure Appl. Math. \textbf{71} (2018), no.~10, 2065--2128.
	
	\bibitem{DAES} 
	H.~Da Cunha, C.~Erignoux and M.~Simon, \emph{Hydrodynamic limit for an open facilitated exclusion process with slow and fast boundaries}. arXiv:2401.16535
	
	\bibitem{dembo2016weakly}
	A.~Dembo and L.~Tsai, \emph{Weakly asymmetric non-simple exclusion process and
		the {Kardar--Parisi--Zhang} equation}, Comm. Math. Phys. \textbf{341} (2016),
	no.~1, 219--261.
	
	\bibitem{ESZ}
	C.~Erignoux, M.~Simon and L.~Zhao, \emph{Mapping hydrodynamics
		for the facilitated exclusion and zero-range processes}, Ann. Appl. Probab. \textbf{34} (2024), no.~1B, 1524--1570.

    \bibitem{EZ23}
    C.~Erignoux and L.~Zhao, \emph{Stationary fluctuations for the facilitated exclusion process}, arXiv:2305.13853.
	
	\bibitem{gartner1987transform}
	J.~G~\!\!\"artner, \emph{Convergence towards {B}urger's equation and
		propagation of chaos for weakly asymmetric exclusion processes}, Stoch. Proc.
	Appl. \textbf{27} (1987), 233--260.
	
	\bibitem{gabel2010FEP}
	A.~Gabel, P.~L. Krapivsky, and S.~Redner, \emph{Facilitated asymmetric
		exclusion}, Phys. Rev. Lett. \textbf{105} (2010), 210603.
	
	\bibitem{gerencser2017singular}
	M.~Gerencs{\'e}r and M.~Hairer, \emph{Singular {SPDE}s in domains with
		boundaries}, Probab. Theor. Rel. Fields \textbf{173} (2019), no.~3-4,
	697--758.
	
	\bibitem{goldstein2019FTASEP}
	S.~Goldstein, J.~L. Lebowitz, and E.~R. Speer, \emph{Exact solution of the
		facilitated totally asymmetric simple exclusion process}, J. Stat. Mech.:
	Theor. Exp. \textbf{2019} (2019), no.~12, 123202.
	
	\bibitem{goldstein2021FTASEP}
	\bysame, \emph{The discrete-time facilitated totally asymmetric simple
		exclusion process}, Pure Appl. Funct. Anal. \textbf{6} (2021), 177--203.
	
	\bibitem{gonccalves2017nonequilibrium}
	P.~Gon{\c{c}}alves, C.~Landim, and A.~Milan{\'e}s, \emph{Nonequilibrium
		fluctuations of one-dimensional boundary driven weakly asymmetric exclusion
		processes}, Ann. Appl. Probab. \textbf{27} (2017), no.~1, 140--177.
	
	\bibitem{goncalves2014BG}
	P.~Gon\c{c}alves and M.~Jara, \emph{{Nonlinear fluctuations of weakly
			asymmetric interacting particle systems}}, {Arch. Ration. Mech. Anal.}
	\textbf{212} (2014), no.~2, 597--644 (English).
	
	\bibitem{GJSi2017}
	P.~Gon\c{c}alves, M.~Jara, and M.~Simon, \emph{{Second order
			{B}oltzmann-{G}ibbs principle for polynomial functions and applications}}, J.
	Stat. Phys. \textbf{166} (2017), no.~1, 90--113.
	
	\bibitem{GPS2020}
	P.~Gon\c{c}alves, N.~Perkowski, and M.~Simon, \emph{Derivation of the
		stochastic {Burgers} equation with {Dirichlet} boundary conditions from the
		{WASEP}}, Ann. Henri Lebesgue \textbf{3} (2020), 87--167 (en).
	
	\bibitem{gubinelli2015paracontrolled}
	M.~Gubinelli, P.~Imkeller, and N.~Perkowski, \emph{Paracontrolled distributions
		and singular {PDE}s}, Forum Math. Pi \textbf{3} (2015).
	
	\bibitem{gubinelli2017kpz}
	M.~Gubinelli and N.~Perkowski, \emph{{KPZ} reloaded}, Comm. Math. Phys.
	\textbf{349} (2017), no.~1, 165--269.
	
	\bibitem{gubinelli2018uniqueness}
	\bysame, \emph{{Energy solutions of KPZ are unique}}, {J. Am. Math. Soc.}
	\textbf{31} (2018), no.~2, 427--471 (English).
	
	\bibitem{gueudre2012directed}
	T.~Gueudr{\'e} and P.~Le~Doussal, \emph{Directed polymer near a hard wall and
		{KPZ} equation in the half-space}, Europhysics Lett. \textbf{100} (2012),
	no.~2, 26006.
	
	\bibitem{hairer2013kpz}
	M.~Hairer, \emph{{Solving the KPZ equation}}, {Ann. Math. (2)} \textbf{178}
	(2013), no.~2, 559--664 (English).
	
	\bibitem{hairer2014regularity}
	\bysame, \emph{{A theory of regularity structures}}, {Invent. Math.}
	\textbf{198} (2014), no.~2, 269--504 (English).
	
	\bibitem{he2024boundary}
	J.~He, \emph{Boundary current fluctuations for the half-space {ASEP} and
		six-vertex model}, Proc. London Math. Soc. \textbf{128} (2024), no.~2,
	e12585.
	
	
	\bibitem{imamura2022solvable}
	T.~Imamura, M.~Mucciconi, and T.~Sasamoto, \emph{Solvable models in the {KPZ}
		class: approach through periodic and free boundary {S}chur measures}, arXiv:2204.08420.
	
	\bibitem{KL}{C. Kipnis} and {C. Landim}, \emph{Scaling limits of interacting particle systems.} Berlin: Springer (1999).
	
	\bibitem{krajenbrink2020replica}
	A.~Krajenbrink and P.~Le~Doussal, \emph{Replica {B}ethe ansatz solution to the
		{Kardar-Parisi-Zhang} equation on the half-line}, SciPost Phys \textbf{8}
	(2020), 035.
	
	\bibitem{mueller1991support}
	C.~Mueller, \emph{On the support of solutions to the heat equation with noise},
	Stochastics \textbf{37} (1991), no.~4, 225--245.
	
	\bibitem{parekh2017kpz}
	S.~Parekh, \emph{The {KPZ} limit of {ASEP} with boundary}, Comm. Math. Phys.
	\textbf{365} (2019), no.~2, 569--649.
	
	\bibitem{parekh2019positive}
	S.~Parekh, \emph{{Positive random walks and an identity for half-space
			SPDEs}}, Electronic Journal of Probability \textbf{27} (2022), 1 -- 47.
	
	\bibitem{prahofer2001current}
	M.~Pr{\"a}hofer and H.~Spohn, \emph{Current fluctuations for the totally
		asymmetric simple exclusion process}, In and Out of Equilibrium (Mambucaba
	2000) Progress in Probability \textbf{51} (2001), 185--204.
	
	\bibitem{pastor2000FEP}
	M.~Rossi, R.~Pastor-Satorras, and A.~Vespignani, \emph{Universality class of
		absorbing phase transitions with a conserved field}, Phys. Rev. Lett.
	\textbf{85} (2000), 1803--1806.
	
	\bibitem{schuetz1997duality}
	G.~M. Sch{\"u}tz, \emph{Duality relations for asymmetric exclusion processes},
	J. Stat. Phys. \textbf{86} (1997), no.~5, 1265--1287.
	
	\bibitem{tracy2013asymmetric}
	C.~A. Tracy and H.~Widom, \emph{The asymmetric simple exclusion process with an
		open boundary}, J. Math. Phys. \textbf{54} (2013), no.~10, 103301.
	
	\bibitem{tracy2013bose}
	\bysame, \emph{The {B}ose gas and asymmetric simple exclusion process on the
		half-line}, J. Stat. Phys. \textbf{150} (2013), no.~1, 1--12.
	
	\bibitem{walsh1986introduction}
	J.~B Walsh, \emph{An introduction to stochastic partial differential
		equations}, {\'E}cole d'{\'E}t{\'e} de Probabilit{\'e}s de Saint Flour
	XIV-1984, Springer, 1986, pp.~265--439.
	
	\bibitem{wu2018intermediate}
	X.~Wu, \emph{Intermediate disorder regime for half-space directed polymers}, J.
	Stat. Phys. \textbf{181} (2020), 2372--2403.
	
	\bibitem{yang2021stochastic}
	K.~Yang, \emph{Stochastic {B}urgers {E}quation via energy solutions from
		non-stationary particle systems}, arXiv:1810.02836.
	
	\bibitem{yang2020kardar}
	\bysame, \emph{{Kardar-Parisi-Zhang} equation from non-simple variations on
		open-{ASEP}}, Probab. Theor. Rel. Fields \textbf{183} (2022).
	
\end{thebibliography}


\providecommand{\bysame}{\leavevmode\hbox to3em{\hrulefill}\thinspace}
\providecommand{\MR}{\relax\ifhmode\unskip\space\fi MR }
\providecommand{\MRhref}[2]{%
	\href{http://www.ams.org/mathscinet-getitem?mr=#1}{#2}
}
\providecommand{\href}[2]{#2}

\subsection*{Acknowledgments}

This project is partially supported by the ANR grant MICMOV (ANR-19-CE40-0012)
of the French National Research Agency (ANR), and by the European Union with the program FEDER ``Fonds européen de développement régional'' with the Région Hauts-de-France. It has also received funding from the European
Research Council (ERC) under the European Union’s Horizon 2020 research and innovative program
(grant agreement n° 715734), and from
Labex CEMPI (ANR-11-LABX-0007-01). This article is also partially  based upon work supported by the National Science Foundation under Grant No. DMS-1928930 while G.B. participated in a
program hosted by the Mathematical Sciences Research Institute in Berkeley, California,
during the Fall 2021 semester.

%
%

\end{document}